\documentclass{amsart}
\usepackage{epsfig}
\usepackage{amsmath}
 \usepackage{hyperref}
\usepackage{amsfonts}
\usepackage{amsmath,amscd}
\usepackage{fullpage}
\usepackage{amssymb}
\usepackage{amsthm}
\usepackage{mathrsfs}
\usepackage{graphicx}
\usepackage{xypic}
\newtheorem{definition}{Definition}[section]
\newtheorem{theorem}[definition]{Theorem}
\newtheorem{remark}[definition]{Remark}
\newtheorem{final Remarks}[definition]{Final Remarks}
\newtheorem{lemma}[definition]{Lemma}
\newtheorem{question}[definition]{Question}
\newtheorem{problem}[definition]{Problem}
\newtheorem{conjecture}[definition]{Conjecture}

\newtheorem{example}[definition]{Example}
\newtheorem{corollary}[definition]{Corollary}

\numberwithin{equation}{section}
\begin{document}
\title{Topological Rigidity Problems}
\vspace{2cm}
\author{Ramesh Kasilingam}

\email{rameshkasilingam.iitb@gmail.com  ; mathsramesh1984@gmail.com  }
\address{Statistics and Mathematics Unit,
Indian Statistical Institute,
Bangalore Centre, Bangalore - 560059, Karnataka, India.}

\date{}
\maketitle
\begin{abstract}
We survey the recent results and current issues on the topological rigidity problem for closed aspherical manifolds, i.e., connected closed manifolds whose universal coverings are contractible. A number of open problems and conjectures are presented during the course of the discussion. We also review the status and applications of the Farrell-Jones Conjecture for algebraic $K$-and $L$-theory for a group ring $RG$ and coefficients in an additive category. These conjectures imply many other well-known and important conjectures. Examples are the Borel Conjecture about the topological rigidity of closed aspherical manifolds, the Novikov Conjecture about the homotopy invariance of higher signatures and the Conjecture for vanishing of the Whitehead group. We here present the status of the Borel, Novikov and vanishing of the Whitehead group Conjectures.
\end{abstract}
\section{\large Introduction}
A classification of manifolds up to $\rm{CAT}$ ($\rm{Diff}$, $\rm{PL}$ or $\rm{Top}$) isomorphism requires the construction of a complete set of invariants such that 
\begin{itemize}
\item[(i)] the invariants of a manifold are computable,
\item[(ii)] two manifolds are  $\rm{CAT}$-isomorphic if and only if they have the same invariants, and 
\item[(iii)] there is given a list of non-$\rm{CAT}$ isomorphisms manifolds realizing every possible set of invariants.
\end{itemize}
The most important invariant of a manifold $M^n$ is its dimension, the number $n\geq0$ such that $M$ is locally homeomorphic to $\mathbb{R}^n$, so that an $n$-dimensional manifold $M^n$ cannot be homeomorphic to an $m$-dimensional manifold $N^m$.  The homology and cohomology of an orientable $n$-dimensional closed manifold $M$ are related by the Poincar$\grave{\rm e}$ duality isomorphisms $H^*(M)\cong H_{n-*}(M)$. Any $n$-dimensional closed manifold $M$ has $\mathbb{Z}_2$-coefficient Poincar$\grave{\rm e}$ duality $H^*(M;\mathbb{Z}_2)\cong H_{n-*}(M;\mathbb{Z}_2)$ with $H_{n}(M;\mathbb{Z}_2)=\mathbb{Z}_2$, $H_{m}(M;\mathbb{Z}_2)=0$ for $m>n$. The dimension of a closed manifold $M$ is thus characterized homologically as the largest integer $n\geq0$ with $H_{n}(M;\mathbb{Z}_2)\neq0$. Homology is homotopy invariant, so that the dimension is also homotopy invariant : if $n\neq m$ an $n$-dimensional closed manifold $M^n$ cannot be homotopy equivalent to an $m$-dimensional closed manifold $N^m$.\\

The flavor of the classification of closed manifolds depends heavily on the dimension. In one dimension, the classification is trivial- all we have is the circle. In dimension two, it is easy. Manifolds are determined by their orientability and their Euler characteristic, and the smooth, piecewise linear, topological, and homotopy categories coincide. For $n\geq3$ there exist $n$-dimensional manifolds which are homotopy equivalent but not diffeomorphic, so that the diffeomorphism and homotopy classifications must necessarily differ.\\

In the 3-dimensional setting there is no distinction between smooth, PL and topological manifolds i.e., the categories of smooth, PL and topological manifolds are equivalent. A lot of techniques have been developed in the last century to study 3-manifolds but most of them are very special and don't generalise to higher dimensions. The first interesting family of 3-manifolds to be classified were the flat Riemannian
manifolds-those which are locally isometric to Euclidean space. David Hilbert, in the $18^{th}$ of his famous problems, asked whether there were only finitely many discrete groups of rigid motions of the Euclidean $n$-space with compact fundamental domain. Ludwig Bieberbach (1886-1982) proved this statement in 1910, and in fact gave a complete classification of such groups. Compact 3-manifolds of constant positive curvature were classified in 1925, by Heinz Hopf (1894-1971). Twenty-five years later, Georges de Rham (1903-1990) showed that Hopf's classification, up to isometry, actually coincides with the classification up to diffeomorphism. The lens spaces, with finite cyclic fundamental group, constitute a subfamily of particular interest. These were classified up to piecewise-linear homeomorphism in 1935 by Reidemeister, Franz, and de Rham, using an invariant which they called torsion. (See Milnor 1966 as well as Milnor and Burlet 1970 for expositions of these ideas.)\\

In the 1980’s Thurston developed another approach to 3-manifolds, see \cite{Thu97} and \cite{CHK00}. He considered 3-manifolds with Riemannian metrics of constant negative curvature $-1$. These manifolds, which are locally isometric to the hyperbolic 3-space, are called hyperbolic manifolds. There are fairly obvious obstructions showing that not every 3-manifold can admit such a metric. Thurston formulated a general conjecture that roughly says that the obvious obstructions are the only ones; should they vanish for a particular 3-manifold then that manifold admits such a
metric. His proof of various important special cases of this conjecture led him to formulate a more general conjecture about the existence of locally homogeneous metrics, hyperbolic or otherwise, for all manifolds; this is called Thurston’s Geometrization Conjecture for 3-manifolds. An important point is that Thurston’s Geometrization Conjecture includes the Poincar$\grave{\rm e}$ Conjecture as a very special case. A proof of Thurston’s Geometrization Conjecture is given in \cite{MT08} following ideas of Perelman. For more details on the history of the Poincar$\grave{\rm e}$ Conjecture, the development of 3-manifold topology, and Thurston’s Geometrization Conjecture, see Milnor's survey article \cite{Mil03}.\\

For $n\geq4$ group-theoretic decision problems prevent a complete classifications of smooth $n$-manifolds, by the following argument. Every smooth manifold $M$ can be triangulated by a finite simplicial complex, so that the fundamental group $\pi_1(M)$ is finitely presented. Homotopy equivalent manifolds have isomorphic fundamental groups. Every finitely presented group arises as the fundamental group $\pi_1(M)$ of an $n$-dimensional manifold $M$. It is not possible to have a complete set of invariants for distinguishing the isomorphisms class of a group from a finite presentation.  Group-theoretic considerations thus make the following questions unanswerable in general :
\begin{itemize}
 \item [(i)] Is $M$ homotopy equivalent to $N$ ?
 \item [(ii)] Is $M$ diffeomorphic to $N$ ?
 \item[(iii)] Is $\pi_1(M)$ isomorphic to $\pi_1(N)$ ?
 \end{itemize}
The surgery method of classifying manifolds seeks to answer a different problem :
\begin{problem}\label{genequestion}
Let $f:N\to M$ denote a homotopy equivalence between manifolds.\\
Is $f$ homotopic to a homeomorphism, a PL homeomorphism or a diffeomorphism ?
\end{problem}
In this paper, we review the status of the above problem \ref{genequestion} for topological rigidity, its recent developments and many interesting open question along this direction.\\
This paper is organized as follows. In section 2, we give the notation and state the basic definitions and results that will be used throughout the paper.\\
In section 3, we study about aspherical manifolds. Aspherical manifolds are manifolds with contractible universal cover. Many examples come from certain kinds of non-positive curvature conditions. Their homotopy types are determined by the fundamental groups. Important rigidity conjectures state more strongly that the geometry of such spaces is specified by their fundamental group. For example, the Borel conjecture states that the homeomorphism type of a closed aspherical manifold is determined by the fundamental group. The counterexamples to most of the old conjectures stem from essentially two different
constructions of aspherical manifolds. The first was the reflection group trick of Michael Davis \cite{Dav83} and the second construction was Gromov's idea of hyperbolization \cite{Gro87}. Here we mention certain results and examples of M. Davis and J.C. Hausmann \cite{DH89} and M. Davis and T. Januszkiewicz \cite{DJ91}, which were given by using these constructions.\\  
In section 4, we discuss Problem \ref{genequestion} for topological rigidity. In other words if $M$, $N$ are two manifolds with isomorphic fundamental groups, then are $M$, $N$ are homeomorphic to one another?. In particular, we study the above topological rigidity problem for aspherical manifolds and non-positively curved manifolds. We also present the status of the problem beginning with the rigidity results of Bieberbach and Mostow.\\
In section 5, we discuss the Farrell-Jones conjectures whose truth would certainly imply the main topological results of Section 3, and would also imply many other well known conjectures in algebraic $K$- and $L$-theory and algebraic topology (e.g. the Novikov Conjectures, the Borel Conjecture in dimensions $\geq 5$, and the Conjecture for vanishing of the Whitehead group). We also here present the status of the Farrell-Jones conjectures.
\section{\large Basic Definitions and Concepts}
In this section, we review some basic definitions, results and notation to be used throughout the article:\\
We write $\rm{Diff}$ for the category of smooth manifolds, $\rm{PL}$ for the category of piecewise-linear manifolds, and $\rm{Top}$ for the category of topological manifolds. We generically write $\rm{CAT}$ for any one of these geometric categories. Let $I=[0, 1]$ be a fixed closed interval in $\mathbb{R}$.
$\mathbb{R}^n$ is $n$-dimensional Euclidean space, $\mathbb{D}^n$ is the unit disk, $\mathbb{S}^n$ is the unit sphere, $\Sigma_g$ is the closed orientable surface of genus $g$ and $T^n=\mathbb{S}^1\times \mathbb{S}^1\times....\times \mathbb{S}^1$ ($n$-factors) is $n$-dimensional torus with their natural smooth structures and orientations. Define $\textbf{H}^n= \{(x_1,x_2,...x_n)\in \mathbb{R}^n : x_1\geq 0\}$. $\mathbb{C}^n$ is $n$-dimensional complex space, and $\mathbb{H}^n$ is $n$-dimensional quaternionic space. $\rm{Out}(G)$ is the group of outer automorphisms of the group $G$, $\rm{Top}(X)$ is the group of all self-homeomorphisms of a topological space $X$ and $\rm{Isom}(M)$ is the group of isometries of a Riemannian Manifold $M$. The general linear group $GL(n,\mathbb{K})$ is the group consisting of all invertible $n\times n$ matrices over the field $\mathbb{K}$, and the group of orthogonal $n\times n$ real matrices is denoted by $O(n)$.\\
Topological spaces are typically denoted by $X$, $Y$, $Z$. Manifolds tend to be denoted by $M^n$, $N^n$, where $n$ indicates the dimension. Homotopy spheres will be represented by $\Sigma^n$.
\begin{definition}\label{struct}(Structure Sets)\rm{
Let $M$ be a closed topological manifold. We define $\mathcal{S}(M)$ to be the set of equivalence classes of pairs $(N,f)$ where $N$ is a closed manifold and $f: N\to M$ is a homotopy equivalence.
And the equivalence relation is defined as follows :
$(N_1,f_1)\sim  (N_2,f_2)$ if there is a homeomorphism $h:N_1\to N_2$ such that $f_2\circ h$ is homotopic to $f_1$.}
\end{definition}
\begin{definition}\label{struct.boun}\rm{
Let $M$ be a compact manifold with boundary $\partial M$. We define $\mathcal{S}(M,\partial M)$ to be the set of equivalence classes of pairs $(N,f)$ where $N$ is a compact manifold with boundary $\partial N$ and $f:(N,\partial N)\to (M, \partial M)$ is a homotopy equivalence such that the restricted map
$f_{\partial N}: \partial N \to  \partial M$ is a homeomorphism. And the equivalence relation is defined as follows :
\begin{center}
$(N_1,f_1)\sim  (N_2,f_2)$ if there is a homeomorphism $h:(N_1,\partial N_1)\to (N_2, \partial N_2)$ such that $f_2\circ h$ is homotopic to $f_1~~\rm{rel}~~ \partial N_1$. 
\end{center}}
\end{definition}
The following definitions are taken from \cite{FJ96}:
\begin{definition}(Normal Cobordism)\label{nor.cob}\rm{
Let $M^n$ be a closed topological manifold.\\ 
$(N,f)\in \mathcal{S}(M)$ is normally cobordant to $(M,id)$ if there exists a compact cobordism $W^{n+1}$ and a map $F:(W,\partial W)\to (M\times [0,1]$, $\partial)$
with the following properties.
\begin{itemize}
\item[(i)] The boundary $\partial W=N\coprod M$. (Set $N= \partial^{+} W$, $M= \partial^{-} W $)
\item[(ii)] The restriction map $F_{|\partial^{+} W} :\partial^{+} W\to M\times 1$ is equal to $f$ and $F_{|\partial^{-} W} :\partial^{-} W\to M\times 0$
is equal to the identity map.
\item[(iii)] The map $F$ is covered by a (vector) bundle isomorphism $\bar{F} :N(W)\to \mathcal{E}$, where $N(W)$ is the stable normal bundle of $W$ and
$\mathcal{E}$ is a bundle over $M\times [0,1]$.
\end{itemize}
From now onward we will denote a normal cobordism by the triple $(W,F,\simeq)$ where $\simeq$ denotes the isomorphism covering F.
We now put an equivalence relation $\sim$ on the set of all normal cobordisms. We say $(W_1,F_1,\simeq_1) \sim (W_2, F_2, \simeq_2)$ if and only if there exists a triple $(\mathcal{W},\mathcal{F}$, $\equiv)$ 
where $\mathcal{W}$ is a cobordism between $W_1$ and $W_2$ and $\mathcal{F}:\mathcal{W}\to M\times I\times I$ satisfies the following properties where $I=[0,1]$. The restriction $\mathcal{F}_{|W_1}:W_1\to M\times 0\times I$ is $F_1$, and $\mathcal{F}_{|W_2}:W_2\to M\times 1\times I$
is $F_2$. Also $\mathcal{F}_{|\mathcal{W}^{-}}=id_{M\times 0\times I}$ and  $\mathcal{F}_{|\mathcal{W}^{+}}$ is a homotopy equivalence, where $\mathcal{W}^{-}$ and $\mathcal{W}^{+}$ are described. Also, $\equiv$ is an isomorphism of $\mathcal{N}(\mathcal{W})$ (the stable normal bundle) to some bundle $\mathcal{\epsilon}$ over
$M\times I\times I$ which covers $\mathcal{F}$ and which restricts to $\simeq_1$ and $\simeq_2$ over $W_1$ and $W_2$, respectively.
When $m=\dim M\geq 4$, the equivalence classes of normal cobordisms form a group which depends only on $\pi_1(M)$ and the first Stiefel-Whitney class $\omega_1(M)\in H^1(M, \mathbb{Z}_2)$ (\cite{Wal71, KS77}).
When $M$ is orientable, then $\omega_1(M)=0$. The group of normal cobordisms modulo equivalence is denoted $L_{m+1}(\pi_1(M))$ where $m=\dim M$. 
(See \cite{Wal71} for a purely algebraic definition of the groups $L_n(\pi)$.) }
\end{definition}
\begin{definition}(Special Normal Cobordism)\rm{
 A normal cobordism  $(W,F,\simeq)$ is called a special normal cobordism if $F_{|\partial ^{+}W} : \partial ^{+}W\to M\times 1$ is a homeomorphism. \\
 We can define a stronger equivalence relation $\sim_s$ on the set of special normal cobordisms by requiring $\mathcal{F}_{|\mathcal{W^{+}}}$ of the earlier equivalence relation $\sim$ in the definition \ref{nor.cob} to be a homeomorphism. This set of special normal cobordisms modulo the equivalence relation $\sim_s$ is also an abelian group and it
 is naturally identified with $[M\times [0,1],\partial ; G/Top]$. Here $G/Top$ is an $H$-space and $[X, A ; G/Top]$ denotes the set of homotopy classes of maps $f:X\to G/Top$ such that the restriction $f_{|A}=1$, where 1 is the homotopy identity element in $G/Top$.}
\end{definition}
\begin{definition}\rm{
 We next define a variant of $\mathcal{S}(M)$ denoted by $\bar{\mathcal{S}}(M)$. The underlying set of  $\bar{\mathcal{S}}(M)$ is the same as that of 
  $\mathcal{S}(M)$. But now $(N_1,f_1)$ is said to be equivalent to $(N_2,f_2)$ if there exist an $h$-cobordism $W$ between $N_1$ and $N_2$ and a map $F:W\to M\times  I$ such that $F_{|\partial^{-}W}=f_1$ and  $F_{|\partial^{+}W}=f_2,$
  where $\partial^{-}W=N_1$ and $\partial^{+}W=N_2$.}
\end{definition}
\begin{remark}\rm{
 Note that $\mathcal{S}(M)= \bar{\mathcal{S}}(M)$ when $Wh(\pi_1(M))=0$ and $\dim M\geq 5$. A set $\bar{\mathcal{S}}(M, \partial M)$ can be defined similarly when $M$ is a compact manifold with boundary. (The notation
 $\bar{\mathcal{S}}(M, \partial M)$ is sometimes abbreviated to $\bar{\mathcal{S}}(M, \partial)$ and likewise $[M,\partial M;G/Top]$ to $[M,\partial; G/Top]$.)}
 \end{remark}
 The formulation of the surgery exact sequence given below is also due to the work of Sullivan \cite{Sul71} and Wall \cite{Wal71} refining the earlier work of Browder and Novikov \cite{Nov64} :
\begin{definition}\label{surexa}(Surgery Exact Sequence)\rm{
Let $M^n$ be a compact connected manifold with non-empty boundary. For any non-negative integer $m$, there is long exact sequence of pointed sets :\\
$....\stackrel{\pi}{\longrightarrow}\bar{\mathcal{S}}(M\times \mathbb{D}^m, \partial)\stackrel{\omega}{\longrightarrow}[M\times \mathbb{D}^m,\partial; G/Top]\stackrel{\sigma}{\longrightarrow}  L_{n+m}(\pi_1M)
\longrightarrow....\\
\longrightarrow \bar{\mathcal{S}}(M\times \mathbb{D}^1, \partial)\stackrel{\omega}{\longrightarrow}[M\times \mathbb{D}^1,\partial; G/Top]\stackrel{\sigma}{\longrightarrow}
L_{n+1}(\pi_1M)\stackrel{\tau}{\longrightarrow}\bar{\mathcal{S}}(M,\partial)\stackrel{\omega}{\longrightarrow}[M ,\partial; G/Top]\\
\stackrel{\sigma}{\longrightarrow}  L_{n}(\pi_1M).$\\

Recall that $L_{n+m}(\pi_1M)= L_{n+m}(\pi_1(M\times \mathbb{D}^{m-1})$) is the set of equivalence classes of normal cobordisms on $M\times \mathbb{D}^{m-1}$ and that 
\begin{center}
 $[M\times \mathbb{D}^m,\partial; G/Top]=[M\times \mathbb{D}^{m-1}\times [0,1],\partial; G/Top]$
\end{center}
consists of the equivalence classes of special normal cobordisms on $M\times \mathbb{D}^{m-1}$. Then, $\sigma$ is the map which forgets the special structure; while, $\tau$ sends a normal cobordism $W$ to its top $\partial^{+}W$.
The maps $\omega$, when $m\geq 1$, similarly have a natural geometric description. We illustrate this when $m=1$. Let $(W,F)$ represent an element $x$ in $\bar{\mathcal{S}}(M\times [0,1],\partial)$. Then, $W$ is an $h$-cobordism between $\partial^{+}W$ and $\partial^{-}W$.
Furthermore, the restrictions $F_{|\partial^{-}W}: \partial^{-}W\to M\times 0$ and $F_{|\partial^{+}W}: \partial^{+}W\to M\times 1$ are both homeomorphisms. If  the first of these two homeomorphisms is $id_M$, then $(W,F,\simeq)$ is also a special normal cobordism and, considered as such, is $\omega(x)$. A bundle isomorphism $\simeq$ with domain $N(W)$ is determined since $F$ is a homotopy equivalence.
But it is easy to see that $(W,F)$ is equivalent in $\bar{\mathcal{S}}(M\times [0,1], \partial)$ to an object $(W',F')$ such that $ \partial^{-}W'=M$ and $F'_{|\partial^{-}W'}=id_M.$ }
\end{definition}
\begin{definition}\label{sur.map}\rm{
 The map $\sigma$ in the surgery sequence in Definition \ref{surexa} is called the surgery map or the assembly map.}
\end{definition}
\begin{definition}(Simple Normal Cobordism)\rm{
 A normal cobordism  $(W,F,\simeq)$ is called a simple normal cobordism if $F_{|\partial ^{+}W}: \partial ^{+}W\to M\times 1$ is a simple homotopy equivalence. \\
 There is an obvious equivalence relation on the set of simple normal cobordisms analogous to the equivalence relation on normal cobordisms and special normal cobordisms.
 Wall \cite{Wal71} showed that the equivalence classes of simple normal cobordisms form an abelian group, denoted by $L^{s}_{n+1}(\pi_1M)$, which depends only on $\pi_1M$ and on the first Stiefel-Whitney class $\omega_{1}(M)$. 
 The forget-structure maps define group homomorphisms
 \begin{center}
  $\bar{\sigma}: [M\times [0,1], \partial; G/Top]\to L^{s}_{n+1}(\pi_1M)$ and \\
 $ \eta :L^{s}_{n+1}(\pi_1M)\to L_{n+1}(\pi_1M)$.
 \end{center}
And these homomorphisms factor the surgery map 
\begin{center}
 $\sigma :  [M\times [0,1], \partial; G/Top]\to L_{n+1}(\pi_1M)$
\end{center}
as $\sigma=\eta\circ \bar{\sigma}$. It is known that $\eta$ is an isomorphism after tensoring with $\mathbb{Z}[\frac{1}{2}]$. This is a consequence of Rothenberg's exact sequence \cite[ p.248]{Wal71}. Of course $\eta$ is an isomorphism before tensoring with $\mathbb{Z}[\frac{1}{2}]$ if $Wh(\pi_1M)=0$ \cite{KL04}.}
\end{definition}
\begin{definition}(Spectrum)\rm{
A spectrum $${\bf{E}} = \{(E_n, \sigma_n) | n\in \mathbb{Z}\}$$ is a sequence of pointed spaces $\{E_n |n\in \mathbb{Z}\}$ together with pointed maps called structure maps $$\sigma_n : E_n\wedge \mathbb{S}^1\to E_{n+1}.$$ A map of spectra $f : {\bf{E}}\to {\bf{E}}^{'}$
is a sequence of maps $f_n : E_n \to E^{'}_n$ which are compatible with the structure maps $\sigma_n$, i.e., $$f_{n + 1}\circ \sigma_n = \sigma^{'}_n\circ (f_n\wedge id_{\mathbb{S}^1})$$ holds for all $n\in \mathbb{Z}$.}
\end{definition}
\begin{definition}($\Omega$-spectrum)\rm{
Given a spectrum ${\bf{E}}$, we can consider instead of the structure map $\sigma_n : E_n\wedge \mathbb{S}^1\to E_{n + 1}$ its adjoint
$\sigma^′_n : E_n\to \Omega E_{n + 1} = \rm{map}(\mathbb{S}^1, E_{n + 1})$. We call ${\bf{E}}$ an $\Omega$-spectrum if each map $\sigma^′_n$ is a weak homotopy equivalence.}
\end{definition}
\begin{definition}(Homotopy groups of a spectrum)\rm{
Given a spectrum ${\bf{E}}$, define for $n\in \mathbb{Z}$ its $n$-th homotopy group
$$\pi_n({\bf{E}}) := colim_{k\to \infty} \pi_{k+n}(E_k)$$ to be the abelian group which is given by the colimit over the directed system indexed by $\mathbb{Z}$ with $k$-th structure map $$\pi_{k+n}(E_k) \stackrel{\sigma^′_k}{\rightarrow}\pi_{k+n}(\Omega E_{k+1})=\pi_{k+n+1}(E_{k+1})$$
Notice that a spectrum can have, in contrast to a space, non-trivial negative homotopy groups. If ${\bf{E}}$ is an  $\Omega$-spectrum, then $\pi_n({\bf{E}})= \pi_n(E_0)$ for all $n\geq 0$.}
\end{definition}
\begin{definition}\rm{
Let $q\in \mathbb{Z}$. An $\Omega$-spectrum ${\bf{F}}$ is $q$-connective if $\pi_n({\bf{F}})= 0$ for $n< q$. A $q$-connective cover of an $\Omega$-spectrum ${\bf{F}}$ is a $q$-connective $\Omega$-spectrum ${\bf{F}}\langle q \rangle $ together with a map ${\bf{F}}\langle q \rangle\mapsto {\bf{F}}$ inducing isomorphisms $\pi_n({\bf{F}}\langle q \rangle)\cong \pi_n({\bf{F}})$ for $n\geq q$. In general, ${\bf{F}}\langle q \rangle$ is obtained from ${\bf{F}}$ by killing the homotopy groups $\pi_n({\bf{F}})$ for $n < q$, using Postnikov decompositions and Eilenberg-MacLane spectra.}
\end{definition}
If $f : (X, x)\to (Y, y)$ is any pointed map of spaces, we can naturally construct a fibration $\widetilde{f} : \widetilde{X}\to Y$ together with a homotopy equivalence $X\mapsto \widetilde{X}$ over $Y$. We denote by $htyfib(f)$, the fibre $\widetilde{f}^{-1}(y)$ of $\widetilde{f}$.
\begin{theorem}(Quillen's plus construction)
Let $G$ be a discrete group and $H\subset G$ be a perfect normal subgroup. Then there exists a CW-complex $BG^{+}$ and
a continuous map $\gamma: BG\to BG^+$ such that $\rm{ker}({\pi_1(BG)\to \pi_1(BG)^+)}= H$ and such that $\widetilde{H}^*(htyfib(\gamma), \mathbb{Z}) =0$. Moreover, $\gamma$ is unique up to homotopy.
\end{theorem}
\begin{definition}\rm{
For any ring $R$, let $\gamma : BGL(R)\to BGL(R)^+$ denote the Quillen's plus construction with respect to $[GL(R),GL(R)]\subset GL(R)$. We define $K_i(R)=\pi_i(BGL(R)^+)$, $i>0$.}
\end{definition}
\begin{definition}\rm{
 We define the $K$-theory space to be ${\bf{K}}(R) =K_0(R)\times BGL(R)^{+}$, and then for all $i\geq 0$ we can set
$$K_i(R)=\pi_i({\bf{K}}(R)).$$}
\end{definition}
\begin{definition}(Negative K-theory)\rm{
Define inductively for $n=-1$, $-2$,....
$$K_n(R):= \rm{coker}(K_{n+1}(R[t])\oplus K_{n+1}(R[t^{-1}])\to K_{n+1}(R[t, t^{-1}])).$$}
\end{definition}
\begin{definition}(K-theory of a unital $\mathbb{C}^{*}$-algebra)\rm{
Let $A$ be a $\mathbb{C}^{*}$-algebra with unit $1_A$. Define
$$K_i(A)=\pi_{i-1}(GL(A)),~i=1, 2, 3,..$$
and $K_0(A)$ as the algebraic $K$-theory group of the ring $A$.\\
Let $G$ be a discrete group and the Hilbert space $$l^2(G)=\{f : G\to \mathbb{C} : \sum_{\gamma \in G}|f(\gamma)|^{2} < \infty \},$$ and let $B(l^2(G))$ be the $\mathbb{C}^*$-algebra of all bounded linear operators $T : l^2(G)\to l^2(G)$ . The reduced $\mathbb{C}^*$-algebra of $G$, denoted by $\mathbb{C}_r^*(G)$ which  is the norm closure of the $^*$-algebra generated operators of the form $$L_{\gamma}(f)(\mu)=f(\gamma^{-1}\mu),$$ for $f\in l^2(G)$, $\gamma$, $\mu \in G$. This amounts to embedding the group ring $\mathbb{C}G$ in $B(l^2(G))$ by letting elements act by left convolution, and then close this embedding with respect to the operator norm on $B(l^2(G))$.}
\end{definition}
\begin{definition}($G$-Homology theory)\rm{
Let $\Lambda$ be a commutative ring. A $G$-homology theory $\mathcal{H}^{G}_∗$  is a covariant functor from the category of $G$-CW-pairs to the category of $\mathbb{Z}$-graded $\Lambda$-modules together with natural transformations $$\partial^{G}_n(X,A) : \mathcal{H}^{G}_n(X,A)\to \mathcal{H}^{G}_{n-1}(A)$$ for $n\in \mathbb{Z}$ satisfying the following axioms:
\begin{itemize}
 \item $G$-Homotopy invariance;
 \item Long exact sequence of pairs;
 \item Excision;
 \item Disjoint union axiom.
\end{itemize}}
\end{definition}
\begin{definition}(Equivariant homology theory)\rm{
An equivariant homology theory $\mathcal{H}^{?}_{∗}$ assigns to every group $G$, a $G$-homology theory $\mathcal{H}^{G}_{∗}$. These are linked together with the following so called induction structure: given a group homomorphism $\alpha: H\to G$ and a $H$-CW-pair $(X,A)$, there are for all $n\in \mathbb{Z}$ natural homomorphisms $$ \rm{ind}_{\alpha}:\mathcal{H}^{H}_n(X,A)\to \mathcal{H}^{G}_{n}(\rm{ind}_{\alpha}(X,A))$$ satisfying
\begin{itemize}
 \item Bijectivity: If $\rm{ker}(\alpha)$ acts freely on $X$, then $\rm{ind}_{\alpha}$ is a bijection;
 \item Compatibility with the boundary homomorphisms;
 \item Functoriality in $\alpha$;
 \item Compatibility with conjugation.
\end{itemize}}
\end{definition}
\begin{theorem}(L$\ddot{\rm u}$ck-Reich (2005))
Given a functor ${\bf{E}}: Groupoids\to Spectra$ sending equivalences to weak equivalences, there exists an equivariant homology
theory $\mathcal{H}{^?}_{∗}(-;{\bf{E}})$ satisfying $$\mathcal{H}^{H}_{n}(pt)\cong \mathcal{H}^{G}_{n}(G/H)\cong \pi_n({\bf{E}}(H))$$
\end{theorem}
\begin{theorem}(Equivariant homology theories associated to $K$ and $L$-theory, Davis-L$\ddot{\rm u}$ck (1998))
Let $R$ be a ring (with involution). There exist covariant functors
\begin{align*}
{\bf{K}}_{R}:& ~Groupoids \to Spectra;\\
{\bf{L}}^{\langle \infty \rangle}_{R}:&~ Groupoids \to Spectra;\\
{\bf{K}}^{top}:&~ inj-Groupoids \to Spectra;
\end{align*}
with the following properties:
\begin{itemize}
 \item They send equivalences of groupoids to weak equivalences of spectra;
 \item For every group $G$ and all $n\in \mathbb{Z}$ we have
 \begin{align*}
  \pi_n({\bf{K}}_{R}(G))&\cong K_n(RG);\\
  \pi_n({\bf{L}}^{\langle \infty \rangle}_{R}(G))&\cong {\bf{L}}^{\langle \infty \rangle}_{n}(RG);\\
  \pi_n({\bf{K}}^{top}(G))&\cong K_n(C^{*}_r(G));
\end{align*}
\end{itemize}
\end{theorem}
\begin{definition}(Smash product)\rm{
Let ${\bf{E}}$ be a spectrum and $X$ be a pointed space. Define the smash product $X\wedge {\bf{E}}$ to be the spectrum whose $n$-th space is $X\wedge E_n$ and whose $n$-th structure map is $$X\wedge E_n\wedge \mathbb{S}^1  \stackrel{id_X\wedge \sigma_n}{\rightarrow}X\wedge E_{n+1}.$$}
\end{definition}
\begin{theorem}(Homology theories and spectra)
Let $E$ be a spectrum. Then we obtain a homology theory $H_∗(−;{\bf{E}})$ by
$$H_n(X,A;{\bf{E}}):= \pi_n((X\cup_A cone(A))\wedge {\bf{E}}).$$
It satisfies $H_n(pt;{\bf{E}})= \pi_n({\bf{E}})$.
\end{theorem}
\section{\large Aspherical Manifolds}
Given a pair of topological spaces $X$ and $Y$ with isomorphic homology groups or homotopy groups, are they homotopy equivalent? An important theorem of Whitehead answers this question for CW-complexes:
\begin{theorem}\label{whitehead}(J.H.C. Whitehead)
A continuous map $f: M\to N$ between CW-complexes is a homotopy equivalence iff it induces an isomorphism on $\pi_{m}$ for all $m$.\\
Moreover, if $\pi_{1}(X)=0=\pi_{1}(Y)$ then $f$ is a homotopy equivalence if and only if $f_*$ is an isomorphism on homology
groups.
\end{theorem}
\begin{remark}\rm{
Whitehead requires that the isomorphism is induced by a continuous map. There are examples where $\pi_{m}(M)\cong \pi_{m}(N)$ for all $m$; but $M$ is not homotopy equivalent to $N$. Here are explicit examples:
\begin{itemize}
\item[\bf{1.}] Consider $M=\mathbb{S}^3\times \mathbb{R}\textbf{P}^2$ and $N=\mathbb{S}^2\times \mathbb{R}\textbf{P}^3$. Both of them have fundamental group $\mathbb{Z}_2$ and universal cover $\mathbb{S}^3\times \mathbb{S}^2$. So their homotopy groups are all the same. On the other hand, only the latter is orientable since $\mathbb{R}\textbf{P}^3$ is orientable but $\mathbb{R}\textbf{P}^2$ isn't, so they have different values on $H_5$ and therefore can't be homotopy equivalent. 
\item[\bf{2.}] Let $M=\mathbb{S}^2\times \mathbb{S}^2$ and $N=S(\eta^2\oplus \epsilon^1)$ where $\eta^2$ is the canonical $\mathbb{C}$-line bundle over $\mathbb{C}\textbf{P}^1=\mathbb{S}^2$, $\epsilon^1$ is the trivial $\mathbb{R}$-line bundle and $S(\eta^2\oplus \epsilon^1)$ denotes the sphere bundle associated to the Whitney sum $\eta^2\oplus \epsilon^1$. Since the fibration $$\mathbb{S}^2\mapsto N\mapsto \mathbb{S}^2$$ has a cross section, it follows
that the long exact homotopy sequence for this fibration splits and that there are isomorphisms $$\pi_i(N)\cong \pi_i(\mathbb{S}^2)\oplus  \pi_{i}(\mathbb{S}^2)\cong \pi_i(\mathbb{S}^2\times \mathbb{S}^2).$$ Finally, to see $M$ and $N$ are not homotopy equivalent, one computes their Stiefel-Whitney classes. It turns out $\omega_2(\mathbb{S}^2\times \mathbb{S}^2)=0$ while $\omega_2(N)\not=0$. Since the Stiefel-Whitney classes can be defined in terms of Steenrod powers, they are homotopy invariants, so $M$ and $N$ are not homotopy equivalent.\\ However there is an important special case where this worry is unnecessary:
\end{itemize}}
\end{remark}
\begin{definition}\rm{
A closed connected manifold $M$ is called an aspherical manifold if $\pi_{m}(M)=0$ for $m\neq 1$. (This is equivalent to requiring that the universal cover $\widetilde{M}$ of $M$ is contractible.}
\end{definition}
From the homotopy theory point of view an aspherical manifold is completely determined by its fundamental group due to Hurewicz :
\begin{theorem}\label{hurewicz}
If $\pi_{1}(M)\cong \pi_{1}(N)$ and both $M$ and $N$ are aspherical, then $M$ and $N$ are homotopy equivalent.
\end{theorem}
\begin{example}(Aspherical Manifolds)\rm{
\begin{itemize}
\item[\bf{1}.] A connected closed $1$-dimensional manifold is homeomorphic to $\mathbb{S}^1$ and hence aspherical.
\item[\bf{2}.] Let $M$ be a connected closed 2-dimensional manifold. Then $M$ is either aspherical or homeomorphic to $\mathbb{S}^2$ or $\mathbb{R}\mathbb{P}^2$. This can be easily seen from the following facts :
Let $M\neq  \mathbb{S}^2$, $\mathbb{R}\mathbb{P}^2$. We may assume that $M$ is a CW-complex, hence its universal cover $\widetilde{M}$ is a 2-manifold, a CW-complex. $\widetilde{M}$ must be non-compact as $\pi_1(M)$ is infinite. Thus $H_2(\widetilde{M})= 0$. Moreover $H_1(\widetilde{M}) = 0$ and $H_i(\widetilde{M}) = 0$, $i > 1$. Hence by Hurewicz theorem $\pi_i(\widetilde{M})=0$, $i\geq 1$. Whitehead theorem now says that $\widetilde{M}$ is contractible.
\item[\bf{3.}] In dimension 3; Note that $\mathbb{S}^2\times \mathbb{S}^1$ has fundamental group $\mathbb{Z}$, $\mathbb{R}\mathbb{P}^3$ has fundamental group $\mathbb{Z}_2$ and clearly they are not aspherical. Consider the 3-manifold $M= \mathbb{S}^2\times \mathbb{S}^1\#\mathbb{S}^2\times \mathbb{S}^1$; it has fundamental group $\mathbb{Z}*\mathbb{Z}$. Thus $\pi_1(M)$ is infinite and $H_1(M)=\mathbb{Z}\oplus \mathbb{Z}$, $H^1(M)=\mathbb{Z}\oplus \mathbb{Z}$ (by Universal coefficient Theorem). Hence by Poincar$\grave{\rm e}$ Duality, $H_2(M)= \mathbb{Z}\oplus \mathbb{Z}$. Now if we suppose $M$ to be aspherical, then $M=K(\mathbb{Z}*\mathbb{Z}, 1)$ is $\mathbb{S}^1\vee \mathbb{S}^1$ up to homotopy. This contradicts the fact that $H_2(M)= \mathbb{Z}\oplus \mathbb{Z}$. Thus all the above cases are non-examples.\\
In fact, all closed oriented 3-manifolds $M$ such that $\pi=\pi_1(M)$ is not isomorphic to $\mathbb{Z}$, a finite group or a non-trivial free product are aspherical. This can be easily seen from the following facts: If $M$ is a closed oriented 3-manifolds then there exists a unique (up to homeomorphism) collection of oriented prime manifolds $P_1$...., $P_k$ such that $M= P_1\#,...\#P_k$, the connected sum.  This is known as prime decomposition theorem. Now if $P$ is an oriented prime 3-manifold then either it is $\mathbb{S}^2\times \mathbb{S}^1$ or any embedded 2-sphere in $P$ bounds a ball (also called irreducible). In other words, $M$ is a connected sum of irreducible 3-manifolds and copies of $\mathbb{S}^2\times \mathbb{S}^1$. Moreover it is known that any orientable prime 3-manifold $P$ with $\pi_2(P)\neq 0$ is $\mathbb{S}^2\times \mathbb{S}^1$.
Thus essentially in dimension 3, orientable prime 3-manifolds for which $\pi_2(P)=0$ and $\pi_1(P)$ is infinite are aspherical.
To see this, let $\widetilde{P}$ be the universal cover of $P$. Since $\pi_1(P)$ is infinite, $\widetilde{P}$ is non-compact and $\pi_2(\widetilde{P})=\pi_2(P)=0$. Moreover $\widetilde{P}$ being a non-compact 3-manifold, $H_3(\widetilde{P})= 0$ and $H_i(\widetilde{P}) = 0$, $i > 1$. Hence by Hurewicz Theorem $\pi_i(\widetilde{P})=0$, $i\geq 1$. It follows that $\widetilde{P}$ is contractible by Whitehead Theorem.\\ Thurston's Geometrization Conjecture implies that a closed 3-manifold is aspherical if and only if its universal covering is homeomorphic to $\mathbb{R}^3$. This follows from \cite[Theorem 13.4 on page 142]{Hem76} and the fact that the 3-dimensional geometries which have compact quotients and whose underlying topological spaces are contractible have as underlying smooth manifold $\mathbb{R}^3$ (see  \cite{Sco83}). A proof of Thurston's Geometrization Conjecture is given in  \cite{MT08} following ideas of Perelman. There are examples of closed orientable 3-manifolds that are aspherical but do not support a 
Riemannian 
metric with nonpositive sectional curvature (see  \cite{Lee95}). For more information about 3-manifolds we refer to  \cite{Hem76, Sco83}.
\item[\bf{4}.] Any complete non-positively curved Riemannian manifold is aspherical. This follows from the Cartan-Hadamard Theorem. Special cases are flat Riemannian manifolds and locally symmetric spaces of non-compact type.
\item[\bf{5}.] Let $G$ be a non-compact Lie group and $K$ a maximal compact subgroup. Then $G/K$ is diffeomorphic to $\mathbb{R}^n$ for some $n$. Let $\Gamma$ be a discrete torsion free subgroup of $G$. The natural $\Gamma$-action on $G/K$ is free and proper. Hence, the double coset space $\Gamma\setminus G/K$ is aspherical. In the special case where G is virtually nilpotent and $\pi_{1}(G)=1$, the double coset space $\Gamma\setminus G/K$ is called an infranilmanifold.
\item[\bf{6}.] On the basis of such examples the following conjecture was made in  \cite{Joh71} :
\end{itemize}}
\end{example}
\begin{conjecture}\label{johnson A}(F. E. A. Johnson)
Let $M$ be a closed manifold of type $K(\pi, 1)$. Then universal covering space $\widetilde{M}$ is homeomorphic to $\mathbb{R}^n$.
\end{conjecture}
More general than Conjecture \ref{johnson A} would be  \cite{Joh74}:
\begin{conjecture}\label{johnson B}(F. E. A. Johnson)
Let $M$ be a manifold of type $K(\pi, 1)$. Then universal covering space $\widetilde{M}$ is homeomorphic to $\mathbb{R}^n$.
\end{conjecture}
\begin{remark}\rm{
\indent
\begin{itemize}
\item [\bf{1.}] F. E. A. Johnson  \cite{Joh74} proved the Conjecture \ref{johnson A} when $n\geq 5$ and $\pi$ is a non-trivial direct product. Finally, F. E. A. Johnson  \cite{Joh74} gave an example to show that the generalized Conjecture \ref{johnson B} is false in each dimension $n\geq 4$.
\item [\bf{2.}] The counterexamples to most of the old conjectures stem from essentially two different constructions of aspherical manifolds. The first was the “reflection group trick” of Michael Davis  \cite{Dav83} which yielded the first aspherical manifolds whose universal covers are not Euclidean spaces. The second construction of aspherical closed manifolds was Gromov's idea of hyperbolization  \cite{Gro87}. It implies that aspherical manifolds exist in abundance. For instance, any closed PL manifold is the image of an aspherical manifold by a degree one tangential map, and any cobordism class can be represented by an aspherical manifold. In both constructions (Gromov's and Michael Davis), the fundamental groups of the aspherical manifolds are centerless. Interestingly, Lee and Raymond  \cite{LR75} showed that if the fundamental group of an aspherical manifold has nontrivial center, or more generally contains a nontrivial abelian normal subgroup, then the universal cover is homeomorphic to an Euclidean space. This is 
rather uncommon in the setting of Davis constructions.
\item[\bf{4}.] Michael W. Davis discussed the reflection group trick using the theory of Coxeter groups to construct a large number of new examples of closed aspherical manifolds in  \cite{Dav84}. The most striking consequence of this construction is the existence of counter examples to the Conjecture \ref{johnson A} in each dimension $\geq 4$.
\end{itemize}}
\end{remark}
First we will discuss Davis construction of aspherical manifolds by using the reflection group trick  \cite{Dav83}:
\begin{definition}\label{coxeter}\rm{
Suppose that $\Gamma$ is a group and $V$ is a set of generators, each element of which has order two. For any pair of elements $(v,w)$ of $V$, let $m(v, w)$ denote the order of $vw$ in $\Gamma$. Since $vw = (wv)^{-1}$, we have $m(v, w) = m(w, v)$. Let $E$ be the set of unordered pairs $\{v, w \}$ of distinct elements in $V$ such that $m(v, w)\neq \infty$. The pair $(\Gamma, V)$ is a Coxeter system and $\Gamma$ is a Coxeter group if the set of generators $V$ together with the relations $v^2= 1$ for $v\in V$, $(vw)^{m(v,w)}=1$ for ${v,w}\in E$ form a presentation for $\Gamma $.}
\end{definition}
\begin{definition}\rm{
Suppose that $(\Gamma, V)$ is a Coxeter system, that $X$ is a Hausdorff space and that $(X_v)_{v\in V}$ is a locally finite family of closed subspaces indexed by $V$ (The $X_v$ are called the panels of $X$). Let $\Gamma_{S}$ be the subgroup generated by $S \subset V$ and let $X_S$ be the face of $X$ defined by $X_S=\cap_{v\in S}X_v$. Denote by $K_{0}(\Gamma,V)$ or $K_0$(resp.$D_0(X)$ or $D_0$), the abstract simplicial complex with vertex set $V$ and with simplices, those non empty subset $S$ of $V$ such that $\Gamma_{s}$ is finite (resp.such that $X_S$ is nonempty). Thus, $D_0$ is the nerve of the covering of $\partial X (= \cup_{v\in V}X_v)$ by its panels.}
\end{definition}
\begin{definition}\rm{
Let $K$ be an abstract simplicial complex and let $S\in K$. The link of $S$ in $K$, denoted by $Link(S;K)$, is the abstract simplicial complex consisting of all simplices $T\in K$ such that $S\cap T = \emptyset $ and $S\cup T\in K$. An $n$-dimensional abstract simplicial complex $K$ is a generalized $n$-manifold (or a Cohen-Macaulay complex) if $H_{*}(Link(S; K))=H_{*}(\mathbb{S}^{n-\dim(S)-l})$ for all $S\in K$. If, in addition, $|K|$ has the homology of $S^{n}$, then it is a generalized homology $n$-sphere.}
\end{definition}
\begin{definition}\rm{
Let $G$ be a discrete group acting on a Hausdorff space $X$.\\ The action is proper if the following three conditions hold:
\begin{itemize}
\item[\bf{(a)}] the orbit space $X/G$ is Hausdorff,
\item[\bf{(b)}] for each $x\in X$ the isotropy subgroup $G_{x}$ is finite,
\item[\bf{(c)}] each $x\in X$ has a $G_{x}$-invariant open neighborhood $U_{x}$ such that $gU_{x}\cap U_{x}=\emptyset$ whenever $g\notin G_{x}$. Next suppose that $X$ is an $n$-dimensional manifold that $G$ acts properly. The action is  locally smooth if 
\item[\bf{(d)}] For each $x\in X$ there is an open neighborhood $U_{x}$ as in (c) and a faithful representation $G_{x}\to O(n)$ so that $U_{x}$ is $G_{x}$-homeomorphic to $\mathbb{R}^n$ with the linear $G_{x}$-action given by the representation. Such a neighborhood $U_{x}$ is called a linear neighborhood $x$.
\end{itemize}}
\end{definition}
\begin{definition}\rm{
A reflection on a connected manifold $M$ is a locally smooth involution  $r: M\to M$ such that the fixed point set $M_{r}$ separates $M$. Suppose that $\Gamma$ is a discrete group acting properly, locally smoothly and effectively on a connected manifold M and that $\Gamma$ is generated by reflections. Then $\Gamma$ is a reflection group on $M$.}
\end{definition}
\begin{definition}\rm{
Let $\Gamma$ be  a reflection group on $M$. Let $R$ denote the set of all reflections in $\Gamma$. For each $x\in M$, let $R(x)$ be the
set of all $r$ in $R$ such that $x$ belongs to $M_{r}$. A point $x$ is nonsingular if $R(x) = \emptyset$; otherwise it is singular. A chamber of $\Gamma$ on $M$ is the closure of a connected component of the set of nonsingular points. Let $Q$ be a chamber. Denote by $V_{Q}$(or simply $V$) the set of reflections $v$ such that $R(x)= \{v \}$ for some $x\in Q$. If $v\in V$, then $Q_{v}=M_{v}\cap Q$ is a panel of $Q$. $V$ is the set of reflections through panels of $Q$. As a convenient shorthand, we shall say that $(\Gamma, V)$ is a reflection system on $M$ with fundamental chamber $Q$. A reflection system is cocompact if its fundamental chamber is compact.\\
For any $x\in Q$ denote by $V(x)$ the intersection of $R(x)$ with $V$. In other words, $V(x)$ is the set of reflections through the panels of $Q$ which contain $x$. For any subset $T$ of $R$ let $\Gamma_T$ denote the subgroup of $\Gamma$ generated by $T$.}
\end{definition}
\begin{definition}\rm{
Let $(\Gamma, V)$ be a Coxeter system and $X$ a space with faces indexed by $V$. Give $\Gamma$ the discrete topology. Define an equivalence relation $\sim$ on $\Gamma\times X$ by $(g, x)\sim (h, y) \Leftrightarrow x = y$ and $g^{-1}h\in \Gamma_{V(x)}$. The natural $\Gamma$-action on $\Gamma\times X$ is compatible with the equivalence relation; hence, it passes to an action on the quotient space $\Gamma\times X/\sim $. Denote this quotient space by $\mathcal{U}(\Gamma, X)$ (or simply by $\mathcal{U}$) and call it the $\Gamma$-space associated to $(\Gamma, X)$.}
\end{definition}
\begin{definition}\rm{
Let $C^n$ be the standard simplicial cone in $\mathbb{R}^n$ defined by the linear inequalities $x_i\geq 0$, $1\leq i \leq n$. For any $x = (x_1,... ,x_n)\in C^n$, its codimension $c(x)$ is the number of $x_i$ which are equal to 0.\\
An $n$-manifold with corners $Q$ is a Hausdorff space together with a maximal atlas of local charts onto open subsets of $C^n$ so that the overlap maps are homeomorphisms which preserve codimension. For any $x\in Q$, its codimension $c(x)$ is then well defined. An open pre-face of $Q$ of codimension $m$ is a connected component of $c^{-1}(m)$. A closed pre-face is the closure of an open pre-face. For any $x\in Q$, let $\Sigma(x)$ be the set of closed pre-faces of codimension one which contain $x$. The manifold with corners $Q$ is nice if $Card(\Sigma(x))= 2$ for any $x$ with $c(x) = 2$.\\
 A manifold with faces is a nice manifold with corners $Q$ together with a panel structure on $Q$ such that: 
 \begin{itemize}
 \item[(a)]Each panel is a pairwise disjoint union of closed pre-faces of codimension one,
 \item[(b)] Each closed pre-face of codimension one is contained in exactly one panel.
 \end{itemize}}
\end{definition}
\begin{definition}\rm{
Suppose $(\Gamma, V)$ is a Coxeter system and that a space $X$ has a panel structure indexed by $V$. The panel structure is $\Gamma$-finite if the subgroup $\Gamma_{V(x)}$ is finite for all $x\in X$.}
\end{definition}
\begin{remark}\rm{
In dimension $\geq 4$ a necessary and sufficient condition for contractible manifold to be homeomorphic to a Euclidean space is that it be simply connected at infinity (A non compact space $Y$ is simply connected at infinity if every neighborhood of infinity  contains a simply connected neighborhood of infinity)  \cite{Dav83}.}
\end{remark}
We will need the following theorems to construct Davis examples of closed aspherical manifolds  \cite{Dav83}:
\begin{theorem}\label{davis1}
Let $L$ be a generalized homology sphere. Then there is a subdivision $L^{*}$ of $L$ and a cocompact reflection system $(\Gamma,V)$ on a contractible manifold with $K_{0}(\Gamma,V)=L^{*}$.
\end{theorem}
\begin{theorem}\label{davis2}
Let $(\Gamma,V)$ be a cocompact reflection system on a contractible manifold $M$. Then $M$ is simply connected at infinity iff $|K_{0}(\Gamma,V)|$ is simply connected.
\end{theorem}
\begin{corollary}\label{refdavis-cocompact}
In every dimension $\geq 4$ there exist cocompact reflection systems on contractible manifolds not homeomorphic to a Euclidean space.
\end{corollary}
\begin{definition}\rm{
A compact manifold with faces is a homology-cell (resp. a homotopy-cell) if each face is acyclic (resp. contractible).}
\end{definition}
\begin{theorem}
If $Q$ is a homology-cell of dimension $n+1$, then $D_o(Q)$ is a generalized homology $n$-sphere.
\end{theorem}
Conversely, we have the following result:
\begin{theorem}\label{davisdual}
Let $K_o$ be a generalized homology $n$-sphere. Then there is a homotopy $(n + 1)$-cell $Q$ with $D_o(Q)= K_0$.
\end{theorem}
\begin{theorem}\label{davisman}
Let $(\Gamma,V)$ be a Coxeter system and let $Q$ be a connected manifold with faces with $\Gamma$-finite panel structure indexed by $V$. Put $M=\mathcal{U}(\Gamma, Q)$. Then $M$ is a manifold and $(\Gamma,V)$ is a reflection system on $M$.
\end{theorem}
\begin{corollary}\label{davis-cocompact}
In every dimension $\geq4$ there exist closed aspherical manifolds whose universal cover is not homeomorphic to Euclidean space.
\end{corollary}
\begin{remark}\rm{
The idea in the construction of examples given by Theorem \ref{davis-cocompact} as follows:\\
Start with a simplicial complex $L$ which is a generalized homology sphere. Choose a Coxeter system $(\Gamma,V)$ with $K_0(\Gamma,V)=L$ given by Theorem \ref{davis1}. Finally, by Theorem \ref{davisdual}, there exists  a compact manifold with faces $X$ with $\Gamma$-finite panel structure satisfying the following condition:
\begin{itemize}
\item[($\star$)] $X$ is contractible and for each subset $S$ of $V$ such that $\Gamma_S$ is finite, the face $X_S$ is acyclic.
\end{itemize}
Now consider the $\Gamma$-space $M=\mathcal{U}(\Gamma, X)$. By Theorem \ref{davisman}, $M$ is a manifold and $(\Gamma, V)$ is a cocompact reflection system on $M$. In  \cite[Corollary 10.3]{Dav83}, Davis showed that Condition $(\star)$ is equivalent to the statement that $M$ is contractible. Since finitely generated Coxetergroups have faithful linear representations(\cite{Bou68}), they are virtually torsion-free (by Selberg's Lemma). Hence there is a torsion-free subgroup $\Gamma'$ of finite index in $\Gamma$. Since each $\Gamma$-isotropy group is finite, each $\Gamma'$-isotropy group is trivial. Hence, $\Gamma'$ acts freely on $M$ and consequently, $\Gamma'\setminus M$ is aspherical. It is closed since the index of $\Gamma'$ in $\Gamma$ is finite. The universal cover of $\Gamma'\setminus M$ is $M$. Since in dimensions $\geq 4$ we can choose $M$ to be non simply connected at infinity by Theorem \ref{refdavis-cocompact}, it follows that there exist closed aspherical manifolds which are not covered by a 
Euclidean space. Thus, the above Conjecture \ref{johnson A} is false in every dimension $\geq 4$.}
\end{remark}
\begin{remark}\rm{
\indent
\begin{itemize}
\item[\bf{1}.] Corollary \ref{davis-cocompact}  follows from Theorem \ref{davis1} and Theorem \ref{davis2} and  the fact that there exist non simply connected homology spheres in dimensions $\geq 3$.
\item[\bf{2}.] Using the reflection group trick of \cite{Dav83}, Michael W. Davis and Jean-Claude Hausmann \cite{DH89} constructed an example of a closed aspherical manifold which does not support any differentiable structure. Here are the results:
\end{itemize}}
\end{remark}
\begin{theorem}\label{davis-haus1}
For each $n\geq 13$, there exists an aspherical closed PL-manifold $M$ of dimension $n$ which does not have the homotopy type of a closed smooth manifold.  
\end{theorem}
\begin{theorem}\label{davis-haus2}
For each $n\geq 8$, there exists an aspherical closed topological manifold $M$ of dimension $n$ such that $M$ is not homeomorphic to a closed PL-manifold.
\end{theorem}
Now we will discuss the second construction of aspherical closed manifolds using Gromov's idea of hyperbolization  \cite{Gro87}:
\begin{remark}\rm{
A very important construction of aspherical manifolds comes from the hyperbolization technique due to Gromov  \cite{Gro87}. A hyperbolization technique of Gromov  \cite{Gro87} is explained in  \cite{DJ91, DFL13}: given a simplicial complex $K$, one can construct a new space $h(K)$ and a map $f: h(K)\to K$ with the following properties.
\begin{itemize}
\item[(a)] $h(K)$ is a locally CAT(0) cubical complex; in particular, it is aspherical.
\item[(b)] The inverse image in $h(K)$ of any simplex of $K$ is a hyperbolized simplex. So, the inverse image of each vertex in $K$ is a point in $h(K)$.
\item[(c)] $f : h(K)\to K$ induces a split injection on cohomology ( \cite[p. 355]{DJ91}).
\item[(d)] Hyperbolization preserves local structure: for any simplex $\sigma$ in $K$ the link of $f^{-1}(\sigma)$ is isomorphic to a subdivision of the link of $\sigma$ in $K$ (\cite[p. 356]{DJ91}). So, if $K$ is a polyhedral homology manifold, then so is $h(K)$.
\item[(e)] If $K$ is a polyhedral homology manifold, then $f : h(K)\to K$ pulls back the Stiefel-Whitney classes of $K$ to those of $h(K)$.
\end{itemize}
In  \cite{DJW01} the above version of hyperbolization is used to define a relative hyperbolization procedure (an idea also due to Gromov  \cite{Gro87}). Given $(K, \partial K)$, a triangulated manifold with boundary, form $K\cup c(\partial K)$ and then define $H(K, \partial K)$ to be the complement of an open neighborhood of the cone point in $h(K\cup c(\partial K))$. Then $H(K, \partial K)$ is a manifold with boundary; its boundary is homeomorphic to $\partial K$. The main results of  \cite{DJW01} are as follows:}
\end{remark}
\begin{theorem}
$H(K,\partial K)$ is aspherical if and only if each component of $\partial K$ is aspherical.
\end{theorem}
\begin{theorem}
The inclusion $\pi_1(\partial K)\to \pi_1(H(K,\partial K))$ is injective.
\end{theorem}
\begin{theorem}
Suppose that each component of a closed manifold $M$ is aspherical and that $M$ is the boundary of a (triangulable) manifold. Then $M$ bounds an aspherical manifold.
\end{theorem}
\begin{remark}\rm{
In  \cite{KS77} Kirby and Siebenmann showed that there are manifolds which do not admit PL structures, the possibility remained that all
manifolds could be triangulated. In  \cite{GS77} Galewski and Stern constructed a closed 5-manifold $M^5$ so that every $n$-manifold,
with $n\geq 5$, can be triangulated if and only if $M^5$ can be triangulated. Moreover, $M^5$ admits a triangulation if and only if the Rokhlin $\mu$-
invariant homomorphism, $\mu: \theta^{H}_3\to \mathbb{Z}_2$, is split. In  \cite{Man13} Manolescu showed that the $\mu$-homomorphism does not split. Consequently, there exist Galewski-Stern manifolds, $M^n$, that are not triangulable for each $n\geq 5$. In  \cite{Fre82} Freedman proved that there exists a topological 4-
manifold with even intersection form of signature 8. It followed from later work of Casson that such 4-manifolds cannot be triangulated  \cite{AM90}.
In  \cite{DJ91} Davis and Januszkiewicz applied Gromov's hyperbolization procedure to Freedman's $E_8$-manifold to show that there exist closed
aspherical 4-manifolds that cannot be triangulated. In  \cite{DFL13} Michael W. Davis, Jim Fowler and Jean-François Lafont applied hyperbolization techniques to the Galewski-Stern manifolds to show that there exist closed aspherical $n$-manifolds that cannot be triangulated for each $n\geq 6$. The question remains open in dimension 5. Here is the result:}
\end{remark}
\begin{theorem}\label{davisnottri}
For each $n\geq 6$ there is a closed aspherical manifold $M^n$ that cannot be triangulated.
\end{theorem}
\section{\large Topological Rigidity and Borel Conjecture }
Recall that if $M$ is an aspherical manifold, then $M$ is a $K(\pi,1)$-space where $\pi =\pi_1(M)$. Now among spaces having the homotopy type of a CW-complex, the $K(\pi,1)$'s are the spaces whose homotopy type is completely determined by the fundamental group alone. Thus one would suspect that if the $K(\pi,1)$-space is also a manifold, then $\pi$ might come close to determining the topology of the manifold. This leads one to perhaps the most difficult and important problem concerning aspherical manifolds:
\begin{problem}
Let $M$ and $N$ be closed aspherical manifolds with $\pi_1(M)$ isomorphic to $\pi_1(N)$. Are $M$ and $N$ homeomorphic?
\end{problem}
Since any isomorphism of the fundamental groups $\alpha: \pi_1(M)\to \pi_1(N)$ may be geometrically realized as the induced isomorphism on the fundamental
group of a homotopy equivalence $f: M\to N$, the problem may be stated perhaps more interestingly as follows:
\begin{problem}(Borel Conjecture)\label{borel}
Let $f:N\to M$ denote a homotopy equivalence between closed aspherical manifolds. Is $f$ homotopic to a homeomorphism?
\end{problem}
\begin{definition}(Topologically rigid)\rm{
We call a closed manifold $M$ topologically rigid if any homotopy equivalence $N\to M$ with a closed manifold $N$ as source is homotopic to a homeomorphism.}
\end{definition}
The Borel Conjecture is equivalent to the statement that every closed aspherical manifold is topologically rigid. 
\begin{remark}\rm{
\indent 
\begin{itemize}
\item[\bf{1.}]When $M$ is aspherical, $Out(\pi_1M)$ acts on $\mathcal{S}(M)$. The statement that $| \mathcal{S}(M)|=0$ is equivalent to Borel's conjecture \ref{borel}. While the statement that  $Out(\pi_1M)$ acts transitively on $\mathcal{S}(M)$ is equivalent to the weaker statement that any closed aspherical manifold $N$ with  $\pi_1M \simeq \pi_1N$ is homeomorphic to $M$.
\item[\bf{2.}]In particular the Borel Conjecture \ref{borel} implies because of Theorem \ref{hurewicz} that two aspherical closed manifolds are homeomorphic if and only if their fundamental groups are isomorphic.
\item[\bf{3.}] The Borel Conjecture \ref{borel} is true in dimension $\leq 2$ by the classification of closed manifolds of dimension 2. It is true in dimension 3 if Thurston’s Geometrization Conjecture is true. This follows from results of Waldhausen (see Hempel  \cite[Lemma 10.1 and Corollary 13.7]{Hem76}) and Turaev (see \cite{Tur88}) as explained for instance in \cite[Section 5]{KL09}.
\item[\bf{4.}] Borel's Conjecture \ref{borel} implies Poincar$\grave{\rm e}$'s Conjecture which says that any simply connected closed $n$-manifold $(n\geq 3)$ is homeomorphic to the unit sphere $\mathbb{S}^n$ in $\mathbb{R}^{n+1}$. This is seen as follows: Let $\Sigma^n$ be a counterexample to Poincar$\grave{\rm e}$'s Conjecture, and consider the connected sum $M=T^n\#\Sigma^n$. Van Kampen's theorem shows that $T^n$ and $M^n$ have isomorphic fundamental groups. And $M$ is seen to be aspherical by applying the Hurewicz isomorphism theorem to the universal cover of $T^n\#\Sigma^n$. Borel's Conjecture is contradicted by showing that $T^n\#\Sigma^n$ is not homeomorphic to $T^n$. For this we use the following two results:
\end{itemize}}
\end{remark}
\begin{theorem}(M. Brown, Schoenflies Theorem \cite{Bro60})\label{schoen}
 Let $f:\mathbb{S}^{n-1}\to \mathbb{S}^n$ be a bicollared embedding, then $f(\mathbb{S}^{n-1})$ bounds closed (topological) balls on both sides.
\end{theorem}
\begin{theorem} (Alexander's Trick) \label{alex}
Let $h:\mathbb{S}^n\to \mathbb{S}^n$ be any homeomorphism. Then $h$ extends to a homeomorphism $\bar{h} :\mathbb{D}^{n+1}\to \mathbb{D}^{n+1}$.
\end{theorem}
Now if $T^n\#\Sigma^n$ were homeomorphic to $T^n$, then the universal cover of $T^n\#\Sigma^n$ is homeomorphic to $\mathbb{R}^n$. Consequently, the Schoenflies theorem \ref{schoen} shows that $\Sigma^n\setminus Int(\mathbb{D}^n)$ is homeomorphic to $\mathbb{D}^n$. (This $Int(\mathbb{D}^n)$ is the interior of the 3-dimensional ball removed from $\Sigma^n$ in forming the connected sum with $T^n$.) Now applying Alexander's trick, we get $\Sigma^n$ is homeomorphic to $\mathbb{S}^n$. It follows that $T^n\#\Sigma^n$ is not homeomorphic to $T^n$.
\begin{remark}\rm{
The following smooth analogue of Borel's Conjecture \ref{borel} are both false due to Michael M. Davis and Tadeusz Januszkiewicz \cite{DJ91}  and Browder \cite{Bro65} :}
\end{remark}
\begin{question}\label{borel.smooth}
\indent
\begin{itemize}
\item[(i)] All closed aspherical manifolds support a smooth structure.
\item[(ii)] Any two closed smooth aspherical manifolds with isomorphic fundamental groups are diffeomorphic.
\end{itemize}
\end{question}
\begin{remark}\label{rem.borel}\rm{
The smooth analogue of Borel's Conjecture \ref{borel.smooth} is false as the following examples show.\\
Let $T^n$ denote the $n$-dimensional torus; i.e., $T^n=\mathbb{S}^{1}\times\mathbb{S}^{1}\times\mathbb{S}^{1}\times....\times\mathbb{S}^{1}$ ($n$-factors). Browder \cite{Bro65} constructed a smooth manifold which is homeomorphic but not diffeomorphic to $T^7$. This shows that (ii) is false. On the other hand, Michael M. Davis and J.C. Hausmann \cite{DH89} constructed in Theorem \ref{davis-haus1} an example of a closed aspherical manifold which does not support any differentiable structure proving (i) to be false as well. Moreover, Michael  M. Davis and Tadeusz Januszkiewicz \cite{DJ91} gave an example of a closed aspherical manifold which can not be triangulated (see Theorem \ref{davisnottri} and \cite{DJ91}).}
\end{remark}
One may view the Borel Conjecture as the topological version of Mostow rigidity for hyperbolic closed manifolds. The conclusion in the Borel Conjecture is weaker, one gets only homeomorphisms and not isometric diffeomorphisms, but the assumption is also weaker, since there are
many more aspherical closed topological manifolds than hyperbolic closed manifolds. The general rigidity results of Mostow \cite{Mos67, Mos73} and of Bieberbach \cite{Bie12} are as follows :
\begin{theorem}\label{bieberbachrig}[Bieberbach's Rigidity Theorem, 1912]
Let $f:N\to M$ be a homotopy equivalence between closed flat Riemannian manifolds. Then $f$ is homotopic to an affine diffeomorphism.
\end{theorem}
\begin{theorem}\label{mostow}(Mostow Rigidity Theorem)
Let $M$ and $N$ be compact, locally symmetric Riemannian manifolds with everywhere nonpositive curvature having no closed one or two dimensional geodesic subspaces which are locally direct factors. If  $f:M\to N$ is a homotopy equivalence, then $f$ is homotopic to an isometry. 
\end{theorem}
\begin{remark}\rm{
\indent
 \begin{itemize}
\item[\bf{1}.] Of particular importance to topologists is the case where $M$ and $N$ are $n$-manifolds $(n\geq 3)$ of constant negative curvature(which can be normalized to be $-1$) with isomorphic fundamental groups. Since $M$ and $N$ are both covered by the hyperbolic $n$-plane, all of their higher homotopy groups vanish. Then by a well known consequence of the classical Eilenberg obstruction theory, $M$ and $N$ are actually homotopy equivalent. So Mostow's Theorem applies, and they are isometric (by an isometry inducing the given isomorphism of fundamental groups).
\item[\bf{2}.] Mostow's Rigidity Theorem implies that atmost one differentiable manifold in a homeomorphism class can support a hyperbolic structure. 
\item[\bf{3}.] Prasad extended Mostow’s results further by replacing the assumption that the manifolds be compact, with the assumption that they have finite volume \cite{Pra73}. A result of Gromov \cite{Thu79} tells us that two homotopy equivalent hyperbolic manifolds actually have the same volume. This again has an implication for the action of $\pi_{1}$ on the sphere at infinity of hyperbolic plane, which can be used to give a proof of Mostow's Theorem.  So we have:
\end{itemize}}
\end{remark}
\begin{theorem}\label{prasad}(Mostow and Prasad Rigidity Theorem)
If $M$ and $N$ are complete hyperbolic $n$-manifolds, $n\geq 3$, with finite volume, and $f:M\to N$ is a homotopy equivalence, then $f$ is homotopic to an isometry. 
\end{theorem}
\begin{remark}\rm{
\indent
\begin{itemize}
\item[\bf{1}.] Mostow's theorem says nothing about what happens for closed orientable hyperbolic manifolds of dimension $2$. These manifolds are exactly the closed orientable surfaces of genus $g\geq 2$, which we denote $\Sigma_{g}$. Are these Mostow rigid? Or do there exist many non-isometric hyperbolic structures on $\Sigma_{g}$ ? Teichmuller theory tells us that the space of all marked hyperbolic structures on $\Sigma_{g}$ is homeomorphic to $\mathbb{R}^{6g-6}$ (\cite[Chapter 9]{FM12}). Therefore, such manifolds can be deformed and are not rigid. The whole point of Mostow rigidity is that this kind of deformations cannot happen in higher dimensions. It was stressed that Mostow's Rigidity Theorem does not hold in dimension $2$. However, for surfaces of genus $g$, the Dehn-Nielsen-Baer theorem \cite{FS08} is an analog of corollary of Mostow's Rigidity Theorem, which states that for a manifold $M$ satisfying the hypotheses of Mostow rigidity, we have $\rm{Out}(\pi_{1}(M ))= \rm{Isom}(M)$. In the current case, outer 
automorphisms do not necessarily arise from isometries, but they do arise from homeomorphisms. Here is the statement of Dehn-Nielsen-Baer theorem: For $g\geq 1$, $\rm{Top}(\Sigma_{g})/ \rm{Top}_{0}(\Sigma_{g})=\rm{Out}(\pi_{1}(\Sigma_{g}))$, where $\rm{Top}_{0}(\Sigma_{g})$ is the group of homeomorphisms isotopic to the identity map. This is a remarkable result of algebraic topology, since it relates a purely topological object $(\rm{Top}(\Sigma_{g})/ \rm{Top}_{0}(\Sigma_{g}))$ to a purely algebraic object $\rm{Out}(\pi_{1}(\Sigma_{g}))$. 
\item[\bf{2}.] If $M$ and $N$ are $2$-manifolds of finite volume, then they are homeomorphic exactly when their fundamental groups are the same. Combining this fact with Prasad's version of Mostow's theorem, we get:
\end{itemize}}
\end{remark}
\begin{theorem}\label{prasad-volum}
If $M$ and $N$ are complete hyperbolic manifolds with finite volume and isomorphic fundamental groups, then they are homeomorphic.
\end{theorem}
\begin{remark}\label{lens}\rm{
Thus, manifolds of constant negative curvature are topologically rigid. There are simple examples that show that analogous results do not hold for manifolds of constant positive curvature. For example, the lens space $L(5,1)$ and $L(5,2)$ have the same homotopy groups, but are not homotopy equivalent, while the lens space $L(7,1)$ and $L(7,2)$ are homotopy equivalent \cite{DK01} but not homeomorphic \cite[Cha74]{DK01}. Now suppose that $M$ is simply connected 4-manifold and admits no Spin structure. Then there exists another simply connected 4-manifold $N$ with the same intersection form but different Kirby Siebenmann invariant (\cite[10.1]{FQ90}). In particular $M$ and $N$ are not homeomorphic but they are oriented homotopy equivalent by \cite{Mil58}. This also shows that the answer to the following Problem \ref{fund.que} is yes for 2 dimensional manifolds and No for 3 and 4-dimensional manifolds:}
\end{remark}
\begin{problem}\label{fund.que}
Let $f:N\to M$ denote a homotopy equivalence between closed manifolds. Is $f$ homotopic to a homeomorphism?
\end{problem}
There are other large classes of 3-manifolds, however, for which topological rigidity Problem \ref{fund.que} does hold. The following result in low 
dimensional topology  is due to Waldhausen \cite{Hem76}:
\begin{theorem}(Waldhausen's Theorem) \label{waldhausen}
If $M$ and $N$ are homotopy equivalent closed prime Haken 3-manifolds, then the homotopy equivalence is homotopic to a homemorphism.
\end{theorem}
\begin{remark}\rm{
\indent
 \begin{itemize}
 \item[\bf{1.}]Turaev \cite{Tur88} has extended this result to showing that a simple homotopy equivalence between 3-manifolds with torsionfree fundamental group is homotopic to a homeomorphism provided that Thurston’s Geometrization Conjecture for irreducible 3-manifolds with infinite fundamental group and the 3-dimensional Poincar$\grave{\rm e}$ Conjecture are true. This statement remains true if one replaces simple homotopy equivalence by homotopy equivalence. This follows from the fact that the Whitehead group of the fundamental group of a 3-manifold vanishes provided that Thurston’s Geometrization Conjecture for irreducible
3-manifolds with infinite fundamental group is true \cite{KL09}. The vanishing of the Whitehead group is proved for Haken manifolds in Waldhausen \cite[Section 19]{Wal76}. In order to prove it for prime 3-manifolds it remains to treat closed hyperbolic manifolds and closed Seifert manifolds. These
cases are consequences of \cite[Theorem 2.1, pp.263 and Proposition 2.3]{FJ93a}.
\item[\bf{2.}] Using Waldhausen's Theorem \ref{waldhausen}, David Gabai has found conditions when certain homotopy equivalences could be replaced by homeomorphisms and also has shown the following result \cite{Gab94} :
\end{itemize}}
\end{remark}
\begin{theorem}\label{gabai}
Let $N$ be a closed hyperbolic 3-manifold containing an embedded hyperbolic tube of radius $\frac{\log 3}{2}$ about a closed geodesic $\frac{\log 3}{2}$. 
\begin{itemize}
\item[\rm{(i)}] If $f:M\to N$ is a homotopy equivalence where $M$ is an irreducible 3-manifold, then $f$ is  homotopic to a homemorphism.
\item[\rm{(ii)}] If $f, g :M\to N$ are homotopic homemorphism, then $f$ is isotopic to $g$.
\item[\rm{(iii)}] The space of hyperbolic metrics on $N$ is path connected.
\end{itemize}
\end{theorem}
\begin{remark}\rm{
\indent
 \begin{itemize}
\item[\bf{1.}] If $M$ is hyperbolic, then conclusion (i) follows from Mostow's rigidity Theorem \ref{prasad}. Actually Mostow's Theorem  implies that $f$ is homotopic to an isometry. If $N$ is instead Haken, then conclusions (i)-(ii) follow from Waldhausen's Theorem \ref{waldhausen}. If $N$ is Haken and hyperbolic, then conclusion (iii) follows by combining Mostow's rigidity Theorem \ref{prasad} and Waldhausen's Theorem \ref{waldhausen}. Conclusions (i), (ii)-(iii) can be viewed as a 2-fold generalization of Mostow's rigidity Theorem \ref{prasad}.
\item[\bf{2.}] Nathaniel Thurston has shown that technical conditions necessary in the proof of the above Theorem \ref{gabai} are satisfied even when the geodesic with the given tube radius does not exist \cite{GRT03}. Thus we have the following result :
\end{itemize}}
\end{remark}
\begin{theorem}\label{thurston1}
Let $N$ be a closed hyperbolic 3-manifold and $M$ irreducible. Then any homotopy equivalence $f:M\to N$ is isotopic to an isometry.
\end{theorem}
\begin{remark}\rm{
\indent
\begin{itemize}
\item[\bf{1.}] By Theorem \ref{thurston1}, we can recognize whether an irreducible 3-manifold is hyperbolic just by looking at its fundamental group.
\item[\bf{2.}] Topological rigidity problem \ref{fund.que} do hold also for some non-aspherical closed manifolds. For instance the sphere $\mathbb{S}^n$ is topologically rigid by the Poincar$\grave{\rm e}$ Conjecture. The Poincar$\grave{\rm e}$ Conjecture is known to be true in all dimensions. This follows in high dimensions from the h-cobordism theorem, in dimension four from the work of Freedman \cite{Fre82}, in dimension three from the work of Perelman as explained in \cite{KL08} and \cite{MT07} and in dimension two from the classification of surfaces. Many more examples of classes of manifolds which are topologically rigid are given and analyzed in Kreck-L$\ddot{\rm u}$ck \cite{KL09}:
\end{itemize}}
\end{remark}
\begin{definition}\rm{
 A manifold $M$ is called a strong Borel manifold if every orientation preserving homotopy equivalence $f : N\to M$ of manifolds is homotopic to a homeomorphism $h: N\to M$.}
\end{definition}
Since the connected sum of two aspherical closed manifolds is, in general, not aspherical, we have the following examples for non-aspherical closed manifolds due to M. Kreck and W. L$\ddot{\rm u}$ck \cite{KL09}:
\begin{theorem}
Let $M$ and $N$ be connected oriented closed topological manifolds of the same dimension $n\geq 5$ such that neither $\pi_1(M)$ nor $\pi_1(N)$ contains elements of order 2 or that $n\equiv 0, 3~\rm{mod}~4$. If both $M$ and $N$ are strong Borel manifolds, then the same is true for their connected sum $M\#N$.
\end{theorem}
\begin{theorem}{\rm \cite{KL09}}
Consider $k, d \in \mathbb{Z}$ with $(k, d \geq 1)$ . 
\begin{itemize}
\item[\rm {(a)}]Suppose that  $k+d\neq 3$. Then $\mathbb{S}^k\times \mathbb{S}^d$ is a strong Borel manifold if and only if both $k$ and $d$ are odd.
\item[\rm {(b)}] For $d\not=2$ the manifolds $\mathbb{S}^1\times \mathbb{S}^d$ is strongly Borel; and $\mathbb{S}^1\times \mathbb{S}^2$ is strongly Borel if and only if the 3-dimensional Poincar$\grave{\rm e}$ Conjecture is true.
\end{itemize}
\end{theorem}
\begin{remark}\rm{
The Borel Conjecture \ref{borel} has the following (slightly weaker when $n\neq 3$) group theoretic interpretation in which $Top(\mathbb{R}^n)$ denotes the group of all self-homeomorphisms of $\mathbb{R}^n$ equipped with the compact open topology. Here are the results :}
\end{remark}
Let $E(n)$ be the group of rigid motions of the $n$-dimensional Euclidean space and $A(n)$ be the group of affine motions of Euclidean $n$-space. Bieberbach proved the following result in 1912. 
\begin{theorem}\label{bieberbach}(Bieberbach)
Let $\Gamma$ and $\Gamma^{'}$ be two torsion-free uniform discrete subgroups of $E(n)$. If $\Gamma$ and $\Gamma^{'}$ are isomorphic, then they are conjugate inside of $A(n)$.
\end{theorem}
Borel posed in 1966 the following question whether one can allow $\Gamma^{'}$ to sit inside the larger group $\rm{Top}(\mathbb{R}^n)$, but require the induced action of $\Gamma^{'}$ on $\mathbb{R}^n$ to be free and properly discontinuous.
\begin{problem}\label{borelgroup}(Topological Strong Rigidity Conjecture)
If $\Gamma$ and $\Gamma^{'}$  are isomorphic, is $\Gamma$  conjugate to $\Gamma^{'}$ inside of $\rm{Top}(\mathbb{R}^n)$?
\end{problem}
\begin{remark}\rm{
F.T. Farrell and W.C. Hsiang \cite{FH78} gave an affirmative answer to Problem \ref{borelgroup} when $n>4$ and $\Gamma$  has odd order holonomy group.(The holonomy group of $\Gamma$ is its image in $O(n)$ the orthogonal group under the canonical projection $E(n)\to O(n)$; Bieberbach (1910) showed this group is always finite.) Here is the result:  }
\end{remark}
\begin{theorem}\label{farall-hsiang}
 Let $M$ denote  a closed flat Riemannian $n$-manifold with fundamental group $\Gamma$ with holonomy group $G$.
If $m + n > 4$ and $|G|$ is odd, then $\mathcal{S}(M\times \mathbb{D}^{m})=0$.
\end{theorem}
\begin{theorem}\label{farall-hsiang A}
Let $M$ be  a closed flat Riemannian $n$-manifold, $N$ be a topological manifold and $f : N \to M$ be a homotopy equivalence.
If $n > 4$ and the holonomy group of $M$ has odd order, then $f$ is homotopic to a homeomorphism.
\end{theorem}
\begin{remark}\rm{
Theorem \ref{farall-hsiang A} is a special case of Theorem \ref{farall-hsiang} when $m= 0$. It has another equivalent formulation as follows:}
\end{remark}
\begin{theorem}\label{farall-hsiang A'} 
Let $M$ be a closed connected $n$-manifold $(n > 4)$. It has flat Riemannian structure with odd order holonomy group if and only if $\pi_{i}(M^n)= 0$ for $i> 1$ and $\pi_{1}(M^n)$ contains an abelian subgroup with odd (finite) index.
\end{theorem}
\begin{remark}\label{fhrig}\rm{
F.T. Farrell and W.C. Hsiang \cite{FH83} verified Borel Conjecture \ref{borel} for aspherical manifolds (of dimensions greater than $4$) whose fundamental groups contain nilpotent subgroups of finite index. In particular, Borel Conjecture \ref{borel} is true for (high dimensional) closed flat Riemannian manifolds. Since fundamental groups of flat Riemannian manifolds are virtually abelian; i.e., contain an abelian subgroup with finite index.  Here are the results:}
\end{remark}
\begin{theorem}\label{farall-hsiang 80}{\rm \cite{FH83}}
Let $M$ be a closed aspherical manifold whose fundamental group is virtually nilpotent and let $E^{m+n}$ be the total space of a
$\mathbb{D}^{m}$-bundle whose base space is $M^n$, $(m + n > 4)$, then $\mathcal{S}(E^{m+n})=0$; in particular, $\mathcal{S}(M^{n})=0$ when $n > 4$.
\end{theorem}
Theorem \ref{farall-hsiang 80} also has the following immediate consequence.
\begin{theorem}\label{farall-hsiang 80ii}
Let $N^n$ $(n \neq 3, 4)$ be a closed connected infranilmanifold and $M^n$ be an aspherical manifold with $\pi_{1}(M^n)$ isomorphic to $\pi_{1}(N^n)$, then $N^n$ and $M^n$ are homeomorphic.
\end{theorem}
\begin{remark}\rm{
\indent
If $N^n$ is a nilmanifold, this result was proven by Wall \cite{Wal71}; and if $N^n$ is the $n$-torus, the result was proven earlier yet by Wall \cite{Wal69}, and Hsiang and Shaneson \cite{HS70}.}
\end{remark}
F. T. Farrell and L. E. Jones \cite{FJ88} proved Borel's Conjecture \ref{borel} for every closed aspherical manifold of $\dim \neq 3$ whose fundamental group is virtually poly-$\mathbb{Z}$. Here are the results:
\begin{theorem}\label{virtualpoly}\cite{FJ88}
Let $M$ be a closed aspherical manifold whose fundamental group is virtually poly-$\mathbb{Z}$ and let $E^{m+n}$ be the total space of an
$\mathbb{D}^{m}$-bundle whose base space is $M^n$, $(m + n > 4)$, then $\mathcal{S}(E^{m+n})=0$; in particular, $\mathcal{S}(M^{n})=0$ when $n > 4$.
\end{theorem}
A more geometric consequence of Theorem \ref{virtualpoly} is the following result:
\begin{theorem}\label{rig.virtu}{\rm \cite{FJ88}}
Let $N^n$ and $M^n$ be two closed connected aspherical manifolds with isomorphic fundamental groups. Suppose $\pi_1(N)$ is virtually poly-$\mathbb{Z}$, then $N$ and $M$ are homeomorphic provided $n\neq 3, 4$.
\end{theorem}
\begin{remark}\rm{
\indent
\begin{itemize}
\item[\bf{1.}] The work of Freedman and Quinn \cite{FQ90} together with Theorem \ref{virtualpoly} should imply that $N^n$ and $M^n$ are homeomorphic even when $n = 4$.
\item[\bf{2.}] We also recall a conjecture of Milnor \cite{Mil77}, viz., that the class of fundamental groups of compact complete affine flat manifolds coincides with the class of torsion-free virtually poly-$\mathbb{Z}$ groups. (The original conjecture was without the compactness assumption, but Margulis \cite{Mar83} has given a counterexample to this more general conjecture.) Some positive evidence for this conjecture is in \cite{Boy89} and \cite{Mil77}. We can relate the conjecture to Theorem \ref{rig.virtu}. Namely, if the conjecture were true, then complete compact flat affine manifolds would be topologically characterized (in dimensions $\neq 3, 4$) as those closed manifolds $M$ such that $\pi_1(M)$ is virtually poly-$\mathbb{Z}$ and $\pi_{i}(M)=0$ for $i\neq 1$. 
\end{itemize}}
\end{remark}
\begin{theorem}\label{hype.regid}{\rm \cite{FJ89}}
Let $M^n$ $(n \neq 3, 4)$ be a complete (connected) real hyperbolic manifold.
If $m$ is any nonnegative integer larger than $4\text{-}n$, then the structure set $\mathcal{S}(M\times \mathbb{R}^m)$ contains only one element, and this element is represented by the identity map of $M^n\times \mathbb{R}^m$.
\end{theorem}
Specializing Theorem \ref{hype.regid} by assuming $M$ is compact and $m = 0$ yields the following result:
\begin{theorem}\label{hyptop-rig}{\rm \cite{FJ89}}
Let $M^n$ be a closed (connected) real hyperbolic manifold and $N$ be a  closed topological aspherical manifold such that $N$ and $M$ have isomorphic fundamental groups. Then $N$ and $M$ are homeomorphic provided the dimension of $M$ differs from $3$ and $4$.
\end{theorem}
\begin{remark}\rm{
Theorem \ref{hyptop-rig} together with Mostow's rigidity Theorem \ref{mostow} yields a characterization of hyperbolic structures on compact manifolds whose dimension is greater than 4; i.e., Theorem \ref{hyptop-rig} is an existence theorem for hyperbolic structures while Mostow's work is the uniqueness theorem.}
\end{remark}
\begin{corollary}\label{hyp-stru}{\rm \cite{FJ89}}
A closed (connected) topological manifold $M$ of dimension $n\neq 3$ and $4$ has a (real) hyperbolic structure  if and only if
\begin{itemize}
\item[(i)] $M$ is aspherical and
\item[(ii)] the fundamental group of $M$ is isomorphic to a discrete cocompact subgroup of the Lie group $O(n, 1)$.
\end{itemize}
Furthermore, by Mostow's rigidity Theorem \ref{mostow}, the structure is unique (up to isometry) provided $n > 2$.
\end{corollary}
More generally, if $M$ has finite volume but perhaps is not compact, we have the following generalization of Theorem \ref{hyptop-rig}.
\begin{corollary}\label{prop-regid}{\rm \cite{FJ89}}
Let $M^n$ be a complete (connected) real hyperbolic manifold. Suppose that  $M^n$  has finite volume and $n\neq 3, 4$, and $5$.
Let $N$ be any topological manifold that is properly homotopically equivalent to $M$, then $N$ and $M$ are homeomorphic. In fact, any proper homotopy equivalence is properly homotopic to a homeomorphism.
\end{corollary}
\begin{definition}\rm{
A smooth map $k : N\to M$ between Riemannian manifolds is harmonic if it is a critical point of the energy functional $\mathcal{E}(k) = \int_{N}\frac{1}{2}|dk|^2$ . An equivalent definition is that the tension field $\tau_k$ of $k$ vanishes everywhere. (The tension
field $\tau_k$ is a section of the bundle $k^{*}TM$ and can be defined in the following way: for $x\in N$ choose an orthonormal basis $\{v_i \}$ of $T_{x}(N)$ and define $\tau_k(x)=\sum w_i$, where $w_i$ is the acceleration vector, at $t = 0$, of $k(\gamma_i)$, and $\gamma_i$ is the geodesic with $\gamma_i(0)=x$ and $\frac{d}{dt}\gamma_i(0)= v_i.)$ Given a map $f: N\to M$ between Riemannian manifolds, we can try to associate to it a harmonic map that is the limit $k=\lim_{t\to \infty} k_t$, where $k_t$ is the unique solution of the heat flow equation, that is, the PDE initial value problem $\frac{\partial k_t}{\partial t}=\tau(k_t)$, $k_0 = f$. If this limit $k$ exists then it is homotopic to $f$ (the homotopy is $t\to k_t$).}
\end{definition}
\begin{remark}\label{harm.Eells}\rm{
Let $N$ and $M$ denote two closed connected Riemannian manifolds which have nonpositive sectional curvature values and whose fundamental groups are isomorphic. Since both $N$ and $M$ are $K(\pi,1)$-spaces, it follows that they must be homotopy equivalent to one another. Now, it follows from the classical result of Eells and Sampson \cite{ES64} that if $f : N\to M$ is a smooth homotopy equivalence between closed negatively curved manifolds the heat flow equation beginning at $f$ converges to a well defined harmonic map $k=\lim_{t\to \infty} k_t$. Moreover, from the results of Hartman \cite{Har67} and Al’ber \cite{Alb68} it follows that $f$ is homotopic to a unique harmonic map. Therefore the homotopy equivalence $f$ in Problem \ref{fund.que} homotopic to unique harmonic maps. A problem with some history behind it is to determine whether or not $N$ and $M$ must be homeomorphic to one another. Cheeger showed in the mid-1970s that the bundles of orthonormal two-frames $V_2(N)$, $V_2(M)$ are homeomorphic provided $M$ and $N$ 
are negatively curved manifolds; and then, under the same hypothesis, Gromov showed that the unit sphere bundles $S(N)$, $S(M)$ are homeomorphic, via a homeomorphism which preserves the orbits of the geodesic flows. Mishchenko \cite{Mis74} showed that the homotopy equivalence $f:N\to M$ pulls the rational Pointrjagin classes of $M$ back to those of $N$; and Farrell and Hsiang \cite{FH81} showed in 1979 that $N\times \mathbb{R}^3$ and $M\times \mathbb{R}^3$ are homeomorphic. Here is the result :}
\end{remark}
\begin{theorem}\label{FH-PROD}
Let $M^n$ be a closed non-positively curved manifold and let $g: N^n\to M^n$ be a homotopy equivalence where $N^n$ is a manifold. Then $g\times id: N^n\times \mathbb{R}^3\to M\times \mathbb{R}^3$ is homotopic to a homeomorphism.
\end{theorem}
\begin{remark}\rm{
It turns out that this theorem is very closely related to the so-called "Novikov's Conjecture". So, let us begin with this conjecture:\\
Let $M$ be a closed oriented manifold. Given a homomorphism $\pi_{1}M\to \pi$, we have a natural map $f:M\to B\pi=K(\pi,1)$. Let 
$$\mathcal{L}(M)\in \bigoplus_{k\in \mathbb{Z},k\geq0}H^{4k}(M,\mathbb{Q})$$ be the total $\mathcal{L}$-genus of $M$.  Its $k$-th entry
$\mathcal{L}_k(M)\in H^{4k}(M;\mathbb{Q})$ is a certain homogeneous polynomial of degree $k$ in the rational Pontrjagin classes $p_i(M;\mathbb{Q})\in H^{4i}(M;\mathbb{Q})$ for $i = 1$, $2$, . . . , $k$ such that the coefficient $s_k$ of the monomial $p_k(M;\mathbb{Q})$ is different from zero. The $\mathcal{L}$-genus $\mathcal{L}(M)$ is determined by all the rational Pontrjagin classes and vice
versa. The $\mathcal{L}$-genus depends on the tangent bundle and thus on the differentiable structure of $M$. For $x\in \prod_{k\geq0}H^{k}(B\pi;\mathbb{Q})$ define the higher signature of $M$ associated to $x$ and $f$ to be the number
\begin{center}
$sign_x(M,f)=\left\langle \mathcal{L}(M)\cup f^{*}(x),[M]\right\rangle \in \mathbb{Q}$.
\end{center} 
We say that $sign_x$ for $x\in H^*(B\pi;\mathbb{Q})$ is homotopy invariant if for two closed oriented smooth manifolds $M$ and $N$ with corresponding maps $f: M\to B\pi$ and $g: N\to B\pi$ we have
$$sign_x(M, f) = sign_x(N, g),$$ whenever there is an orientation preserving homotopy equivalence $h : M\to N$ such
that $g\circ h$ and $f$ are homotopic. If $x=1\in H^0(B\pi)$, then the higher signature $sign_x(M,f)$ is by the Hirzebruch signature formula (see \cite{Hir58, Hir71}) the signature of $M$ itself and hence an invariant of the oriented homotopy type.  Several years ago, Novikov made the following conjecture :}
\end{remark}
\begin{conjecture}(Novikov Conjecture)\label{Novikov}
For every group $\pi$ and each element $$x\in \prod_{k\in \mathbb{Z},k\geq0}H^{k}(B\pi;\mathbb{Q}),$$ the number $sign_x()$ is a homotopy invariant. 
\end{conjecture}
\begin{remark}\rm{
Let the map $f : M\to N$ be a homotopy equivalence of aspherical closed oriented manifolds. Then the Novikov Conjecture \ref{Novikov} implies that $f^*(p_i(N;\mathbb{Q}))=p_i(M;\mathbb{Q})$. This is certainly true if $f$ is a diffeomorphism. On the other hand, in general the rational Pontrjagin classes are not homotopy invariants and the integral Pontrjagin classes $p_k(M;\mathbb{Q})$ are not homeomorphism invariants (see for instance \cite[Example 1.6 and Theorem 4.8]{KL05}). This seems to shed doubts about the Novikov Conjecture. However, if the Borel Conjecture \ref{borel} is true, the map $f : M\to N$ is homotopic to a homeomorphism and the conclusion $f^*(p_i(N;\mathbb{Q}))=p_i(M;\mathbb{Q})$ does follow from the following deep result due to Novikov \cite{Nov65, Nov65a, Nov66}.}
\end{remark}
\begin{theorem}(Novikov 1966)
 If $f:M\to N$ is a homeomorphism between smooth manifolds, then $f^*(p_i(N;\mathbb{Q}))=p_i(M;\mathbb{Q})$.
\end{theorem}
For a fixed group $\pi$ but all classes $x\in H^{*}(B\pi,\mathbb{Q})$, we call the restricted conjecture by Novikov Conjecture for the group $(\pi)$. Since then, Novikov Conjecture for the group $(\pi)$ has been verified for various $\pi$ (\cite{Cap76, FH73, FH78, FH79, FH83}).
If there exists a closed aspherical manifold $M$ with fundamental group $\pi$ (i.e., a closed manifold $K(\pi,1))$, then Novikov Conjecture for  the group $(\pi)$ (for all $N$) reduces to the following two equivalent forms:
\begin{conjecture}\label{Novikov1}
Let $M^n$ be a closed aspherical manifold (with fundamental group $\pi$).
\begin{itemize}
\item[(i)] If $g:N^{n+m}\to M^n\times \mathbb{D}^m (m\geq 0)$ is a homotopy equivalence between manifolds which restricts to a homeomorphism from $\partial N^{n+m}\to M^n\times \partial \mathbb{D}^m$, then $g^{*}p_{m}(M^n;\mathbb{Q})=p_{m}(N^n;\mathbb{Q})$ (for all $m$).
\item[(ii)] The rationalized surgery map \\
$\bar{\sigma}:[M^{n} \times \mathbb{D}^m, \partial;G/Top]\otimes \mathbb{Q}\to L^{s}_{n+m}(\pi_{1}M^n, w_{1}(M^n))\otimes \mathbb{Q}$ $(n+m>4)$ is a monomorphism.
\end{itemize}
\end{conjecture}
F. T. Farrell and W. C. Hsiang \cite{FH81} strengthened Conjecture \ref{Novikov1} to the following form.
\begin{conjecture}\label{Novikov2}
If $M^n$ is a closed aspherical manifold, then the surgery map $\bar{\sigma}:[M^{n} \times \mathbb{D}^i, \partial;G/Top]\to L^{s}_{n+i}(\pi_{1}M^n,w_{1}(M^n))$ $(n+i>4)$  is a split monomorphism.
\end{conjecture}
\begin{remark}\rm{
\indent
\begin{itemize}
\item[\bf{1.}] Conjecture \ref{Novikov2} implies that $g:N^n\to M^n$ is a simple homotopy equivalence, then $g\times id:N^n\times \mathbb{R}^3\to M^n\times \mathbb{R}^3$ is homotopic to a homeomorphism. This statement can not be gotten from Conjecture \ref{Novikov1}. Miscenko [Mis74] verified Conjecture \ref{Novikov1} for $M^n$ a closed non-positively curved manifold via elliptic operators. 
\item[\bf{2.}] F. T. Farrell and W. C. Hsiang \cite{FH81} verified Conjecture \ref{Novikov2} for a class of aspherical manifolds including all closed non positively curved $M^n$. For this purpose, F. T. Farrell and W. C. Hsiang considered an aspherical manifold $M^n$ satisfying the following two conditions:
\item[\bf{*}] The universal cover $\widetilde{M}^n$ of $M^n$ has a compactification $\overline{M}^n=\mathbb{D}^n$ such that the covering transformations extend to an action of $\pi_{1}M^n$ on $\mathbb{D}^n$ (not necessarily free on $\partial \mathbb{D}^n)$.
\item[\bf{**}] Any homotopy $h:M^n\times[0,1]\to M^n$ with $h(x,0)=x$ (for all $x\in M^n$) lifts to a homotopy $\overline{h}:\overline{M}^n\times[0,1]\to \overline{M}^n$ with $\overline{h}(x,t)=y$ if either $t=0$ or $y\in \partial \mathbb{D}^n$ (and $p(\overline{h}(y,t))=h(p(y),t)$ for all $y\in \widetilde{M}^n$, $t\in[0,1]$ where $p:\widetilde{M}^n\to M^n$ is the covering projection).
\end{itemize}}
\end{remark}
\begin{theorem}\label{sur.map}
Let $M^n$ be a closed (triangulable) aspherical manifold satisfying $(*)$ and $(**)$. Then the surgery map\\ $\bar{\sigma}:[M^n\times \mathbb{D}^m, \partial;G/Top] \to L^s_{n+m}(\pi_{1}M^n,w_{1}(M^n)) (n+m>4)$ is a split monomorphism. In particular, if $M^n$ is a closed non-positively curved manifold, then $(*)$ and $(**)$ are satisfied and Conjecture \ref{Novikov2} is valid for $M^n$.
\end{theorem}
\begin{remark}\rm{
Wall \cite[pp. 263-267]{Wal71} expanding on ideas of Novikov \cite{Nov70}, gives the following relationship between Novikov's conjecture and the surgery map.}
\end{remark}
\begin{theorem}\label{Nov.for}
Let $M^n$ be a compact, orientable, aspherical manifold with $\pi_{1}M^n=\pi$. Then, Novikov's Conjecture for $(\pi)$ is true if and only if the (rational) surgery maps $\bar{\sigma}_{m}:[M^n\times \mathbb{D}^m, \partial;G/Top]\otimes \mathbb{Q}\to L^s_{n+m}(\pi_{1}M^n)\otimes \mathbb{Q}$ are monomorphisms for all integer $m$ satisfying both $m\geq 2$ and $n+m\geq 7$.
\end{theorem}
\begin{remark}\rm{
\indent
 \begin{itemize}
\item[\bf{1.}] Hence, Theorem \ref{sur.map} implies that Novikov's Conjecture for $(\pi)$ is true when $\pi=\pi_{1}M^n$ and $M^n$ is a closed (connected) non-positively curved Riemannian manifold. However. this result was proven much earlier and via a different technique in Miscenko's seminal paper \cite{Mis74}.
\item[\bf{2.}] Although much work has been done verifying Novikov's Conjecture for a very large class of groups $\pi$, it remains open and is still an active area of research. See Kasparov's paper \cite{Kas88} for a description of the state of the conjecture as of 1988. Additional important work on it has been done since that date.
\item[\bf{3.}] F. T. Farrell and W. C. Hsiang \cite{FH81} proved Theorem \ref{sur.map} by using Theorem \ref{Nov.for} and a well known result of Wall known as the $\pi\text{-}\pi$ theorem. That states that in higher dimensions a normal map of a manifold with boundary to a simple Poincar$\grave{\rm e}$ pair with $\pi_{1}(X)\cong \pi_{1}(Y)$ is normally bordant to a simple homotopy equivalence of pairs.
\end{itemize}}
\end{remark}
Of course homeomorphism implies homotopy equivalence and the converse is, in general, not true. But for closed negatively curved manifolds (dimensions $\neq 3, 4)$ F.T. Farrell and L.E. Jones \cite{FJ91} proved that these two conditions are really equivalent. In fact they proved that this is true when just one of the manifolds is non-positively curved. Here are the result:
\begin{theorem}(Farrell and Jones Topological Rigidity Theorem)\label{gentoprigd}\cite{FJ93}
 Let $M^m$ be a closed non-positively curved Riemannian manifold. Then $|\mathcal{S}(M^m\times \mathbb{D}^{n},\partial)|=1$ when $m+n\geq 5$.
\end{theorem}
For proving Theorem \ref{gentoprigd}, F.T. Farrell and L.E. Jones used Theorem \ref{sur.map} and  the following result \cite{FJ91}: 
\begin{theorem}(Vanishing Theorem)\label{whigro}
 Let $M$ be a closed (connected) non-positively curved Riemannian manifold. Then $Wh(\pi_1(M))=0$.
\end{theorem}
\begin{remark}\rm{
The special cases of Theorem \ref{whigro} when $M$ is the $n$-torus $T^n$ was proven by Bass-Heller-Swan \cite{BHS64} and for Riemannian flat or real hyperbolic were proven earlier by Farrell and Hsiang in \cite{FH78} and by Farrell and Jones in \cite{FJ86}, respectively. Farrell and Hsiang also showed in \cite{FH81a} that $Wh(\pi_1(M))=0$ when $M$ is a closed infrasolvmanifold. }
\end{remark}
\begin{definition}\rm{
A Riemannian manifold $M$ is $A$-regular if there exists a sequence of positive number $A_0$, $A_1$, $A_2$, $A_3$,... with $|D^n(K)|\leq A_n$. Here $K$ is the curvature tensor and $D$ is covariant differentiation.}
\end{definition}
\begin{remark}\rm{
\indent
\begin{itemize}
\item[\bf{1.}] Every closed Riemannian manifold and locally symmetric space is $A$-regular. 
\item[\bf{2.}] F.T. Farrell and L.E. Jones \cite{FJ98} proved the following generalization of the Vanishing Theorem \ref{whigro} to the case where $M$ is complete but not necessarily compact :
\end{itemize}}
\end{remark}
\begin{theorem}\label{Awhigro}
 Let $M$ be any complete Riemannian manifold which is both non-positively curved and $A$-regular. Then $Wh(\pi_1(M))=0$.
\end{theorem}
F.T. Farrell and L.E. Jones have proven the following Theorem \ref{farall top-rig} which more or less settles Problem \ref{fund.que} for non-positively curved manifolds.
\begin{theorem}\label{farall top-rig}(Topological Rigidity Theorem)
Let $M^n$  and $N^n$ be a pair of closed aspherical manifolds.
Then any isomorphism from $\pi_{1}(M)$ to $\pi_{1}(N)$ is induced (up to conjugacy) by a homeomorphism from $M$  to $N$ provided $M$ is nonpositively curved with $\dim M\neq 3,4$. 
\end{theorem}
\begin{proof}
This result is classical when $n=1$ or $2$. When $n\geq 5$ set $m=0$ in Theorem \ref{gentoprigd} to conclude that $M$  and $N$ are h-cobordant and hence homeomorphic by the s-cobordism since $Wh(\pi_1(M))=0$ because of the Vanishing Theorem \ref{whigro}.
\end{proof}
\begin{remark}\rm{
\indent
\begin{itemize}
\item[\bf{1}.] This result is an analogue of Mostow's Rigidity Theorem \ref{mostow} and proves Borel's Conjecture \ref{borel} for closed non-positively curved manifolds $(dim \neq 3, 4)$. 
\item[\bf{2}.] The special cases of Theorem \ref{farall top-rig} when $M$ is Riemannian flat or real hyperbolic were proven earlier by Farrell and Hsiang in \cite{FH83} and by Farrell and Jones in \cite{FJ89}, respectively (see Theorem \ref{farall-hsiang 80} and Theorem \ref{hyptop-rig}). 
\item[\bf{3.}] The conclusion of Theorem \ref{farall top-rig} is also true when $M^n$ is a closed infrasolvmanifold. This was proven in Theorem \ref{rig.virtu}. Yau showed in \cite{Yau71} that a closed infrasolvmanifold $M^m$ supports a non-positively curved Riemannian metric only when $\pi_{1}(M)$ is virtually abelian ; hence, neither class of manifolds contains the other.
\item[\bf{4.}] In further work, Farrell-Jones \cite{FJ98} extended Theorem \ref{farall top-rig} to cover compact complete affine flat manifolds of dimension $\geq 5$. This is done by considering complete non-positively curved manifolds that are not necessary compact. Note that the universal cover is in these cases always homeomorphic to Euclidean space.
\item[\bf{5}.] Theorem \ref{farall top-rig} yields a topological characterization of closed locally symmetric spaces of non-compact type. Here is the results:
\end{itemize}}
\end{remark}
\begin{corollary}\label{farall-asp}{\rm \cite{FJ91}}
A closed topological manifold $M$ of dimension $n\neq 3$ and $4$ supports the structure of a locally symmetric space of non compact type  if and only if
\begin{itemize}
\item[(i)] $M$ is aspherical and
\item[(ii)] the fundamental group of $M$ is isomorphic to a discrete cocompact subgroup of a (virtually connected) linear semisimple Lie group.
\end{itemize}
\end{corollary}
\begin{remark}\rm{
\indent
\begin{itemize}
\item[\bf{1}.] It would be interesting to know if Corollary \ref{farall-asp} or Corollary \ref{hyp-stru} is true in dimension $3$ and $4$. The corresponding result for flat, almost flat and infrasolvmanifolds \cite{FH83} is true when $\dim M=4$. This uses the work of Freedman and Quinn \cite{FQ90} showing the topological surgery theory is valid for four dimensional manifolds with virtually poly-$\mathbb{Z}$ fundamental groups. But it is unknown whether topological surgery theory works for four manifolds whose fundamental groups are isomorphic to discrete, torsion-free, cocompact subgroups of $O(4,1,\mathbb{R})$. If it does, then Corollary \ref{hyp-stru} would still be true when $\dim=4$.
\item[\bf{2}.] A similar comment can be made about Corollary  \ref{farall-asp}. It is an immediate consequence of results of Kneser \cite{Kne29} and Milnor \cite{Mil62} that the truth of Corollary \ref{hyp-stru} when $n = 3$ (which is equivalent to Corollary \ref{farall-asp} in this case) or of the corresponding topological characterization of three dimensional compact flat, infranil, or infrasolvmanifolds  would imply the truth of the Poincar$\grave{\rm e}$ Conjecture. On the other hand, Thurston has conjectured an even stronger characterization of compact hyperbolic three manifolds. The following is one case of his geometrization conjecture \cite{Thu82} :
\end{itemize}}
\end{remark}
\begin{theorem}\label{thurston}(Thurston's Conjecture)
Let $M$ be a closed three dimensional manifold. Then $M$ has a hyperbolic structure iff
\begin{itemize}
\item[(i)] $M$ is aspherical and
\item[(ii)] every abelian subgroup of $\pi_{1}(M)$ is cyclic.
\end{itemize}
\end{theorem}
\begin{remark}\rm{
\indent
\begin{itemize}
\item[\bf{1.}] In particular, Thurston's Conjecture would imply that any closed three dimensional Riemannian manifold with negative sectional curvature is homeomorphic to a hyperbolic manifold.
\item[\bf{2.}] In further work, Farrell-Jones \cite{FJ98} extended Theorem \ref{farall top-rig} to cover compact complete affine flat manifolds of dimension $\geq 5$. This is done by considering complete non-positively curved manifolds that are not necessary compact. Note that the universal cover is in these cases always homeomorphic to Euclidean space. Here are the results :
\end{itemize}}
\end{remark}
\begin{definition}\rm{
Let $M$ be a manifold with non-empty boundary. We say that $M$ is topologically rigid if it has the following property: Let $h:(N,\partial N)\to (M,\partial M)$ be any proper homotopy equivalence where $N$ is another manifold. Suppose there exists a compact subset $C\subset N$ such that the restriction of $h$ to $\partial N\cup (N\setminus C)$ is a homeomorphism. Then there exists a proper homotopy $h_t:(N,\partial N)\to (M,\partial M)$ from $h$ to a homeomorphism and a perhaps larger compact subset $K$ of $N$ such that the restriction of $h_t$ and $h$ 
to $\partial N\cup (N\setminus K)$ agree for all $t\in [0,1]$. (When $M$ and $N$ are closed, this just says that a homotopy equivalence $h:N\to M$ is homotopic to a homeomorphism.)}
\end{definition}
\begin{theorem}\label{Atrtaffine}
Let $M^n$ be an aspherical Riemannian manifold with $n\geq 5$ (it can be non-compact and can have non-empty boundary). Suppose $\pi_1(M)$ is isomorphic to the fundamental group of an $A$-regular complete non-positively curved Riemannian manifold (This happens for example when $\pi_1(M)$ is isomorphic to a torsion-free discrete subgroup of $GL_{n}(\mathbb{R})$). Then $M$ is topologically rigid. In particular, every $A$-regular complete non-positively curved Riemannian manifold of $\dim \geq 5$ is topologically rigid.
\end{theorem}
\begin{theorem}\label{trtaffine}
Let $M$  and $N$ be a pair of closed affine flat manifolds. Then any isomorphism from $\pi_{1}(M)$ to $\pi_{1}(N)$ is induced by a homeomorphism from $M$  to $N$.
\end{theorem}
\begin{remark}\rm{
\indent
\begin{itemize}
\item[\bf{1.}] Theorem \ref{trtaffine} is an affine analogue of the classical Bieberbach Rigidity Theorem \ref{bieberbachrig}. Note that Theorem \ref{trtaffine} (when $\dim(M)\geq 5$) does not follow from Topological Rigidity Theorem \ref{gentoprigd} since there are closed affine flat manifolds $M$ which cannot support a Riemannian metric of non-positive curvature. For example $M^3=\mathbb{R}^3/\Gamma$ where $\Gamma$ is the group generated by the three affine motions $\alpha$, $\beta$, $\gamma$ of $\mathbb{R}^3$ with 
\begin{equation*}
\begin{split}
 \alpha(x,y,z)&=(x+1,y,z)\\
 \beta(x,y,z)&=(x,y+1,z)\\
 \gamma(x,y,z)&=(x+y,2x+3y,z+1).
 \end{split}
\end{equation*}
Since $\Gamma$ is solvable but not virtually abelian. However Gromoll and Wolf \cite{GW71} and Yau \cite{Yau71} independently proved that if $M$ is a closed non-positively curved Riemannian manifold and $\pi_1(M)$ is solvable, then  $\pi_1(M)$ is virtually abelian. This shows that $M^3$ cannot support a Riemannian metric of non-positive curvature. But Theorem \ref{trtaffine} (when $\dim(M)\geq 5$) does follow from Theorem \ref{Atrtaffine} since $M$ is aspherical and $\pi_1(M)$ is a discrete subgroup of Aff$(\mathbb{R}^n)$ which is closed subgroup of $GL_{n+1}(\mathbb{R})$.
\item[\bf{2.}] Theorem \ref{trtaffine} is a classical result when $\dim(M)\geq 2$. And, when $\dim(M)=3$, Theorem \ref{trtaffine} was proven by D. Fried and W.M. Goldman in \cite{FG83}. Hence it remains to discuss the case when $\dim(M)=4$. In this case (in fact more generally when $\dim(M)\geq 6$) H. Abels, G.A. Margulis and G.A. Soifer \cite{AMS97} proved that $\pi_1(M)$ is virtually polycyclic. And hence Theorem \ref{trtaffine} follows from Theorem \ref{rig.virtu} when $\dim(M)=4$ (see also Remark \ref{rig.virtu}). A key ingredient in Theorem \ref{rig.virtu} is that M. Freedman and F. Quinn \cite{FQ90} have shown that topological surgery works in dimension 4 for manifolds with virtually poly-cyclic fundamental groups.
\end{itemize}}
\end{remark}
\section{\large The Farrell-Jones Conjecture}
In this section we will discuss the Farrell-Jones Conjecture. Why is the Farrell-Jones Conjecture so important? One reason is that it plays
an important role in the classification and geometry of manifolds. A second reason is that it implies a variety of well-known conjectures, such as the ones due to Borel and Novikov, and also the conjecture for the vanishing of Whithead group.

\begin{definition}\rm{
Let $G$ be any group. A family $\mathcal{F}$ of subgroups of $G$ is a set of subgroups of $G$ which is closed under taking subgroups and conjugations.}
\end{definition}
\begin{example}\rm{
Examples for $\mathcal{F}$ are :
\begin{align*}
\mathcal{F}_{\mathcal{TR}}&= \{trivial~~ subgroup\};\\
\mathcal{F}_\mathcal{FIN} &= \{finite~~ subgroups\};\\
\mathcal{F}_\mathcal{VCYC} &= \{virtually~~ cyclic~~ subgroups\};\\
\mathcal{F}_\mathcal{COM} &= \{compact~~ subgroups\};\\
\mathcal{F}_\mathcal{COMOP} &= \{compact~~ open~~ subgroups\};\\
\mathcal{F}_\mathcal{ALL} &= \{all ~~subgroups\}.
\end{align*}}
\end{example}
\begin{definition}(Classifying $G$-CW-complex for a family of subgroups)\rm{
Let $\mathcal{F}$ be a family of subgroups of $G$. A model $E_{\mathcal{F}}(G)$ for the classifying $G$-CW-complex for the family $\mathcal{F}$ of subgroups is a $G$-CW-complex $E_{\mathcal{F}}(G)$ which has the following properties:
\begin{itemize}
\item[(i)] All isotropy groups of $E_{\mathcal{F}}(G)$ belong to $\mathcal{F}$.
\item[(ii)] For any $G$-CW-complex $Y$, whose isotropy groups belong to $\mathcal{F}$, there is, up to $G$-homotopy, precisely one $G$-map $Y\to E_{\mathcal{F}}(G)$.
 \end{itemize}
We abbreviate $\underline{E}G := E_{\mathcal{F}_\mathcal{COM}}(G)$ and call it the universal $G$-CW-complex for proper $G$-actions. We also write $EG= E_{\mathcal{F}_{\mathcal{TR}}}(G)$.}
\end{definition}
\begin{definition}(Homotopy characterization of $E_{\mathcal{F}}(G)$)\rm{
Let $\mathcal{F}$ be a family of subgroups.
\begin{itemize}
\item[(i)] There exists a model for $E_{\mathcal{F}}(G)$;
\item[(ii)] A $G$-CW-complex $X$ is a model for $E_{\mathcal{F}}(G)$ if and only if all its isotropy groups belong to $\mathcal{F}$ and for each $H\in \mathcal{F}$, the $H$-fixed point set is weakly contractible.
\end{itemize}
For more information about these spaces $E_{\mathcal{F}}(G)$ we refer to the survey article \cite{Luc05}.}
\end{definition}
\begin{conjecture}\label{fjcok}(K-theoretic Farrell-Jones-Conjecture)
Let $R$ be any associative ring with unit (with involution) and let $G$ be a discrete group. Then the assembly map $$H^{G}_n(E_{\mathcal{F}_\mathcal{VCYC}}(G),{\bf{K}}_R)\mapsto H^{G}_n(pt, {\bf{K}}_R)\cong K_n(RG)$$ induced by the projection $E_{\mathcal{F}_\mathcal{VCYC}}(G)\mapsto pt$ is bijective for all $n\in \mathbb{Z}$.
\end{conjecture}
\begin{conjecture}\label{fjcol}(L-theoretic Farrell-Jones-Conjecture)
Let $R$ be any associative ring with unit (with involution) and let $G$ be a discrete group. Then the assembly map $$H^{G}_n(E_{\mathcal{F}_{\mathcal{VCYC}}}(G),{\bf{L}}_{R}^{<-\infty>})\mapsto H^{G}_n(pt, {\bf{L}}_{R}^{<-\infty>})\cong L_{n}^{<-\infty>}(RG)$$ induced by the projection $E_{\mathcal{F}_{\mathcal{VCYC}}}(G)\mapsto pt$ is bijective for all $n\in \mathbb{Z}$.
\end{conjecture}
\begin{conjecture}\label{bccon}(Baum-Connes Conjecture)
Let $R$ be any associative ring with unit (with involution) and let $G$ be a discrete group. Then the assembly map $$K^{G}_n(EG)=H^{G}_n(E_{\mathcal{F}_{\mathcal{FIN}}}(G),{\bf{K}}^{top})\mapsto H^{G}_n(pt, {\bf{K}}^{top})=K_n(C^{*}_r(G))$$ induced by the projection $E_{\mathcal{F}_{\mathcal{FIN}}}(G)\mapsto pt$ is bijective for all $n\in \mathbb{Z}$.
\end{conjecture}
Next we want to discuss, whether one can pass to smaller or larger families in the formulations of the Conjectures. The point is to find the family as small as possible.
\begin{theorem}\label{fjtp}(Transitivity Principle)\cite{BL07}
Let $\mathcal{F}\subseteq \mathcal{G}$ be two families of subgroups of $G$. Let $\mathcal{H}^{?}_{*}$ be an equivariant homology theory. Assume that for every element $H\in \mathcal{G}$ and $n\in \mathbb{Z}$ the assembly map
$$\mathcal{H}^{H}_{n}(E_{\mathcal{F}_{|H}}(H))\mapsto \mathcal{H}^{H}_n(pt)$$
is bijective, where $\mathcal{F}_{|H} = \{K\cap H | K\in \mathcal{F}\}$. Then the relative assembly map induced by the up to $G$-homotopy unique $G$-map $E_{\mathcal{F}}(G)\mapsto E_{\mathcal{G}}(G)$ $$\mathcal{H}^{G}_{n}(E_{\mathcal{F}}(G))\mapsto \mathcal{H}^{G}_{n}(E_{\mathcal{G}}(G))$$ is bijective for all $n\in \mathbb{Z}$.
\end{theorem}
\begin{remark}\rm{
\indent
 \begin{itemize}
\item[\bf{1.}] The Baum-Connes Conjecture \ref{bccon} is known to be true for virtually cyclic groups. The Transitivity Principle \ref{fjtp} implies that the relative assembly $$K^{G}_n(EG)\mapsto K^{G}_n(E_{\mathcal{F}_{\mathcal{VCYC}}}(G))$$ is bijective for all $n\in \mathbb{Z}$. Hence it does not matter in the context of the Baum-Connes Conjecture whether we consider the family $\mathcal{F}_\mathcal{FIN}$ or $\mathcal{F}_{\mathcal{VCYC}}$.
\item[\bf{2.}] In general, the relative assembly maps 
\begin{center}
 $H^{G}_n(\underline{E}G;{\bf{K}}_R)\mapsto H^{G}_n(E_{\mathcal{F}_{\mathcal{VCYC}}}(G); {\bf{K}}_R);$\\
 $H^{G}_n(\underline{E}G;{\bf{L}}_{R}^{<-\infty>})\mapsto H^{G}_n(E_{\mathcal{F}_{\mathcal{VCYC}}}(G); {\bf{L}}_{R}^{<-\infty>}).$
\end{center}
are not bijective \cite{BL06}. Hence in the Farrell-Jones setting one has to pass to $\mathcal{F}_{\mathcal{VCYC}}$ and cannot use the easier to handle family $\mathcal{F}_{\mathcal{FIN}}$.
\item[\bf{3.}] The Farrell-Jones Conjecture \ref{fjcok} for algebraic $K$-theory for the group $\mathbb{Z}$ is true for trivial reasons since $\mathbb{Z}$ is virtually cyclic and hence the projection $E_{\mathcal{F}_{\mathcal{VCYC}}}(\mathbb{Z})\mapsto pt$ is a homotopy equivalence.
 \end{itemize}}
\end{remark}
Let $\mathcal{FJ}_{K}(R)$ and $\mathcal{FJ}_{L}(R)$ be the class of groups which satisfy the $K$-theoretic and $L$-theoretic respectively Farrell-Jones Conjecture for the coefficient ring (with involution) $R$. Let $\mathcal{BC}$ be the class of groups which satisfy the Baum-Connes Conjecture.
Recall that a ring $R$ is called regular if it is Noetherian and every finitely generated $R$-module possesses a finite projective resolution.
\begin{theorem} (Lower and middle K-theory of group rings in the torsion free case)\label{torsfreefj}
Suppose that $G$ is torsionfree.
\begin{itemize}
\item[(i)] If $R$ is regular and $G\in \mathcal{FJ}_{K}(R)$, then
\begin{itemize}
\item[(a)] $K_n(RG) = 0$ for $n\leq -1$;
\item[(b)] The change of rings map $K_0(R)\mapsto K_0(RG)$ is bijective;
\item[(c)] In particular, $\widetilde{K}_0(RG)$ is trivial if and only if $\widetilde{K}_0(R)$ is trivial.
\end{itemize}
\item[(ii)] If $G\in \mathcal{FJ}_{K}(\mathbb{Z})$, then the Whitehead group $Wh(G)$ is trivial.
\end{itemize}
\end{theorem}
\begin{proof}
The idea of the proof is to study the Atiyah-Hirzebruch spectral sequence. It converges to $H_n(BG;{\bf{K}}_{R})$ \cite{LR05} which is isomorphic to $K_n(RG)$ by the assumption that $G\in \mathcal{FJ}_{K}(R)$. The $E^2$-term is given by $$E^{2}_{p,q}=H_p(BG,K_q(R)).$$
Claim(i):
Since $R$ is regular by assumption, we get $K_q(R)=0$ for $q\leq -1$ \cite[5.3.30 on page 295]{Ros94}. Hence the spectral sequence is a first quadrant spectral sequence. This implies $K_n(RG)\cong H_n(BG;{\bf{K}}_{R})=0$ for $n\leq −1$ and the edge homomorphism yields an isomorphism $$K_0(R)=H_0(pt,K_0(R))\stackrel{\cong}{\longrightarrow}H_0(BG;{\bf{K}}_{R})\cong K_0(RG).$$ This proves (i).\\
Claim(ii):
We have $K_0(\mathbb{Z})=\mathbb{Z}$ and $K_1(\mathbb{Z})=\{\pm 1\}$. We get an exact sequence $$0\mapsto H_0(BG;K_1(\mathbb{Z}))=\{\pm 1 \}\mapsto H_1(BG;{\bf{K}}_{\mathbb{Z}})\cong K_1(\mathbb{Z}G)\mapsto H_1(BG;K_0(\mathbb{Z}))=G/[G,G]\mapsto 0.$$ This implies $$Wh(G):= K_1(\mathbb{Z}G)/ \{\pm g | g\in G \} = 0.$$ This proves (ii).
\end{proof}
By using Theorem \ref{torsfreefj}, we have the following:
\begin{theorem}\label{torsfreefjint}
Let $G\in \mathcal{FJ}_{K}(\mathbb{Z})$ be a torsionfree group. Then 
 \begin{itemize}
\item[(i)] $K_n(\mathbb{Z}G) = 0$ for $n\leq −1$;
\item[(ii)] $\widetilde{K}_0(\mathbb{Z}G) = 0$;
\item[(iii)] $Wh(G)=0$;
\item[(iv)] Every finitely dominated CW-complex $X$ with $G = \pi_1(X)$ is homotopy equivalent to a finite CW-complex;
\item[(v)] Every compact $h$-cobordism $W$ of dimension $\geq 6$ with $G\cong \pi_1(W)$ is trivial (For $G={1}$ this implies the Poincar$\grave{\rm e}$ Conjecture in dimension $\geq 5$).
\end{itemize}
\end{theorem}
\begin{remark}\rm{
\indent
 \begin{itemize}
\item[\bf{1.}] Theorem \ref{torsfreefjint} (iv) is a consequence of the following fact: Let $G$ be a finitely presented group. The vanishing of $\widetilde{K}_0(\mathbb{Z}G)$ is equivalent to the geometric statement that any finitely dominated space $X$ with $G\cong \pi_1(X)$ is homotopy equivalent to a finite CW-complex. Since the fundamental group of a finitely dominated CW-complex is always finitely presented \cite{LR05}.
\item[\bf{2.}] Theorem \ref{torsfreefjint} (v) follows from the s-cobordism theorem. In fact, for a finitely presented group $G$ the vanishing of the
Whitehead group $Wh(G)$ is equivalent to the statement that each $h$-cobordism over a closed connected manifold $M$ of dimension $\dim(M)\geq 5$ with fundamental group $\pi_1(M)\cong G$ is trivial \cite{LR05}.
\end{itemize}}
 \end{remark}
Let ${\bf{L}}\langle 1 \rangle$ be the 1-connective cover of the $L$-theory spectrum ${\bf{L}}$. It is characterized by the following property: There is a natural map of spectra $u:{\bf{L}}\langle 1 \rangle\longrightarrow {\bf{L}}$ which induces an isomorphism on the homotopy groups in dimensions $n\geq 1$ and the homotopy groups of  ${\bf{L}}\langle 1 \rangle$ vanish in dimensions $n\leq 0$.
\begin{theorem}\label{ranicki}(Ranicki (1992))
There is an exact sequence of abelian groups, called algebraic surgery exact sequence, for an $n$-dimensional closed manifold $M$\\
$....\stackrel{\sigma_{n+1}}{\longrightarrow}H_{n+1}(M;{\bf{L}}\langle 1 \rangle)\stackrel{A_{n+1}}{\longrightarrow}L_{n+1}(\mathbb{Z}[\pi_1(M)])\stackrel{\partial_{n+1}}{\longrightarrow}\mathcal{S}(M)\stackrel{\sigma_{n}}{\longrightarrow} 
H_{n}(M;{\bf{L}}\langle 1 \rangle)\stackrel{A_{n}}{\longrightarrow}\\
L_{n}(\mathbb{Z}[\pi_1(M)])\stackrel{\partial_{n}}{\longrightarrow}...$\\
It can be identified with the classical geometric surgery sequence \ref{surexa} due to Sullivan and Wall in high dimensions.
\end{theorem}
\begin{theorem}\label{fjimpbor}(The Farrell-Jones Conjecture and the Borel Conjecture)
If the $K$-theoretic Farrell-Jones Conjecture \ref{fjcok} and $L$-theoretic Farrell-Jones Conjecture \ref{fjcol} hold for $G$ in the case $R=\mathbb{Z}$, then the Borel Conjecture \ref{borel} is true in dimension $\geq 5$ and in dimension 4 if $G$ is good in the sense of Freedman.
\end{theorem}
\paragraph{Sketch of the proof.}
The $K$-theoretic version of the Farrell-Jones Conjecture ensures that we do not have to deal with decorations, e.g., it does not matter if we consider ${\bf{L}}$ or ${\bf{L}}^{\langle -\infty \rangle }$. This follows from the so called Rothenberg sequences \cite{KL04}. The $L$-theoretic version of the Farrell-Jones Conjecture implies that $H_n(M;{\bf{L}})\to L_n(\mathbb{Z}\pi_1(M))$ is bijective for all $n\in \mathbb{Z}$. Let $\bf{F}$ be the homotopy fiber of $u:{\bf{L}}\langle 1 \rangle\longrightarrow {\bf{L}}$. Hence we have a fibration of spectra $$\bf{F}\longrightarrow \bf{L}\langle 1 \rangle\stackrel{u}{\longrightarrow} \bf{L}$$ which induces a long exact sequence\\ $...\longrightarrow H_{k+1}(M;{\bf{L}}\langle 1 \rangle)\longrightarrow H_{k+1}(M;{\bf{L}})\longrightarrow H_{k}(M;{\bf{F}})\longrightarrow H_{k}(M;{\bf{L}}\langle 1 \rangle)\longrightarrow H_{k}(M;{\bf{L}})\longrightarrow...$\\ Since $\pi_q(\bf{F})=0$ for $q\geq 0$, an easy spectral sequence argument shows that $H_{k}(M;{\bf{F}})=0$ for $k\geq 
n$. Hence the map $H_k(M; {\bf{L}}\langle 1 \rangle) \longrightarrow H_
k(M;{\bf{L}})$ is bijective for $k\geq n + 1$ and injective for $k = n$. For $k = n$ and $k = n+1$, the map $A_k$ is the composite of the map $H_k(M; {\bf{L}}\langle 1 \rangle) \longrightarrow H_k(M;{\bf{L}})$ with the map $H_k(M;{\bf{L}})\longrightarrow L_k(\mathbb{Z}\pi_1(M))$. Hence $A_{n+1}$ is surjective and $A_n$ is injective. Theorem \ref{ranicki} implies that $\mathcal{S}(M)$ consist of one element. This complete the proof of Theorem \ref{fjimpbor}.\\

The Farrell-Jones Conjecture and the Baum-Conjecture imply certain other well-known conjectures :
\begin{theorem}\label{novifj}(The Farrell-Jones, the Baum-Connes and the Novikov Conjecture)
Suppose that one of the following assembly maps
 $$H^{G}_n(E_{\mathcal{F}_{\mathcal{VCYC}}}(G),{\bf{L}}_{R}^{<-\infty>})\mapsto H^{G}_n(pt, {\bf{L}}_{R}^{<-\infty>})\cong L_{n}^{<-\infty>}(RG);$$
 $$K^{G}_n(EG)=H^{G}_n(E_{\mathcal{F}_{\mathcal{FIN}}}(G),{\bf{K}}^{top})\mapsto H^{G}_n(pt, {\bf{K}}^{top})=K_n(C^{*}_r(G)),$$
is rationally injective.\\
Then the Novikov Conjecture holds for the group $G$.
\end{theorem}
James F. Davis and Wolfgang L$\ddot{\rm u}$ck gave a unified approach to the Isomorphism Conjecture of Farrell and Jones on the algebraic K- and L-theory of integral group rings and to the Baum-Connes Conjecture on the topological $K$-theory of reduced group $C^{*}$-algebras. The approach is through spectra over the orbit category of a discrete group $G$ \cite{DL98}:
\begin{definition}\label{}\cite{DL98}
Let $G$ be a group and let $\mathcal{F}$ be a family of subgroups of $G$. The orbit category $Or(G)$ has as objects homogeneous $G$-spaces $G/H$ and as morphisms $G$-maps. The orbit category $Or(G,\mathcal{F})$ with respect to $\mathcal{F}$ is the full subcategory of $Or(G)$ consisting of those objects $G/H$ for which $H$ belongs to $\mathcal{F}$.
\end{definition}
Let ${\bf{E}}: Or(G)\to SPECTRA$ be a covariant functor and an extension of $\bf{E}$ to the category of $G$-spaces by
$$\widetilde{\bf{E}}: G\text{-}SPACES\to SPECTRA$$ such that $\widetilde{\bf{E}}(X)=\rm{map}_{G}(-,X)_{+}\otimes_{Or(G)}\bf{E}$. Recall that $\rm{map}_{G}(-,X)_{+}\otimes_{Or(G)}{\bf{E}}=\coprod_{H\subset G}X^{H}_{+}\wedge{\bf{E}}(G/H)/\sim$ where $\sim$ is the equivalence relation generated by $(x\phi,y)\sim (x,\phi y)$ for $x\in X^{H}_{+}=\rm{map}_{G}(G/H,X_{+})$, $y\in {\bf{E}}(G/H)$ and $\phi:G/H \to G/K$ \cite{DL98}. Then $\pi_{*}(\widetilde{\bf{E}}(X))$ is an equivariant homology theory in the sense of Bredon \cite{Bre67}. Let $E_{\mathcal{F}}(G)$ be the classifying space for a family of subgroups of $G$. The map 
$$\pi_{*}\widetilde{\bf{E}}(E_{\mathcal{F}}(G))\longrightarrow \pi_{*}\widetilde{\bf{E}}(E_{\mathcal{F}_{\mathcal{ALL}}}(G))$$ given by applying $\widetilde{\bf{E}}$ to the constant map and then taking homotopy groups is called the $({\bf{E}}, \mathcal{F}, G)$-assembly map.
\begin{definition}\label{}\cite{DL98}
The $({\bf{E}}, \mathcal{F}, G)$-Isomorphism Conjecture for a discrete group $G$, a family of subgroups $\mathcal{F}$, and a covariant $Or(G)$-spectrum ${\bf{E}}$ is that the $({\bf{E}}, \mathcal{F}, G)$-assembly map is an isomorphism. For an integer $i$, the $({\bf{E}}, \mathcal{F}, G)$-Isomorphism Conjecture is that the $({\bf{E}}, \mathcal{F}, G)$-assembly map is an isomorphism in dimension $i$.
\end{definition}
\begin{remark}\rm{
\indent
\begin{itemize}
\item[\bf{1.}] When ${\bf{E}}$ equals the algebraic $K$-theory spectra ${\bf{K}}^{alg}$ or the algebraic $L$-theory spectra ${\bf{L}}^{<-\infty>}$ and $\mathcal{F}$ is the family $\mathcal{F}_{\mathcal{VCYC}}$ of virtually cyclic subgroups of $G$, then the Isomorphism Conjecture is the one of Farrell-Jones Isomorphism Conjecture \ref{fjcok} and Conjecture \ref{fjcol} respectively \cite{DL98}.
\item[\bf{2.}] When ${\bf{E}}$ equals the topological $K$-theory spectrum ${\bf{K}}^{top}$ and $\mathcal{F}$ is the family $\mathcal{F}_{\mathcal{FIN}}$ of finite subgroups of $G$, then the Isomorphism Conjecture is the Baum- Connes Conjecture \ref{bccon} \cite{DL98}.
\end{itemize}}
\end{remark}
Let $X$ be a connected CW-complex (perhaps a manifold), and let $\mathcal{P}_{*}$, $\mathcal{P}_{*}^{\rm{Diff}}$ denote the functor that maps $X$ to the $\Omega$-spectrum of stable topological (smooth) pseudo-isotopies on $X$. Denote by $\mathcal{K}_{*}()$ the functor that maps $X$ to the algebraic $K$-theoretic (non-connective) $\Omega$-spectrum for the integral group ring $\mathbb{Z}\pi_1(X)$ \cite{PW85}. Let ${\bf{K}}^{alg} : Or(G)\to \Omega \text{-}SPECTRA$ be the algebraic $K$-theory functor \cite{DL98}. The homotopy groups of the spectrum ${\bf{K}}^{alg}(G/H)$ are isomorphic to the $K$-theory groups of $\mathbb{Z}H$. Finally, let ${hocolim}_{Or(G,\mathcal{F})}~{\bf{K}}^{alg}$ be the homotopy colimit of the ${\bf{K}}^{alg}$ functor over the $\mathcal{F}$-orbit category \cite{DL98}. For $\mathcal{F}_{\mathcal{TR}}$ and $\mathcal{F}_{\mathcal{ALL}}$ the following identifications can be made: $$\pi_n({hocolim}_{Or(G)}~{\bf{K}}^{alg})\cong K_n(\mathbb{Z}G), \forall n,$$ and $$\pi_n({hocolim}_{Or(G,\mathcal{F}_{\mathcal{TR}
})}~{\bf{K}}^{alg})\cong \mathbb{H}_n(BG;\mathcal{K}_{*}(pt)), \forall n$$ where $\mathbb{H}$ denotes homology with coefficients in a spectrum, and $\mathcal{K}_{*}(pt)$ denotes the algebraic $K$-theory spectrum of the integers. Moreover, given two families $\mathcal{F}\subseteq \mathcal{F}^{'}$ of subgroups of $G$, the inclusion induces a map
$$A_{\mathcal{F},\mathcal{F}^{'}}:\pi_{n}({hocolim}_{Or(G,\mathcal{F})}~{\bf{K}}^{alg})\to \pi_{n}({hocolim}_{Or(G,\mathcal{F}^{'})}~{\bf{K}}^{alg}).$$ These are collectively known as assembly maps (\cite{DL98, FJ93a}). In the special case $\mathcal{F}=\mathcal{F}_{\mathcal{TR}}$ and $\mathcal{F}^{'}=\mathcal{F}_{\mathcal{ALL}}$, $A_{\mathcal{F},\mathcal{F}^{'}}$ is the classical assembly map $$A:\mathbb{H}_n(BG;\mathcal{K}_{*}(pt))\to K_n(\mathbb{Z}G).$$ These assembly maps have the property that given families $\mathcal{F}_1\subseteq \mathcal{F}_{2}\subseteq \mathcal{F}_{3}$ of subgroups of $G$, we have $$A_{\mathcal{F}_1,\mathcal{F}_3}=A_{\mathcal{F}_2,\mathcal{F}_3}\circ A_{\mathcal{F}_1,\mathcal{F}_2}.$$ In general these assembly maps need not be isomorphisms. However, they are key maps when trying to approach the algebraic $K$-groups of a given group through a special collection of its subgroups.\\
Let $\mathcal{S}_{*}$ be a homotopy invariant (covariant) functor from the category of topological spaces to spectra. Important examples of such functors are the stable topological pseudo-isotopy functor $\mathcal{P}_{*}$, the algebraic $K$-theory functor $\mathcal{K}_*$, and the $L$-theory functor $\mathcal{L}^{-\infty}_*$ \cite{FJ93a}. Let $\mathcal{M}$ denote the category of continuous surjective maps; i.e., an object in $\mathcal{M}$ is a continuous map $p: E\to B$ between topological spaces $E$ and $B$, while a morphism from $p_1 : E_1\to B_1$ to $p_2 : E_2\to B_2$ is a pair of continuous maps $f : E_1\to E_2$, $g : B_1\to B_2$ making the following diagram a commutative square of maps:
$$\begin{CD}
E_1 @>f>> E_2\\
@Vp_1VV @Vp_2VV\\
B_1 @>g>> B_2
\end{CD}$$
Quinn \cite[appendix]{Qui82} constructed a functor from $\mathcal{M}$ to the category of $\Omega$-spectra which associates to the map $p$ the spectrum $\mathbb{H}(B; \mathcal{S}(p))$ in such a way that $$\mathbb{H}(B; \mathcal{S}(p))=\mathcal{S}(E)$$ in the special case that $B$ is a single point pt. Furthermore the map of spectra $$A:\mathbb{H}(B; \mathcal{S}(p))=\mathcal{S}(E)$$ functorially associated to the commutative square
$$\begin{CD}
E @>id>> E\\
@VpVV @VVV\\
B @>>> pt
\end{CD}$$
is called the (Quinn) assembly map.
\begin{definition}\rm{
Let $G$ denote a (discrete) group, and let $\mathcal{F}$ denote a family of subgroups of $G$. We define a universal $(G, \mathcal{F})$-space to be a regular cell complex $Z$ together with a group action $G\times Z\to Z$, which satisfies the following properties :
\begin{itemize}
\item[(a)] For each $g\in G$ the homeomorphism $Z\to Z$ given by $z\longrightarrow g(z)$ is cellular; moreover, for each cell $e\in Z$ if $g(e)=e$ then $g_{|e} = $inclusion.
\item[(b)] For any $z\in Z$ we have $G^z\in \mathcal{F}$, where $G^z$ is the isotropy group at $z$ for this action. 
\item[(c)] For any $\Gamma\in \mathcal{F}$ the fixed point set of $\Gamma\times Z\to Z$ is a nonempty contractible subcomplex of $Z$.
\end{itemize}}
\end{definition}
\begin{definition}\rm{
Let $X$ denote any connected CW-complex, and let $\mathcal{F}_{\mathcal{VCYC}}(X)$ consists of all virtually cyclic subgroups of $\pi_1(X)$. F.T. Farrell \cite{FJ93a} proved that there is a universal $(\pi_1(X), \mathcal{F}_{\mathcal{VCYC}}(X))$-space $\pi_1(X)\times A \longrightarrow A$. Let $\widetilde{X}$ denote the universal covering space for $X$, and let $\pi_1(X)\times (\widetilde{X}\times A) \longrightarrow \widetilde{X}\times A$ denote the diagonal action. Define $\rho:\mathcal{E}(X)\to \mathcal{B}(X)$ to be the quotient of the standard projection $\widetilde{X}\times A \longrightarrow A$ under the relevant $\pi_1(X)$-actions, and define $f : \mathcal{E}(X)\to X$ to be the quotient of the standard projection $\widetilde{X}\times A \mapsto \widetilde{X}$ under the relevant $\pi_1(X)$-actions. Because the universal $(\pi_1(X), \mathcal{F}_{\mathcal{VCYC}}(X))$-space $\pi_1(X)\times A \longrightarrow A$ is uniquely determined by $X$ up to $\pi_1(X)$- equivariant homotopy type \cite{FJ93a}, we get the following 
lemma. Recall that for each 
$H\in \mathcal{F}_{\mathcal{VCYC}}(X)$ we denote by $X_H\to X$ the covering space for $X$ corresponding to $H$.}
\end{definition}
\begin{lemma}
$\rho:\mathcal{E}(X)\to \mathcal{B}(X)$ is a simplicially stratified fibration. Each fiber $\rho^{-1}(z)$ of $\rho$ is one of the connected covering spaces $\{X_H: H\in \mathcal{F}_{\mathcal{VCYC}}(X)\}$ for $X$, in fact, the restricted map $f:\rho^{-1}(z)\to X$ is a covering space projection whose image on the fundamental group level is contained in $\mathcal{F}_{\mathcal{VCYC}}(X)$. Moreover, $\rho:\mathcal{E}(X)\to \mathcal{B}(X)$ is uniquely determined up to fibered homotopy type by the homotopy type of $X$.
\end{lemma}
\begin{conjecture}\label{pesduic}(Pseudoisotopy Version of Farrell-Jones Isomorphism Conjecture)
The composite $$\mathbb{H}_{*}(\mathcal{B}(X),\mathcal{S}_{*}(\rho))\stackrel{\mathcal{S}_{*}(f)\circ A_{*}}{\longrightarrow}\mathcal{S}_{*}(X)$$ is a (weak) equivalence of  $\Omega$-spectra, where $A_*$ is the assembly map for the simplicially stratified fibration $\rho:\mathcal{E}(X)\to \mathcal{B}(X)$, the functor $\mathcal{S}_{*}()$ is any of $\mathcal{P}_{*}$, $\mathcal{P}_{*}^{\rm{Diff}}$, $\mathcal{K}_{*}$ or the $L^{-\infty}$-surgery functor $\mathcal{L}^{-\infty}_{*}$, and $\mathcal{S}_{*}(f)$ is the image of the map $f : \mathcal{E}(X)\to X$ under $\mathcal{S}_{*}()$.
\end{conjecture}
Farrell and Jones \cite{FJ93a} proved the Isomorphism Conjecture \ref{pesduic} for discrete cocompact virtually torsion-free subgroups of the isometry group of the universal cover of a closed non-positively curved manifold. Here are the results :
\begin{theorem}
The Isomorphism Conjecture \ref{pesduic} is true for the functors $\mathcal{P}_{*}()$, $\mathcal{P}_{*}^{\rm{Diff}}()$ on the space $X$ provided that there exists a simply connected symmetric Riemannian manifold $M$ with non-positive sectional curvature everywhere such that $M$ admits a properly discontinuous cocompact (i.e., such that the orbit space $X= M/G$ is compact) group action of $G=\pi_1(X)$ by isometries of $M$.
\end{theorem}
\begin{theorem}
Let $X$ be a connected CW-complex such that $\pi_1(X)$ is a subgroup of a cocompact discrete subgroup of a virtually connected Lie group. Then the Isomorphism Conjecture \ref{pesduic} is true for the functors $\mathcal{P}_{*}()$, $\mathcal{P}_{*}^{\rm{Diff}}()$ on the space $X$.
\end{theorem}
E. Berkove, F.T. Farrell, D. Juan-pineda and K. Pearson \cite{BFJP00} proved the Farrell-Jones Isomorphism Conjecture \ref{pesduic} for groups acting on complete hyperbolic manifolds with finite volume orbit space. Here are the results:
\begin{theorem}\label{bfjpcon}
The Isomorphism Conjecture is true for the functors $\mathcal{P}_{*}()$, $\mathcal{P}_{*}^{\rm{Diff}}()$ on the space $X$ provided that there exists a properly discontinuous finite co-volume group action by isometries of $G=\pi_1(X)$ on a hyperbolic space $\mathbb{H}^n$.
\end{theorem}
\begin{corollary}
Let $G$ be a group for which the Isomorphism Conjecture \ref{pesduic} holds. Let $\mathcal{F}_{\mathcal{VCYC}}$ and $\mathcal{F}_{\mathcal{ALL}}$ be the families of virtually cyclic and of all subgroups of $G$, respectively. Then the assembly map
$$A_{\mathcal{F}_{\mathcal{VCYC}},\mathcal{F}_{\mathcal{ALL}}}:\pi_n({hocolim}_{\mathcal{D}(G,\mathcal{F}_{\mathcal{VCYC}})}~{\bf{K}}^{alg})\to K_n(\mathbb{Z}G)$$ is an isomorphism for $n\leq 1$.
\end{corollary}
E. Berkove, F.T. Farrell, D. Juan-pineda and K. Pearson \cite{BFJP00} classified all virtually cyclic subgroups in the Bianchi group family and then show that the lower algebraic K-theory of all the virtually cyclic subgroups vanishes. Here are the results :
\begin{theorem}
Let $G$ be a group for which the Isomorphism Conjecture \ref{pesduic} holds for the functor $\mathcal{P}_*$. Suppose that for every virtually cyclic subgroup $\Gamma$ of $G$  $Wh(\Gamma)$, $\widetilde{K}_0(\mathbb{Z}\Gamma)$ and $K_n(\mathbb{Z}\Gamma)$ for $n\leq −1$, all vanish. Then 
 \begin{itemize}
\item[(i)] $K_n(\mathbb{Z}G) = 0$ for $n\leq −1$;
\item[(ii)] $\widetilde{K}_0(\mathbb{Z}G) = 0$;
\item[(iii)] $Wh(G)=0$;
\end{itemize}
\end{theorem}
\begin{theorem}\label{}\cite{BFJP00}
Let $G$ be a Bianchi group. Then 
 \begin{itemize}
\item[(i)] $K_n(\mathbb{Z}G) = 0$ for $n\leq −1$;
\item[(ii)] $\widetilde{K}_0(\mathbb{Z}G) = 0$;
\item[(iii)] $Wh(G)=0$;
\end{itemize}
\end{theorem}
There is a stronger version of the Farrell-Jones Conjecture, the so called Fibered Farrell-Jones Conjecture \cite{FJ93a}. The Fibered Farrell-Jones Conjecture does imply the Farrell-Jones Conjecture and has better inheritance properties than the Farrell-Jones Conjecture :
\begin{conjecture}\label{fibcj}(The Fibered isomorphism conjecture)
Let $X$ be a connected CW-complex, and let $\xi= Y\to X$ denote a Serre fibration over $X$. Let $\mathcal{E}(\xi)$ denote the total space of the pullback of $\xi$ along the map $f : \mathcal{E}(X)\to X$, let $\rho(\xi):\mathcal{E}(\xi)\to \mathcal{B}(\xi)$ denote the composite map $\mathcal{E}(\xi)\stackrel{proj}{\longrightarrow}\mathcal{E}(X)\stackrel{\rho}{\longrightarrow} \mathcal{B}(X)$, and let $f(\xi) : \mathcal{E}(\xi)\to Y$ denote the map which covers the map $f : \mathcal{E}(X)\to X$. The Fibered Isomorphism Conjecture states that the composite $$\mathbb{H}_{*}(\mathcal{B}(\xi),\mathcal{S}_{*}(\rho(\xi)))\stackrel{\mathcal{S}_{*}(f(\xi))\circ A_{*}}{\longrightarrow}\mathcal{S}_{*}(Y)$$ is a (weak) equivalence of  $\Omega$-spectra, where $A_*$ is the assembly map for the simplicially stratified fibration $\rho(\xi):\mathcal{S}(\xi)\to \mathcal{B}(\xi)$, the functor $\mathcal{S}_{*}()$ is any of $\mathcal{P}_{*}$, $\mathcal{P}_{*}^{Diff}$, $\mathcal{K}_{*}$ or the $L^{-\infty}$-surgery functor $\
mathcal{L}^{-\infty}_{*}$, and $\mathcal{S}_{*}(f(\xi))$ is the image of the map $f(\xi) : \mathcal{E}(\xi)\to Y$ under $\mathcal{S}_{*}()$.
\end{conjecture}
\begin{definition}\label{equifjcoj}((Fibered) Isomorphism Conjecture for $\mathcal{H}^{?}_{*}$)\rm{
Given a group homomorphism $\phi:K\to G$ and a family $\mathcal{F}$ of subgroups of $G$, define the family $\phi^{*}\mathcal{F}$ of subgroups of $K$ by $\phi^{*}\mathcal{F}=\{H\subseteq K |\phi(H)\in (H)\mathcal{F}\}$. Let $\mathcal{H}^{?}_{*}$ be an equivariant homology theory with values in $\Lambda$-modules for a commutative associative ring $\Lambda$ with unit from \cite[Section 1]{Luc02}. This essentially means that we get for each group $G$, a $G$-homology theory $\mathcal{H}^{G}_{*}$ which assigns to a (not necessarily proper or cocompact) pair of $G$-CWcomplexes $(X,A)$, a $\mathbb{Z}$-graded $\Lambda$-module $\mathcal{H}^{G}_{n}(X,A)$, such that there exists natural long exact sequences of pairs and $G$-homotopy invariance, excision, and the disjoint union axiom are satisfied. Moreover, an induction structure is required which in particular implies for a subgroup $H\subseteq G$ and a $H$-CW-pair $(X,A)$ that there is a natural isomorphism $$\mathcal{H}^{H}_{n}(X,A)\stackrel{\cong}{\longrightarrow}\
mathcal{
H}
^{G}_{n}(G\times_{H}(X,A)).$$ \\
A group $G$ together with a family of subgroups $\mathcal{F}$ satisfies the Isomorphism Conjecture for $\mathcal{H}^{?}_{*}$ if the projection $pr : E_{\mathcal{F}}(G)\to pt$ induces an isomorphism $$\mathcal{H}^{G}_{n}(pr):\mathcal{H}^{G}_{n}(E_{\mathcal{F}}(G))\stackrel{\cong}{\longrightarrow}\mathcal{H}^{G}_{n}(pt)$$ for $n\in \mathbb{Z}$. The pair $(G,\mathcal{F})$ satisfies the Fibered Isomorphism Conjecture for $\mathcal{H}^{?}_{*}$ if for every group homomorphism $\phi: K\to G$, the pair $(K,\phi^{*}\mathcal{F})$ satisfies the Isomorphism Conjecture.}
\end{definition}
\begin{remark}\label{ficrema}\rm{
The Fibered Isomorphism Conjecture \ref{fibcj} for algebraic $K$-theory or algebraic $L$-theory respectively for the group $G$  is equivalent to the (Fibered) Isomorphisms Conjecture \ref{equifjcoj} for $\mathcal{H}^{?}_{*}(-,{\bf{K}}_{R})$ and $\mathcal{H}^{?}_{*}(-,{\bf{L}}_{R})$ for the pair $(G,\mathcal{F}_{\mathcal{VCYC}})$ (see Remark 6.6 in \cite{BL06}). Arthur Bartels, Wolfgang L¨uck and Holger Reich \cite{BLR08} proved the following results :}
\end{remark}
\begin{theorem}\label{nilfjcon}
Let $R$ be an associative ring with unit. Let $\mathcal{FJ}(R)$ be the class of groups which satisfy the Fibered Farrell-Jones Conjecture \ref{fibcj} for algebraic $K$-theory with coefficients in $R$. Then
\begin{itemize}
\item[(i)] Every word-hyperbolic group and every virtually nilpotent group belong to $\mathcal{FJ}(R)$;
\item[(ii)] If $G_1$ and $G_2$ belong to $\mathcal{FJ}(R)$, then $G_1\times G_2$ belongs to $\mathcal{FJ}(R)$;
\item[(iii)] Let $\{G_i | i\in I\}$ be a directed system of groups (with not necessarily injective structure maps) such that $G_{i}\in \mathcal{F}$ for $i\in I$. Then ${colim}_{i\in I}G_i$ belongs to $\mathcal{FJ}(R)$;
\item[(iv)] If $H$ is a subgroup of $G$ and $G\in \mathcal{FJ}(R)$, then $H\in \mathcal{FJ}(R)$.
\end{itemize}
\end{theorem}
\begin{theorem}\label{worfjcon}
Let $R$ be an associative ring with unit. Consider the following assertions for a group $G$:
\begin{itemize}
\item[(KH)] The group $G$ satisfies the Fibered Farrell-Jones Conjecture for homotopy $K$-theory with coefficients in $R$;
\item[(FC)] The ring $R$ has finite characteristic $N$. The Fibered Farrell-Jones Conjecture for algebraic $K$-theory for $G$ with coefficients in $R$ for both the families $\mathcal{F}_{\mathcal{FIN}}$ and $\mathbb{F}_{\mathcal{VCYC}}$ is true after applying −$\otimes_{\mathbb{Z}}\mathbb{Z}[1/N]$ to the assembly map.
\end{itemize}
Let $\mathcal{FJ}_{KH}(R)$ be the class of groups for which assertion (KH) holds. If $R$ has finite characteristic, then let $\mathcal{FJ}_{FC}(R)$ be the class of groups for which assertion (FC) is true. Let $\mathcal{F}$ be $\mathcal{FJ}_{FC}(R)$ or $\mathcal{FJ}_{KH}(R)$. Then:
\begin{itemize}
\item[(i)] Every word-hyperbolic and every elementary amenable group belongs to $\mathcal{F}$;
\item[(ii)] If $G_1$ and $G_2$ belong to $\mathcal{F}$, then $G_1\times G_2$  belongs to $\mathcal{F}$;
\item[(iii)] Let $\{G_i | i\in I\}$ be a directed system of groups (with not necessarily injective structure maps) such that $G_i\in \mathcal{F}$ for $i\in I$. Then ${colim}_{i\in I}G_i$ belongs to $\mathcal{FJ}(R)$;
\item[(iv)] If $H$ is a subgroup of $G$ and $G\in \mathcal{FJ}(R)$, then $H\in \mathcal{FJ}(R)$;
\item[(v)] Let $1\to H\to G\to Q\to 1$ be an extension of groups such that $H$ is either elementary amenable or word-hyperbolic and $Q$ belongs to $\mathcal{F}$. Then $G$ belongs to $\mathcal{F}$;
\item[(vi)] Suppose that $G$ acts on a tree $T$. Assume that for each $x\in T$ the isotropy group $G_x$ belongs to $\mathcal{F}$. Then $G$ belongs to $\mathcal{F}$.
\end{itemize}
Moreover, if $R$ has finite characteristic then we have $\mathcal{FJ}_{KH}(R)\subseteq \mathcal{FJ}_{FC}(R)$.
\end{theorem}
\begin{corollary}\label{corrworfjcon}
Let $R$ be a regular associative ring with unit of finite characteristic $N$. Let $G$ be torsionfree. Suppose that $G$ belongs to the class $\mathcal{FJ}_{FC}(R)$ defined in Theorem \ref{worfjcon}. Then
\begin{itemize}
\item[(i)] $K_n(RG)[1/N]=0$ for $n\leq −1$;
\item[(ii)] The change of rings map induces a bijection $K_0(R)[1/N]\to K_0(RG)[1/N]$. In particular $\widetilde{K}_0(RG)[1/N]$ is trivial if and only if $\widetilde{K}_0(R)[1/N]$ is trivial;
\item[(iii)] $Wh^R(G)[1/N]$ is trivial.
\end{itemize}
\end{corollary}
\begin{remark}\rm{
Corollary \ref{corrworfjcon} together with Theorem \ref{worfjcon} substantially extends the following Theorem \ref{far-linn} of Farrell-Linnell \cite{FL03}, where $Wh^{F}(G)\otimes_{\mathbb{Z}} \mathbb{Q}=0$ is proven for $G$ a torsionfree elementary amenable group and $\mathbb{F}$ a field of prime characteristic. }
\end{remark}
\begin{theorem}\label{far-linn}
Let $G$ be a torsion-free virtually solvable subgroup of $GL_{n}(\mathbb{C})$. Then $Wh(G)=0$.
\end{theorem}
\begin{definition}\rm{
Let $G$ be a group. Then $G$ is nearly crystallographic means that $G$ is finitely generated and there exists $A \lhd G$ such that $A$ is torsion-free abelian of finite rank (i.e., is isomorphic to a subgroup of $\mathbb{Q}^n$ for some integer $n$), $C\leq G$ such that $C$ is virtually cyclic, $A\cap C=1$ and $AC=G$ (so $G\rtimes AC)$, and the conjugation action of $C$ on $A$ makes $A \otimes \mathbb{Q}$ into an irreducible $\mathbb{Q}C$-module.}
\end{definition}
\begin{theorem}\label{fjcrysolv}\cite{FL03}
Suppose the Fibered Isomorphism Conjecture \ref{fibcj} is true for all nearly crystallographic groups. Then the Fibered Isomorphism Conjecture \ref{fibcj} is true for all virtually solvable groups.
\end{theorem}
\begin{remark}\rm{
 We review some open problems on the Fibered Isomorphism Conjecture \ref{equifjcoj}:
 \begin{itemize}
\item[\bf{1.}] Show that the Fibered Isomorphism Conjecture is true for $A\rtimes \mathbb{Z}$ for a torsion-free abelian groups $A$ and for an arbitrary action of $\mathbb{Z}$ on $A$. Note that a positive answer to this problem will imply the conjecture for all solvable groups \cite{Rou14}.
\item[\bf{2.}] Show that the Fibered Isomorphism Conjecture is true for $G\rtimes \mathbb{Z}$ assuming the conjecture for $G$. This is a very important open problem and will imply the conjecture for poly free groups. It is open even when $G$ is finitely generated and free. For certain situations the answers are known, for example, when $G$ is a surface group and the action is realizable by diffeomorphism of the surface \cite{Rou14}.
\item[\bf{3.}] Prove that the Fibered Isomorphism Conjecture for the fundamental group of a graph of virtually cyclic groups. Even for the graph of infinite cyclic groups this is an open problem \cite{Rou14}.
\end{itemize}}
\end{remark}
A. Bartels and H. Reich \cite{BR07} introduced the Farrell-Jones Conjecture with coefficients in an additive category with $G$-action. This is a variant of the Farrell-Jones Conjecture about the algebraic $K$- or $L$-theory of a group ring $RG$. It allows to treat twisted group rings and crossed product rings. The conjecture with coefficients is stronger than the original conjecture but it has better inheritance properties. Here are the results :\\
In the following, we will consider additive categories $\mathcal{A}$ with a right G-action, i.e., to every group element $g$ we assign an additive covariant functor $g^{*}: \mathcal{A}\to \mathcal{A}$, such that $1^*= id$ and composition of functors (denoted $\circ$) relates to multiplication in the group via $g^{*}\circ h^{*}=(hg)^*$.
\begin{definition}\label{}\cite{BR07}\rm{
Let $\mathcal{A}$ be an additive category with a right $G$-action and let $T$ be a left $G$-set. We define a new additive category denoted $\mathcal{A}*_{G}T$ as follows: An object $A$ in $\mathcal{A}*_{G}T$ is a family $A=(A_{t})_{t\in T}$ of objects in $A$ where we require that $\{t\in T | A_t\neq 0\}$ is a finite set. A morphism $\phi: A\to B$ is a collection of morphisms $\phi=(\phi_{g,t})_{(g,t)\in G\times T}$, $\phi_{g,t}:A_t\to g^{*}(B_{gt})$ is a morphism in $\mathcal{A}$. We require that the set of pairs $(g,t)\in G\times T$ with $\phi_{g,t}\neq 0$ is finite. }
\end{definition}
Let $\mathbb{K}^{-\infty}: Add~Cat \to Sp$ be the functor that associates the non-connective $K$-theory spectrum to an additive category (using the split exact structure). This functor is constructed in \cite{PW85}. See \cite{BFJR04} for a brief review of this functor and its properties. To any such functor one can associate a $G$-homology theory $H^{G}_{*}(-, {\bf{K}}_{\mathcal{A}})$ (see \cite[Section 4 and 7]{DL98}).
\begin{definition}
Let $G$ be a group and let $\mathcal{A}$ be an additive category with right $G$-action. The $Or(G)$-spectrum ${\bf{K}}_{\mathcal{A}}$ is defined by ${\bf{K}}_{\mathcal{A}}(T)=\mathbb{K}^{-\infty}(\mathcal{A}*_{G}T)$.
\end{definition}
\begin{conjecture}\label{kfjciwthcof}(Algebraic K-Theory Farrell-Jones-Conjecture with Coefficients)
Let $G$ be a group and let $\mathcal{F}_{\mathcal{VCYC}}$  be the family of virtually cyclic subgroups of $G$. Let $\mathcal{A}$ be an additive category with a right $G$-action. Then the assembly map $$H^{G}_{*}(E_{\mathcal{F}_{\mathcal{VCYC}}}(G);{\bf{K}}_{\mathcal{A}})\to H^{G}_{*}(pt;{\bf{K}}_{\mathcal{A}})\cong K_{*}(\mathcal{A}*_{G}pt)$$ is an isomorphism. The right hand side of the assembly map $H^{G}_{*}(pt;{\bf{K}}_{\mathcal{A}})$ can be identified with $K_{*}(\mathcal{A}*_{G}pt)$, the $K$-theory of a certain additive category $\mathcal{A}*_{G}pt$.
\end{conjecture}
\begin{remark}\rm{
If $\mathcal{A}$ is the category of finitely generated free $R$-modules and is equipped with the trivial $G$-action, then $\pi_n({\bf{K}}_{\mathcal{A}}(G/G))\cong  K_n(RG)$ and the assembly map becomes $$H^{G}_{n}(E_{\mathcal{F}_{\mathcal{VCYC}}}(G),{\bf{K}}_R)\mapsto H^{G}_n(pt, {\bf{K}}_R)\cong K_n(RG).$$ This map can be identified with the one that appears in the original formulation of the Farrell-Jones Conjecture \ref{pesduic}. So the following Theorem \ref{hyfjwc} implies that the $K$-theoretic version of the Farrell-Jones Conjecture is true for hyperbolic groups and any coefficient ring $R$ \cite{BLR08a} :}
\end{remark}
\begin{theorem}\label{hyfjwc}
Let $G$ be a hyperbolic group. Then $G$ satisfies the $K$-theoretic Farrell-Jones Conjecture with coefficients, i.e., if $\mathcal{A}$ is an additive category with a right $G$-action, then for every $n\in \mathbb{Z}$ the assembly map $$H^{G}_{n}(E_{\mathcal{F}_{\mathcal{VCYC}}}(G);{\bf{K}}_{\mathcal{A}})\to H^{G}_{n}(pt;{\bf{K}}_{\mathcal{A}})\cong K_{n}(\mathcal{A}*_{G}pt)$$ is an isomorphism.
\end{theorem}
\begin{corollary}\label{corfjccoff}
Let $\phi : K\to G$ be a group homomorphism. Let $\mathcal{F}$ be a family of subgroups of $G$. Suppose that for every additive category $\mathcal{A}$ with $G$-action the assembly map $$H^{G}_{n}(E_{\mathcal{F}_{\mathcal{VCYC}}}(G);{\bf{K}}_{\mathcal{A}})\to H^{G}_{n}(pt;{\bf{K}}_{\mathcal{A}})$$
is injective. Then for every additive category $\mathcal{C}$ with $K$-action the assembly map $$H^{G}_{n}(E_{\phi^{*}\mathcal{F}_{\mathcal{VCYC}}}(G);{\bf{K}}_{\mathcal{C}})\to H^{G}_{n}(pt;{\bf{K}}_{\mathcal{C}})$$ is injective. The same statement holds with injectivity replaced by surjectivity in assumption and conclusion.
\end{corollary}
\begin{remark}\label{fjwithcore2}\rm{
Recall that the fibered version of the Farrell-Jones Conjecture \ref{fibcj} in algebraic $K$-theory for a group $G$ (and a ring $R$) can be formulated as follows: for every group homeomorphism $\phi:K\to G$ the assembly map $$H^{K}_{*}(E_{\phi^{*}\mathcal{F}_{\mathcal{VCYC}}}(K);{\bf{K}}_{R})\to H^{K}_{*}(pt;{\bf{K}}_R)$$ is an isomorphism (see Remark \ref{ficrema}). Therefore by Corollary \ref{corfjccoff} the Farrell-Jones Conjecture \ref{kfjciwthcof} implies the Fibered Farrell-Jones Conjecture \ref{fibcj}.}
\end{remark}
There is a functor $\mathbb{L}^{-\infty}: Add~ Cat~ Inv \to Sp$ that associates the $L$-theory spectrum to an additive category with involution constructed by Ranicki \cite{Ran92}. We consider the $Or(G)$-spectrum ${\bf{L}}_{\mathcal{A}}$ defined by ${\bf{L}}_{\mathcal{A}}(T)=\mathbb{L}^{-\infty}(\mathcal{A}*_{G}T)$. 
\begin{conjecture}\label{lfjciwthcof}(L-Theory Farrell-Jones-Conjecture with Coefficients)
Let $G$ be a group and let $\mathcal{F}_{\mathcal{VCYC}}$  be the family of virtually cyclic subgroups of $G$. Let $\mathcal{A}$ be an additive category with a right $G$-action. Then the assembly map $$H^{G}_{*}(E_{\mathcal{F}_{\mathcal{VCYC}}}(G);{\bf{L}}_{\mathcal{A}})\to H^{G}_{*}(pt;{\bf{L}}_{\mathcal{A}})$$ is an isomorphism.
\end{conjecture}
\begin{remark}\rm{
The only property of the functor $\mathbb{K}^{-\infty}$ that was used in the proof of Corollary \ref{corfjccoff} is that it sends equivalences of categories to equivalences of spectra. Because this property holds also for the functor $\mathbb{K}^{-\infty}$ there is also the $L$-theory version of Corollary \ref{corfjccoff}. Therefore there are also $L$-theory versions of Corollary \ref{corfjccoff} and  Remark \ref{fjwithcore2}.}
\end{remark}
The existing proofs for results about the Farrell-Jones Conjecture without coefficients can often be carried over to the context with
coefficients.  The following is a generalization of the main theorem in \cite{BR05}:
\begin{theorem}
Let $G$ be the fundamental group of a closed Riemannian manifold of strictly negative sectional curvature. Then the algebraic $K$-theory Farrell-Jones Conjecture with Coefficients \ref{kfjciwthcof} holds for $G$.
\end{theorem}
The following is a generalization of the main result from \cite{Bar03}:
\begin{theorem}
Let $G$ be a group of finite asymptotic dimension that admits a finite model for the classifying space $BG$. Let $\mathcal{A}$ be an additive category with right $G$-action. Then the assembly map $$H^{G}_{*}(EG;{\bf{K}}_{\mathcal{A}})\to H^{G}_{*}(pt;{\bf{K}}_{\mathcal{A}})$$ is split injective.
\end{theorem}
\begin{remark}\rm{
\indent
\begin{itemize}
\item[\bf{1.}]A. Bartels and W. L$\ddot{u}$ck \cite{BL12} verified Borel Conjecture \ref{borel} for considerably beyond the world of Riemannian manifolds of non-positive curvature. In particular, they proved the Borel Conjecture for closed aspherical manifolds of dimension $\geq5$, whose fundamental group is hyperbolic in the sense of Gromov \cite{BH99, Gro87} or is non-positively curved in the sense, that it admits a cocompact isometric proper action on a finite dimensional CAT(0)-space.
\item[\bf{2.}] Recall that the $K$-theoretic Farrell-Jones Conjecture (with coefficients in an arbitrary additive category) for hyperbolic groups has been proven by Bartels-L$\ddot{\rm u}$ck-Reich in \cite{BLR08a}(see Theorem \ref{hyfjwc}).  A. Bartels and W. L$\ddot{\rm u}$ck extended Theorem \ref{hyfjwc} to the $L$-theoretic Farrell-Jones Conjecture and (apart from higher $K$-theory) to CAT(0)-groups \cite{BL12}. Here are the results :
\end{itemize}}
\end{remark}
\begin{definition}(The class of groups $\mathcal{B}$). 
Let $\mathcal{B}$ be the smallest class of groups satisfying the following conditions :
\begin{itemize}
\item[(i)] Hyperbolic groups belong to $\mathcal{B}$;
\item[(ii)] If $G$ acts properly cocompactly and isometrically on a finite-dimensional CAT(0)-space, then $G\in \mathcal{B}$;
\item[(iii)] The class $\mathcal{B}$ is closed under taking subgroups;
\item[(iv)] Let $\pi:G \to H$ be a group homomorphism. If $H \in \mathcal{B}$ and $\pi^{-1}(V)\in \mathcal{B}$ for all virtually cyclic subgroups $V$ of $H$, then $G \in \mathcal{B}$;
\item[(v)] $\mathcal{B}$ is closed under finite direct products;
\item[(vi)] $\mathcal{B}$ is closed under finite free products;
\item[(vii)] The class $\mathcal{B}$ is closed under directed colimits.
\end{itemize}
\end{definition}
\begin{remark}\rm{
A group is said to be a CAT(0)-group if it acts geometrically, i.e., properly discontinuously and cocompactly by isometries, on a CAT(0)-space. One should think of a CAT(0)-space as a geodesic metric space in which every geodesic triangle is atleast as thin as its comparison
triangle in Euclidean plane. For basic facts about CAT(0)-spaces and groups a general reference is \cite{BH99}.}
\end{remark}
A. Bartels and W. L$\ddot{u}$ck \cite{BL12} proved the following result :
\begin{theorem}\label{luck}
Let $M$ be a closed aspherical manifold of dimension $\geq5$. If $\pi_1(M)\in \mathcal{B}$, then $M$ is topologically rigid.
\end{theorem}
\begin{remark}\rm{
\indent
\begin{itemize}
\item[\bf{1.}] A. Bartels and W. L$\ddot{\rm u}$ck prove the Farrell-Jones Conjecture about the algebraic $K$- and $L$-theory of group rings which does imply the claim appearing in Theorem \ref{luck} by surgery theory. Theorem \ref{luck} above remains true in dimension four if one additionally assumes that the fundamental group is good in the sense of Freedman \cite{Fre83}.
\item[\bf{2.}] A Coxeter system $(W, S)$ is a group $W$ together with a fundamental set $S$ of generators, see Definition \ref{coxeter}. Associated to the Coxeter system $(W, S)$ is a simplicial complex $K$ with a metric \cite[Chapter 7]{Dav08} and a proper isometric $W$-action. Moussong \cite{Mou87} showed that $K$ is a CAT(0)-space, see also \cite[Theorem 12.3.3]{Dav08}. In particular, if $K$ is finite dimensional and the action is cocompact, then $W$ is finite dimensional CAT(0)-group and belongs to $\mathcal{B}$. This is in particular the case if $S$ is finite. If $S$ is infinite, then any finite subset $S_0\subset S$ generates a Coxeter group $W_0$, see \cite[Theorem 4.1.6]{Dav08}. Then $W_0$ belongs to $\mathcal{B}$ and so does $W$ as it is the colimit of the $W_0$. Therefore Coxeter groups belong to $\mathcal{B}$.
\item[\bf{3.}] Recall that Davis constructed for every $n\geq 4$ closed aspherical manifolds whose universal cover is not homeomorphic to Euclidean space (see Theorem \ref{davis-cocompact}). In particular, these manifolds do not support metrics of non-positive sectional curvature. The fundamental groups of these examples are finite index subgroups of Coxeter groups $\Gamma$. Thus these fundamental groups lie in $\mathcal{B}$ and Theorem \ref{luck} implies that Davis' examples are topological rigid (if the dimension is atleast 5).
\item[\bf{4.}] Davis and Januszkiewicz used Gromov's hyperbolization technique to construct further exotic aspherical manifolds. They showed that for every $n\geq 5$ there are closed aspherical $n$-dimensional manifolds whose universal cover is a CAT(0)-space whose fundamental group at infinity is non-trivial (\cite[Theorem 5b.1]{DJ91}). In particular, these universal covers are not homeomorphic to Euclidean space. Because these examples are in addition non-positively curved polyhedron, their fundamental groups are finite-dimensional CAT(0)-groups and belong to $\mathcal{B}$. There is a variation of this construction that uses the strict hyperbolization of Charney-Davis \cite{CD95} and produces closed aspherical manifolds whose universal cover is not homeomorphic to Euclidean space and whose fundamental group is hyperbolic. All these examples are topologically rigid by Theorem \ref{luck}.
\item[\bf{5.}] Limit groups as they appear for instance in \cite{Sel01} have been in the focus of geometric group theory for the last years. Expositions about limit groups are for instance \cite{CG05} and \cite{Pau04}. Alibegovic-Bestvina have shown that limit groups are CAT(0)-groups \cite{AB06}. A straight forward analysis of their argument shows, that limit groups are finite dimensional CAT(0)-groups and belong therefore to class $\mathcal{B}$.
\item[\bf{6.}] If a locally compact group $L$ acts properly cocompactly and isometrically on a finite dimensional CAT(0)-space, then the same is true for any discrete cocompact subgroup of $L$. Such subgroups therefore belong to $\mathcal{B}$. For example, let $G$ be a reductive algebraic group defined over a global field $\mathbb{K}$ whose $\mathbb{K}$-rank is 0. Let $S$ be a finite set of places of $\mathbb{K}$ that contains the infinite places of $\mathbb{K}$. The group $G_S := \prod_{v\in S}G(\mathbb{K}_v)$ admits an isometric proper cocompact action on a finite dimensional CAT(0)-space, see for example \cite[pp. 40]{Ji07}. Because S-arithmetic subgroups of $G(\mathbb{K})$ can be realized (by the diagonal embedding) as discrete cocompact subgroups of $G_S$ (see for example \cite{Ji07}), these S-arithmetic groups belong to $\mathcal{B}$.
\item[\bf{7.}] Finitely generated virtually abelian groups are finite dimensional CAT(0)-groups and belong to $\mathcal{B}$. A simple induction shows that this implies that all virtually nilpotent groups belong to $\mathcal{B}$, compare the proof of \cite[Lemma 1.13]{BLR08}. All these examples are topologically rigid by Theorem \ref{luck}.
\end{itemize}}
\end{remark}
\begin{theorem}\label{catfjco}
Let $G\in \mathcal{B}$.
\begin{itemize}
\item[(i)] The $K$-theoretic assembly map in Conjecture \ref{kfjciwthcof} is bijective in degree $n\leq 0$ and surjective in degree $n = 1$ for any additive $G$-category $\mathcal{A}$;
\item[(ii)] The $L$-theoretic Farrell-Jones assembly map in Conjecture \ref{lfjciwthcof} with coefficients in any additive $G$-category $\mathcal{A}$ with involution is an isomorphism.
\end{itemize}
\end{theorem}
\begin{remark}\rm{
\indent
\begin{itemize}
\item[\bf{1.}]For virtually abelian groups Quinn \cite{Qui05} proved that the assembly map in Conjecture \ref{kfjciwthcof} is an isomorphism for all $n$ (more precisely in \cite{Qui05} only the untwisted case is considered: $\mathcal{A}$ is the category of finitely generated free $R$-modules for some ring $R$). 
\item[\bf{2.}]The class $\mathcal{B}$ contains in particular directed colimits of hyperbolic groups. The $K$-theory version of the Farrell-Jones Conjecture holds in all degrees for directed colimits of hyperbolic groups \cite[Theorem 0.8 (i)]{BEL08}. Thus, Theorem \ref{catfjco} implies that the Farrell-Jones Conjecture in $K$- and $L$-theory hold for directed colimits of hyperbolic groups. This class of groups contains a number of groups with unusual properties. Counterexamples to the Baum-Connes Conjecture with coefficients are groups with expanders \cite{HLS02}. The only known construction of such groups is as directed colimits of hyperbolic groups (see \cite{AD08}). Thus the Farrell-Jones Conjecture in $K$-and $L$-theory holds for the only at present known counterexamples to the Baum-Connes Conjecture with coefficients. The class of directed colimits of hyperbolic groups contains for instance a torsion-free non-cyclic group all whose proper subgroups are cyclic constructed by Ol'shanskii \cite{Ols79}. Further examples are 
mentioned in \cite[ pp. 5]{OOS07} and \cite[Section 4]{Sap07} These later examples all lie in the class of lacunary groups. Lacunary groups can be characterized as certain colimits of hyperbolic groups.
\item[\bf{3.}] Next we explain the relation between Theorem \ref{catfjco} and Theorem \ref{luck} :
\end{itemize}}
\end{remark}
\begin{theorem}\label{catlfjc}
Let $G$ be a torsion-free group. Suppose that the $K$-theoretic assembly map
$$H^{G}_m(E_{\mathcal{VCYC}}(G),{\bf{K}}_{\mathbb{Z}})\mapsto K_m(\mathbb{Z}G)$$
is an isomorphism for $m\leq 0$ and surjective for $m = 1$ and that the $L$-theoretic assembly map
$$H^{G}_m(E_{\mathcal{VCYC}}(G),{\bf{L}}_{\mathbb{Z}}^{<-\infty>})\mapsto L_{m}^{<-\infty>}(\mathbb{Z}G)$$
is an isomorphism for all $m\in \mathbb{Z}$, where we allow a twisting by any homomorphism $w: G\to \{\pm 1\}$. Then the following holds:
\begin{itemize}
\item[(i)] The assembly map
\begin{equation}\label{fj1}
 H_n(BG,{\bf{L}}_{\mathbb{Z}}^{s})\mapsto L_{n}^{s}(\mathbb{Z}G)
\end{equation}
is an isomorphism for all $n$;
\item[(ii)] The Borel Conjecture \ref{borel} is true in dimension $\geq 5$, i.e., if $M$ and $N$ are closed aspherical manifolds of dimensions $\geq 5$ with $\pi_1(M)\cong \pi_1(N)\cong G$, then $M$ and $N$ are homeomorphic and any homotopy equivalence $M\mapsto N$ is homotopic to a homeomorphism (This is also true in dimension 4 if we assume that $G$ is good in the sense of Freedman (see \cite{Fre82,Fre83});
\item[(iii)] Let $X$ be a finitely dominated Poincar$\grave{\rm e}$ complex of dimension $\geq 6$ with $\pi_1(X)\cong G$. Then $X$ is homotopy equivalent to a compact ANR-homology manifold.
\end{itemize}
\end{theorem}
\paragraph{Sketch of the proof:}
Claim (i): Because $G$ is torsion-free and $\mathbb{Z}$ is regular, the above assembly maps are equivalent to the maps
\begin{equation}\label{fj2}
 H_m(BG,{\bf{K}}_{\mathbb{Z}})\mapsto K_m(\mathbb{Z}G)
 \end{equation}
 \begin{equation}\label{fj3}
 H_m(BG,{\bf{L}}_{\mathbb{Z}}^{<-\infty>})\mapsto L_{m}^{<-\infty>}(\mathbb{Z}G)
\end{equation}
compare \cite[pp. 685, Proposition 2.2]{LR05}. Because (\ref{fj2}) is bijective for $m\leq 0$ and surjective for $m = 1$, we have $Wh(G) = 0$, $\widetilde{K}_0(\mathbb{Z}G) = 0$ and $K_i(\mathbb{Z}G) = 0$ for $i < 0$, compare \cite[pp. 653, Conjecture 1.3 and pp. 679, Remark 2.5]{LR05}. This implies that (\ref{fj3})  is equivalent to (\ref{fj1}), compare \cite[pp. 664, Proposition 1.5]{LR05}.\\
Claim (ii) : The Borel Conjecture for a group $G$ is equivalent to the statement that for every closed aspherical manifold $M$ with $G\cong \pi_1(M)$ its topological structure set $\mathcal{S}^{top}(M)$ consists of a single element, namely, the class of $id: M\to M$.
This follows from (i) and the algebraic surgery exact sequence of Ranicki which agrees for an $n$-dimensional manifold for $n\geq 5$ with the Sullivan-Wall geometric exact surgery sequence (see Theorem \ref{ranicki}).\\
Claim (iii) : See \cite[pp. 297, Remark 25.13]{Ran92}, \cite[pp. 439, Main Theorem and Section 8]{BFMW96} and \cite[Theorem A and Theorem B]{BFMW07}.
\begin{remark}\rm{
\indent
\begin{itemize}
\item[\bf{1.}] The assembly maps appearing in the Theorem \ref{catlfjc} above are special cases of the assembly maps in Conjecture \ref{kfjciwthcof} and Conjecture \ref{lfjciwthcof}. In particular, Theorem \ref{luck} follows from Theorem \ref{catfjco} and the above Theorem \ref{catlfjc}.
\item[\bf{2.}] A. Bartels and W. L$\ddot{\rm u}$ck \cite{BL12} gave a number of further important applications of Theorem \ref{catfjco}, which can be summarized as follows: The Novikov Conjecture and the Bass Conjecture hold for all groups $G$ that belong to $\mathcal{B}$. If $G$ is torsion-free and belongs to $\mathcal{B}$, then the Whitehead group $Wh(G)$ of $G$ is trivial, $\widetilde{K}_0(RG)= 0$ if $R$ is a principal ideal domain, and $K_n(RG)= 0$ for $n\leq -1$ if $R$ a regular ring. Furthermore the Kaplansky Conjecture holds for such $G$. These and further applications of the Farrell-Jones Conjecture are discussed in detail in \cite{BLR08} and \cite{LR05}. We remark that Hu \cite{Hu93} proved that if $G$ is the fundamental group of a finite polyhedron with non-positive curvature, then $Wh(G) = 0$, $\widetilde{K}_0(\mathbb{Z}G)= 0$ and $K_n(\mathbb{Z}G)= 0$ for $n\leq -1$.
\end{itemize}}
\end{remark}
C. Wegner proved the $K$-theoretic Farrell-Jones conjecture with (twisted) coefficients for CAT(0)-groups \cite{Weg10}. Here are the results:
\begin{definition}\rm{
A strong homotopy action of a group $G$ on a topological space $X$ is a continuous map $$\Psi :\coprod_{j=0}^{\infty}((G\times [0,1])^j\times G\times X)\to X$$ with the following properties:
\begin{itemize}
\item[(1)] $\Psi(....,g_{l},0,g_{l-1},....)=\Psi(....,g_{l},\Psi(g_{l-1},....))$
\item[(2)] $\Psi(....,g_{l},1,g_{l-1},....)=\Psi(....,g_{l}.g_{l-1},....)$
\item[(3)] $\Psi(e,t_{j},g_{j-1},....)=\Psi(g_{j-1},....)$
\item[(4)] $\Psi(....,t_{l},e,t_{l-1},....)=\Psi(....,t_{l}.t_{l-1},....)$
\item[(5)] $\Psi(......,t_{1},e,x)=\Psi(.....,x)$
\item[(6)] $\Psi(e,x)=x$
\end{itemize}}
\end{definition}
\begin{definition}\rm{
Let $\Psi$ be a strong homotopy $G$-action on a metric space $(X, d_X)$. Let $S\subseteq G$ be a finite symmetric subset which contains the trivial element $e\in G$. Let $k\in \mathbb{N}$ be a natural number.
\begin{itemize}
\item[(1)] For $g\in G$ we define $F_g(\Psi, S, k)\subset map(X,X)$ by $$F_g(\Psi, S, k):=\{ \Psi(g_{k},t_{k},...g_{0},?):X\to X ~~| g_{i}\in S, t_{i}\in [0,1], g_{k}....g_{0}=g\}.$$
\item[(2)]For $(g, x)\in G\times X$ we define $S^1_{\Psi,S,k}(g, x)\subset G\times X$ as the subset consisting of all $(h, y)\in G\times X$ with the following property: There are $a, b\in S$, $f\in F_a(\Psi, S, k)$ and $\widetilde{f}\in F_b(\Psi, S, k)$ such that $f(x) =\widetilde{f}(y)$ and $h= ga^{-1}b$. For $n\in \mathbb{N}^{\geq2}$ we set $$S^{n}_{\Psi,S,k}(g,x):=\{S^1_{\Psi,S,k}(h, y)|(h,y)\in S^{n-1}_{\Psi,S,k}(g,x)\}.$$
\end{itemize}}
\end{definition}
\begin{definition}\rm{
Let $\mathcal{F}$ be a family of subgroups of $G$. The group $G$ is called strongly transfer reducible over $\mathcal{F}$ if there exists a natural number $N\in \mathbb{N}$ with the following property: For every finite symmetric subset $S\subseteq G$ containing the trivial element $e\in G$ and every natural numbers $k$, $n \in \mathbb{N}$ there are
\begin{itemize}
\item[(i)] a compact contractible controlled $N$-dominated metric space $X$,
\item[(ii)] a strong homotopy $G$-action $\Psi$ on $X$ and
\item[(iii)] a cover $\mathcal{U}$ of $G\times X$ by open sets such that
\end{itemize}
\begin{itemize}
\item[(a)] $\mathcal{U}$ is an open $\mathcal{F}$-cover,
\item[(b)] $\dim(\mathcal{U})\leq N$,
\item[(c)] for every $(g, x)\in G\times X$ there exists $U\in \mathcal{U}$ with $S^{n}_{\Psi,S,k}(g,x)\subseteq U$.
\end{itemize}}
\end{definition}
\begin{theorem}\label{kcatfjciwthcof}
Let $G$ be a group which is strongly transfer reducible over a family $\mathcal{F}$ of subgroups of $G$. Let $\mathcal{A}$ be an additive category with a right $G$-action. Then the assembly map $$H^{G}_{m}(E_{\mathcal{F}}(G);{\bf{K}}_{\mathcal{A}})\to H^{G}_{m}(pt;{\bf{K}}_{\mathcal{A}})\cong K_{m}(\mathcal{A}*_{G}pt)$$ is an isomorphism for all $m\in \mathbb{Z}$. 
\end{theorem}
\begin{theorem}\label{hycatgrou}
Hyperbolic groups and CAT(0)-groups are strongly transfer reducible over the family of virtually cyclic subgroups.
\end{theorem}
Following the proof of \cite[Lemma 2.3]{BL12a} we see that Theorem \ref{kcatfjciwthcof} and Theorem \ref{hycatgrou} imply :
\begin{corollary}
Let $G_1$, $G_2$ be groups which satisfy the $K$-theoretic Farrell-Jones Conjecture \ref{kfjciwthcof}. Then the groups $G_1\times G_2$ and $G_1*G_2$ satisfy the $K$-theoretic Farrell-Jones Conjecture \ref{kfjciwthcof} too.
\end{corollary}
\begin{remark}\rm{
\indent
\begin{itemize}
\item[\bf{1.}] Theorem \ref{kcatfjciwthcof} extends the $K$-theoretic result in Theorem \ref{catfjco} for CAT(0)-groups to all dimensions.
\item[\bf{2.}] A. Bartels, F.T. Farrell and W. L$\ddot{\rm u}$ck proved the K-theoretic Farrell-Jones Conjecture (up to dimension one) and the $L$-theoretic Farrell-Jones Conjecture \ref{lfjciwthcof} for cocompact lattices in virtually connected Lie groups [BFL11]:
\end{itemize}}
\end{remark}
\begin{theorem}\label{vitpolyfj}(Virtually poly-$\mathbb{Z}$-groups)
Let $G$ be a virtually poly-$\mathbb{Z}$-group. Then both the $K$-theoretic and the $L$-theoretic Farrell-Jones Conjecture \ref{kfjciwthcof} and Conjecture \ref{lfjciwthcof} hold for $G$.
 \end{theorem}
 \begin{definition}\label{fjcnupto}\rm{
Let $G$ be a group and let $\mathcal{F}$ be the family of subgroups. Then $G$ satisfies the $K$-theoretic Farrell-Jones Conjecture with additive categories as coefficients with respect to $\mathcal{F}$ up to dimension one if for any additive
$G$-category $\mathcal{A}$ the assembly map $$H^{G}_{n}(E_{\mathcal{F}}(G);{\bf{K}}^{<-\infty>}_{\mathcal{A}})\to H^{G}_{n}(pt;{\bf{K}}^{<-\infty>}_{\mathcal{A}})\cong K^{<-\infty>}_{n}(\mathcal{A}*_{G}pt)$$ induced by the projection $E_{\mathcal{F}}(G)\mapsto pt$ is bijective for all $n\leq0$ and surjective for $n=1$.}
 \end{definition}
 \begin{theorem}\label{fjcolat}(Cocompact lattices in virtually connected Lie groups)
Let $G$ be a cocompact lattice in a virtually connected Lie group. Then both the $K$-theoretic Farrell-Jones Conjecture \ref{fjcnupto} with respect to the family $\mathcal{F}_{\mathcal{VCYC}}$ and the $L$-theoretic Farrell-Jones Conjecture \ref{lfjciwthcof} hold for $G$.
\end{theorem}
\begin{corollary}\label{corfjcolat}(Fundamental groups of 3-manifolds)
Let $\pi$ be the fundamental group of a 3-manifold (possibly non-compact, possibly non-orientable and possibly with boundary). Then both the $K$-theoretic Farrell-Jones Conjecture \ref{fjcnupto} with respect to the family $\mathcal{F}_{\mathcal{VCYC}}$ and the $L$-theoretic Farrell-Jones Conjecture \ref{lfjciwthcof} hold for $\pi$.
\end{corollary}
\begin{remark}\rm{
Recall that Theorem \ref{kcatfjciwthcof} extends the $K$-theoretic result in Theorem \ref{catfjco} for CAT(0)-groups to all dimensions. Using Theorem \ref{kcatfjciwthcof} it is possible to drown to dimension one in Theorem \ref{fjcolat} and Corollary \ref{corfjcolat}. So using Theorem \ref{kcatfjciwthcof} we get the full $K$-theoretic Farrell-Jones Conjecture in these cases.}
\end{remark}
\begin{remark}\rm{
Recall that Farrell and Linnell proved in Theorem \ref{fjcrysolv} that if the fibered isomorphism conjecture is true for all nearly crystallographic groups, then it is true for all virtually solvable groups. However, they were not able to verify the fibered isomorphism conjecture for all nearly crystallographic groups. In particular, they pointed out that the fibered isomorphism conjecture has not been verified for the
group $\mathbb{Z}[\frac{1}{2}]\rtimes_{\alpha}\mathbb{Z}$ where $\alpha$ is multiplication by 2. Note this group is isomorphic to the Baumslag-Solitar group $BS(1,2)$. Recall that the Baumslag-Solitar group $BS(m, n)$ is defined by $\langle a,b |ba^mb^{-1}=a^n \rangle$ and all the solvable ones are isomorphic
to $BS(1,d)$. Note that $BS(m,n)\cong BS(n,m)\cong BS(-m,-n)$. Farrell and Xiaolei Wu \cite{FW13} proved the following result :}
\end{remark}
\begin{theorem}
The $K$-theoretic and $L$-theoretic Farrell-Jones Conjecture is true for the solvable Baumslag-Solitar groups with coefficients in an additive category.
\end{theorem}
\begin{remark}\rm{
Farrell and Xiaolei Wu \cite{FW13} pointed out that the Farrell-Jones Conjecture has not been verified for all Baumslag-Solitar groups. For example, we do
not know whether the Farrell-Jones conjecture is true for the group $BS(2, 3)$.}
\end{remark}
\begin{final Remarks}\rm{
 Recall we have given in Remark \ref{rem.borel} that the obvious smooth analogue of Borel's Conjecture is false. Namely, Browder had shown in \cite{Bro65} that it is false even in the basic case where $M$ is an $n$-torus. In fact, surgery theory shows that the manifolds $T^n$ and $T^n\#\Sigma^n$ $(n\geq 5)$ are not diffeomorphic when $\Sigma^n$ is an exotic sphere; although they are clearly homeomorphic. This uses three ingredients: Bieberbach's Rigidity Theorem \ref{bieberbachrig}, Farrell and Hsiang topological rigidity for $T^n\times I$ (see Remark \ref{fhrig}), and the (stable) parallelizability of the torus.
But when it is assumed that both $M$ and $N$ in Problem \ref{fund.que} are non-positively curved Riemannian manifolds, then smooth rigidity frequently happens. The most fundamental instance of this is an immediate consequence of Mostow's Rigidity Theorem \ref{mostow}. The problem of changing the smooth structure on closed locally symmetric spaces of noncompact type $M^n$ were considered in \cite{FJ89a, FJ94, AF03, AF04, Oku02} and the problem was solved in many cases by forming the connected sum $M^n\#\Sigma^n$ where $\Sigma^n$ is a homotopy sphere.\\

Using Mostow's Rigidity Theorem \ref{mostow}, Farrell-Jones Topological Rigidity Theorem \ref{gentoprigd} and Kirby and Siebenmann results \cite[Theorem 4.1, pp.25; Theorem 10.1, pp. 194]{KS77} together with the fundamental paper of Kervaire and Milnor \cite{KM63}, the problem of determining when connected sum with a homotopy sphere $\Sigma$ changing the differential structure on a closed locally symmetric space of noncompact type $M^n$ is essentially reduced to a (non-trivial) question about the stable homotopy group of $M^n$. The main result of Okun \cite[Theorem 5.1]{Oku01} gives a finite sheeted cover $\mathcal{N}^n$ of $M^n$ and a nonzero degree tangential map $f : \mathcal{N}^n\to M_u$ where $M_u$ is global dual twin of compact type of $M$. And it is also showed in \cite{Oku02} that $\mathcal{N}^n\times \mathbb{D}^{n+1}$ is diffeomorphic to a codimension 0-submanifold of the interior of $M_u\times \mathbb{D}^{n+1}$. This allows us to look at the above question via \cite[Theorem 3.6]{Oku02} on the specific manifold $M_u$ instead 
of the arbitrary closed locally symmetric space of noncompact type $M^n$. Note that $M_u$ is the  global dual twin of compact type of $M$. Since the dual symmetric spaces of real, complex, quarternionic or Cayley hyperbolic manifolds are the sphere, complex projective space, quaternionic projective space or Cayley projective plane respectively. In view of this, we can look at the problem of detecting exotic structure on sphere, complex projective space, quaternionic projective space or Cayley projective plane instead of the arbitrary closed $\mathbb{K}$-hyperbolic manifold where $\mathbb{K}=\mathbb{R}, \mathbb{C}, \mathbb{H}$ or $\mathbb{O}$ respectively. In the papers \cite{FJ89a, FJ94, AF03, AF04, Oku02}, the authors considered these observations to produce exotic smooth structures on a closed locally symmetric space of noncompact type.\\

As we have observed in Remark \ref{harm.Eells}, every homotopy equivalence $f:M\to N$ where $N$ is a closed non-positively curved Riemannian manifold is homotopic to a unique harmonic map $\phi$. It was conjectured by Lawson and Yau that $f$ is necessarily a diffeomorphism. F.T. Farrell, L.E. Jones and P. Ontaneda have earlier found celebrated counterexamples to this conjecture \cite{FJ89a, Ont94, FJO98}. But the possibility remained that $\phi$ might still be a homeomorphism. If so, it would provide another route to Borel’s Conjecture that homotopically equivalent closed aspherical manifolds are homeomorphic \cite{FJ91}. On the other hand, F.T. Farrell and L.E. Jones showed that the unique harmonic map $\phi:M\to N$ is not a homeomorphism even though $N$ is negatively curved. But in these examples it is unknown if $M$ can also be non-positively curved \cite{FJ96, FJO98a}.\\ \\
For more details about Smooth and PL rigidity Problem \ref{genequestion} for negatively curved manifolds and many interesting open problems along this direction, see survey article \cite{RA13}.}
\end{final Remarks}


\newpage
\begin{thebibliography}{1}
\bibitem [AB06]{AB06}
E. Alibegovi and M. Bestvina, Limit groups are $CAT(0)$, \emph{J. London Math. Soc.,} (2), \textbf{74}(1) (2006), 259-272.
\bibitem [AD08]{AD08}
G. Arzhantseva and T. Delzant, \emph{Examples of random groups,} Preprint, 2008.
\bibitem [AF03]{AF03}
C.S. Aravinda and F.T. Farrell,  Exotic negatively curved structures on Cayley hyperbolic manifolds, \emph{J. Differential Geom.,} \textbf{63} (2003), 41-62.
\bibitem [AF04]{AF04}
C.S. Aravinda and F.T. Farrell, Exotic structures and quaternionic hyperbolic manifolds, \emph{Algebraic groups and arithmetic,} 507-524, Tata Inst. Fund. Res., Mumbai, 2004.
\bibitem [Alb68]{Alb68}
S.I. Al'ber, Spaces of mappings into manifold of negative curvature, \emph{Dokl. Akad. Nauk. SSSR.,} \textbf{178} (1968), 13-16.
\bibitem [AMS97]{AMS97}
H. Abels, G.A. Margulis and G.A. Soifer, Properly discontinuous groups of affine transformations with orthogonal linear part, \emph{C. R. Acad. Sci. Paris Sér. I Math.,} \textbf{324} (1997), no. 3, 253-258.
\bibitem [AM90]{AM90}
S. Akbulut and J.D. McCarthy, \emph{Casson's invariant for oriented homology 3-spheres,} Math. Notes \textbf{36}, Princeton University Press, Princeton, NJ, 1990
\bibitem [Bar03]{Bar03}
A. Bartels, Squeezing and higher algebraic K-theory, \emph{K-Theory.,} \textbf{28} (1) (2003), 19-37.
\bibitem [BEL08]{BEL08}
A. Bartels, S. Echterhoff and W. L$\ddot{\rm u}$ck, Inheritance of isomorphism conjectures under colimits. In Cortinaz, Cuntz, Karoubi, Nest, and Weibel, editors, K-Theory and noncommutative geometry, EMS-Series of Congress Reports, 41-70. \emph{European Mathematical Society}, 2008.
\bibitem [BFJP00]{BFJP00}
E. Berkove, F.T. Farrell, D. Juan-Pineda and K. Pearson, The Farrell-Jones isomorphism conjecture for finite covolume hyperbolic actions and the algebraic K-theory of Bianchi groups, \emph{Trans. Amer. Math. Soc.,} 352 (2000), no. 12, 5689-5702.
\bibitem [BFJR04]{BFJR04}
A. Bartels, F.T. Farrell, L.E. Jones and H. Reich. On the isomorphism conjecture in algebraic K-theory, \emph{Topology.,} \textbf{43}(1) (2004), 157-213.
\bibitem [BFL11]{BFL11}
A. Bartels, F.T. Farrell, W. L$\ddot{\rm u}$ck, The Farrell-Jones Conjecture for cocompact lattices in virtually connected Lie groups, Jan 3, 2011, arXiv:1101.0469
\bibitem [BFMW96]{BFMW96}
J. Bryant, S. Ferry, W. Mio and S. Weinberger, Topology of homology manifolds, \emph{Ann. of Math.,} (2), \textbf{143}(3) (1996), 435-467.
\bibitem [BFMW07]{BFMW07}
J. Bryant, S. Ferry, W. Mio, and S. Weinberger, Desingularizing homology manifolds, \emph{Geom. Topol.,} 11:1289–1314, 2007.
\bibitem [BH99]{BH99}
M.R. Bridson and A. Haefliger, Metric spaces of non-positive curvature, Grundlehren der Mathematischen Wissenschaften [Fundamental
Principles of Mathematical Sciences], vol. 319, \emph{Springer-Verlag, Berlin}, 1999.
\bibitem [BHS64]{BHS64}
H. Bass, A. Heller and R.G. Swan. The Whitehead group of a polynomial extension. \emph{Inst. Hautes Études Sci. Publ. Math.,} no. 22 (1964), 61-79. 
\bibitem [Bie12]{Bie12}
L. Bieberbach, Uber die Bewegungsgruppen der Euklidischen R ̈$\ddot{\rm u}$me II, \emph{Math. Ann.,} \textbf{72} (1912), 400-412.
\bibitem [BL06]{BL06}
A. Bartels and W. L$\ddot{\rm u}$ck, Isomorphism conjecture for homotopy K-theory and groups acting on trees,\emph{ J. Pure Appl. Algebra}, \textbf{205} (2006), no. 3, 660-696.
\bibitem [BL07]{BL07}
A. Bartels and W. L$\ddot{\rm u}$ck, Induction theorems and isomorphism conjectures for $K$-and $L$-theory, \emph{Forum Math.,} \textbf{19} (2007), 379-406.
\bibitem [BLR08]{BLR08}
A. Bartels, W. L$\ddot{\rm u}$ck and H. Reich, On the Farrell-Jones conjecture and its applications, \emph{J. Topol.,} \textbf{1} (2008), no. 1, 57-86.
\bibitem [BLR08a]{BLR08a}
A. Bartels, W. L$\ddot{\rm u}$ck and H. Reich, The K-theoretic Farrell-Jones conjecture for hyperbolic groups, \emph{Invent. Math.,} \textbf{172} (2008), no. 1, 29-70.
\bibitem [BL12]{BL12}
A. Bartels and W. L$\ddot{\rm u}$ck, The Borel conjecture for hyperbolic and $CAT(0)$-groups, \emph{Ann. of Math.,} (2) \textbf{175} (2012), no. 2, 631-689.
\bibitem [BL12a]{BL12a}
A. Bartels and W. L$\ddot{\rm u}$ck, Geodesic flow for $CAT(0)$-groups, \emph{Geom. Topol.,} \textbf{16} (2012), no. 3, 1345-1391. 
\bibitem [Bou68]{Bou68}
N. Bourbaki, \emph{Groupes et Algebres de Lie, Chapters IV-VI}, Hermann, Paris, 1968.
\bibitem [Boy89]{Boy89}
N. Boyom, \emph{The lifting problem for affine structures in nilpotent Lie groups}, \emph{Trans. Amer. Math. Soc.,} 313 (1989), no. 1, 347-379.
\bibitem [Bre67]{Bre67}
E. Bredon, Equivariant cohomology theories, \emph{Bull. Amer. Math. Soc.,} \textbf{73} (1967), 266-268.
\bibitem [Bro60]{Bro60}
M. Brown, A proof of the generalized Schoenflies theorem, \emph{ Bull. Amer. Math. Soc.,} Vol \textbf{66}, no. 2 (1960), 74-76.
\bibitem [Bro65]{Bro65}
 W. Browder, \emph{On the action of $\Theta^{n}\rm{(\partial \pi)}$, Differential and Combinatorial Topology} (A Symposium in Honor of Marston Morse), 23-36, Princeton Univ. Press, Princeton, N.J., (1965).
 \bibitem [Bro65]{Bro65}
  W. Browder, 
 \bibitem [BR05]{BR05}
A. Bartels and H. Reich. On the Farrell-Jones conjecture for higher algebraic K-theory, \emph{J. Amer. Math. Soc.,} 18(3) (2005), 501-545.
\bibitem [BR07]{BR07}
A. Bartels and H. Reich, Coefficients for the Farrell-Jones conjecture, \emph{Adv. Math.,} \textbf{209} (2007), no. 1, 337-362. 
\bibitem [Cap76]{Cap76}
S.E. Cappell, On homotopy invariance of higher signatures, \emph{Invent. Math.,} \textbf{33} (1976), 171-196. 
\bibitem [Cas80]{Cas80}
J. Casson, \emph{Three lectures on new-infinite constructions in 4-dimensional manifolds,} 1980
\bibitem [CD95]{CD95}
R.M. Charney and M.W. Davis, Strict hyperbolization, \emph{Topology} \textbf{34}(2) (1995), 329-350.
\bibitem [CG05]{CG05}
C. Champetier and V. Guirardel, Limit groups as limits of free groups, \emph{Israel J. Math.,} \textbf{146} (2005), 1-75.
\bibitem [Cha74]{Cha74}
T.A. Chapman, Topological invariance of Whitehead torsion, \emph{Amer. J. Math.,} \textbf{96}, 488-497, 1974.
\bibitem [CHK00]{CHK00}
D. Cooper, C. Hodgson and S. Kerchoff, 3-dimensional orbifolds and cone-manifolds, \emph{Math. Soc. Japan Memoirs.,} Vol. \textbf{5}, Tokyo, 2000. 
\bibitem [CR77]{CR77}
P.E. Conner and F. Raymond, Deforming homotopy equivalences to homeomorphisms in aspherical manifolds, \emph{Bull. AMS.,} \textbf{83} (1977), 36-85.
\bibitem [Dav83]{Dav83}
M. Davis, Groups generated by reflections and aspherical manifolds not covered by Euclidean space, \emph{Ann. of Math.,} \textbf{117} (1983), 293-324.
\bibitem [Dav84]{Dav84}
M. Davis, Coxeter groups and aspherical manifolds, 
\emph{Algebraic Topology Aarhus 1982}, Lecture Notes in Mathematics Vol \textbf{1051}, 1984, 197-221.
\bibitem [Dav08]{Dav08}
M.W. Davis. The geometry and topology of Coxeter groups, \emph{London Mathematical Society Monographs Series,} Vol 32, Princeton University Press, Princeton, NJ, 2008.
\bibitem [DFL13]{DFL13}
M.W. Davis, J. Fowler and J.F. Lafont, Aspherical manifolds that cannot be triangulated, 2013, arXiv:1304.3730v2 [math.GT].
\bibitem [DH89]{DH89}
M. Davis and J.C. Hausmann, Aspherical manifolds without smooth or PL structure,  Lect. Notes in Math., Vol.1370, \emph{Springer-Verlag,} New York, 1989, 135-142.
\bibitem [DJW01]{DJW01}
M.W. Davis, T. Januszkiewicz and S. Weinberger, Relative hyperbolizations and aspherical bordisms, an addendum to Hyperbolization of polyhedra, \emph{J. Diff. Geom.,} \textbf{58}, (2001), 535-541.
\bibitem [DJ91]{DJ91}
M. Davis and T. Januszkiewic, Hyperbolization of polyhedra, \emph{J. Differential Geom.,} Vol \textbf{34}, No 2 (1991), 347-388.
\bibitem [DK01]{DK01}
J.F. Davis and P. Kirk, Lecture notes in algebraic topology, Graduate Studies in Mathematics, vol 35, \emph{American Mathematical Society}, Providence, RI, 2001.
\bibitem [DL98]{DL98}
J.F. Davis and W. L$\ddot{\rm u}$ck, Spaces over a category and assembly maps in isomorphism conjectures in $K$- and $L$-Theory, \emph{K-theory}., \textbf{15} (1998), 201-252 
\bibitem [Edw78]{Edw78}
R. Edwards, The topology of manifolds and cell like maps, \emph{Proc. Intemat. Congress Math.,} Helsinki, 1978, 111-128.
\bibitem [ES64]{ES64}
J. Eells and J.H. Sampson, Harmonic mappings of Riemannian manifolds, \emph{Amer. J. Math.,} \textbf{86} (1964), 109-160.
\bibitem [FG83]{FG83}
D. Fried and M. Goldman,  Three-dimensional affine crystallographic groups, \emph{ Adv. in Math.,} \textbf{47} (1983), no. 1, 1-49.
\bibitem [FH73]{FH73}
F.T. Farrell and W.C. Hsiang, Manifolds with $\pi_{1}= G\times_{\alpha} T$, \emph{Amer. J. Math.,} \textbf{95} (1973), 813-848. 
\bibitem [FH78]{FH78}
F.T. Farrell and W.C. Hsiang, The Topological-Euclidean Space Form Problem, \emph{Invent. math.,} \textbf{45} (1978), 181-192.
\bibitem [FH79]{FH79}
F.T. Farrell and W.C. Hsiang, Remarks on Novikov's conjecture and the topological-euclidean space form problem, \emph{Lecture Notes in Mathematics}, 1979, Vol \textbf{763}, 635-642.
\bibitem [FH81]{FH81}
F.T. Farrell and W.C. Hsiang, On Novikov's Conjecture for Non-Positively Curved Manifolds, I, \emph{Ann. of Math.,} \textbf{113} (1981), 199-209.
\bibitem [FH81a]{FH81a}
F.T. Farrell and W.C. Hsiang, The Whitehead group of poly-(finite or cyclic) groups, \emph{J. London Math. Soc.,} \textbf{14} (1981), 308-324.
\bibitem [FH83]{FH83}
F.T. Farrell and W.C. Hsiang, Topological Characterization of Flat and Almost Flat Riemannian Manifolds $M^n (n\neq 3, 4)$, \emph{Amer. J. Math.,} \textbf{105} (1983), 641-672.
\bibitem [FJ86]{FJ86}
F.T. Farrell and L.E. Jones, K-Theory and dynamics I, \emph{Ann. of Math.,} \textbf{2} 124 (1986), no. 3, 531-569. 
\bibitem [FJ88]{FJ88}
F.T. Farrell and L.E. Jones, The surgery L-groups of poly-(finite or cyclic)
groups, \emph{Invent. Math.,} \textbf{91} (1988), no. 3, 559-586.
\bibitem [FJ89]{FJ89}
F.T. Farrell and L.E. Jones, A topological analogue of Mostow’s rigidity theorem, \emph{Journal of the American Mathematical Society}, \textbf{2} (1989), 257-370.
\bibitem [FJ89a]{FJ89a}
F.T. Farrell and L.E. Jones, Negatively curved manifolds with exotic smooth structures,
 \emph{J. Amer. Math. Soc.,} \textbf{2} (1989), 899-908.
\bibitem [FJ91]{FJ91}
F.T. Farrell and L.E. Jones, Rigidity in geometry and topology, Proc. of the ICM, Vol.I, II (Kyoto, 1990), \emph{Math. Soc., Japan, Tokyo,} 1991, 653-663.
\bibitem [FJ93]{FJ93}
F.T. Farrell and L.E. Jones, Topological rigidity for compact nonpositively curved manifolds, \emph{Proc. Sympos. Pure Math.,} \textbf{54} (1993), 229-274.
\bibitem [FJ93a]{FJ93a}
F.T. Farrell and L.E. Jones, Isomorphism conjectures in algebraic K-theory, J. Amer.Math.Soc., \textbf{6}(1993), no. 2, 249-297.
\bibitem [FJ94]{FJ94}
F.T. Farrell and L.E. Jones, Complex hyperbolic manifolds and exotic smooth structures, \emph{Invent. Math.,} \textbf{117} (1994), no. 1, 57-74.
\bibitem [FJ96]{FJ96}
F. T. Farrell and L. E. Jones, Some non-homeomorphic harmonic homotopy equivalences, \emph{Bull. London Math. Soc.,} \textbf{28} (1996), no. 2, 177-182.
\bibitem [FJ98]{FJ98}
F. T. Farrell and L. E. Jones, Rigidity for aspherical manifolds with $\pi_1\subset GL_m(R)$, \emph{ Asian J. Math.,} \textbf{2} (1998), no. 2, 215-262.
\bibitem [FJO98]{FJO98}
F.T. Farrell, L.E. Jones, and P. Ontaneda, Hyperbolic manifolds with negatively curved exotic triangulations in dimensions greater than five, \emph{J. Differential Geom.,} \textbf{48} (1998), no. 2, 319-322.
\bibitem [FJO98a]{FJO98a}
F.T. Farrell, L.E. Jones, and P. Ontaneda, Examples of non-homeomorphic harmonic maps between negatively curved manifolds, \emph{Bull. London Math. Soc.,} \textbf{30} (1998), 295-296.
\bibitem [FL03]{FL03}
F.T. Farrell and P. A. Linnell, $K$-theory of solvable groups, \emph{ Proc. London Math. Soc.,} (3) 87 (2003), no. 2, 309-336. 
\bibitem [FM12]{FM12}
B. Farb and D. Margalit, \emph{A primer on mapping class groups}, Princeton Mathematical Series, 49, Princeton University Press, Princeton, NJ, 2012.
\bibitem [FQ90]{FQ90}
M.H. Freedman and F. Quinn, \emph{Topology of 4-manifolds,} Vol. 39 of Princeton Mathematical Series, Princeton University
Press, Princeton, NJ, 1990.
\bibitem [Fre82]{Fre82}
M.H. Freedman, The topology of four-dimensional manifolds,  \emph{J. Differential Geom.,} Vol \textbf{17}, No 3 (1982), 357-453. 
\bibitem [Fre83]{Fre83}
M.H. Freedman. The disk theorem for four-dimensional manifolds, \emph{in Proceedings of the International Congress of Mathematicians,} Vol. 1, 2 (Warsaw, 1983), pages 647-663, Warsaw, 1984. PWN.
\bibitem [FS08]{FS08}
J. Frazier and A. Sanders. \emph{The Dehn-Nielsen-Baer Theorem}, October 2008.
\bibitem [FW13]{FW13}
F.T. Farrell and Xiaolei Wu, Farrell-Jones Conjecture for the solvable Baumslag-Solitar groups, 2013, Available at http://arxiv.org/abs/1304.4779
\bibitem [Gab94]{Gab94}
D. Gabai, On the Geometric and Topological Rigidity of Hyperbolic 3-Manifolds, \emph{Bull. Amer. Math. Soc.,} \textbf{31} (1994), 228-232. 
\bibitem [Gro87]{Gro87}
M. Gromov,  Hyperbolic groups,  in: Essays in Group Theory, 75-263, \emph{Math. Sci. Res. Inst. Publ. 8. Springer.,} 1987.
\bibitem [GRT03]{GRT03}
D. Gabai, G. Robert Meyerhoff and N. Thurston,  Homotopy hyperbolic 3-manifolds are hyperbolic, \emph{Ann. of Math.,} (2) \textbf{157} (2003), no. 2, 335-431.
\bibitem [GS77]{GS77}
D. Galewski and R. Stern, A universal 5-manifold with respect to simplicial triangulations, \emph{in Geometric Topology (Proc. Georgia Topology
Conf., Athens, Ga., 1977)}, 345-350, Academic Press, New York-London, 1979.
\bibitem [GW71]{GW71}
D. Gromoll and J. Wolf, Some relations between the metric structure and the algebraic structure of the fundamental group in manifolds of nonpositive curvature, \emph{Bull. Amer. Math. Soc.,} \textbf{77} (1971) 545-552.
\bibitem [Har67]{Har67}
Hartman, P,  On homotopic harmonic maps, \emph{Canad. J. Math.,} \textbf{19} (1967), 673-687.
 \bibitem [Hem76]{Hem76}
J. Hempel, \emph{3-Manifolds}, Princeton University Press, Princeton, N. J., 1976. 
\bibitem [Hir58]{Hir58}
F. Hirzebruch, Automorphe Formen und der Satz von Riemann-Roch, 1958 Symposium internacional de topología algebraica, \emph{Universidad Nacional Autónoma de México and UNESCO,} Mexico City,(1958) 129-144.
\bibitem [Hir71]{Hir71}
F. Hirzebruch, The signature theorem: reminiscences and recreation, Prospects in mathematics (Proc. Sympos., Princeton Univ., Princeton, N.J., 1970), \emph{Ann. of Math. Studies,} no. 70, Princeton Univ. Press, Princeton, N.J. (1971) 3-31. 
\bibitem [HLS02]{HLS02}
N. Higson, V. Lafforgue, and G. Skandalis, Counterexamples to the Baum-Connes conjecture, \emph{Geom. Funct. Anal.,} \textbf{12}(2) (2002), 330-354.
\bibitem [HS70]{HS70}
W.C. Hsiang and J. L. Shaneson, \emph{Fake tori, Topology of Manifolds,} Markham, Chicago, 1970, 18-51.
\bibitem [Hu93]{Hu93}
B.Z. Hu. Whitehead groups of finite polyhedra with nonpositive curvature, \emph{J. Differential Geom.,} \textbf{38} (3) (1993), 501-517.
\bibitem [Ji07]{Ji07}
L. Ji, The integral Novikov conjectures for S-arithmetic groups. I., \emph{K-Theory.,} \textbf{38} (1) (2007), 35-47.
\bibitem [Joh71]{Joh71}
F.E.A. Johnson, Manifolds of homotopy type  $K(\pi,1)$. I, \emph{Mathematical Proceedings of the Cambridge Philosophical Society,} Vol \textbf{70}, November 1971, 387-393.
\bibitem [Joh74]{Joh74}
F.E.A. Johnson, Manifolds of homotopy type  $K(\pi,1)$. II, \emph{Mathematical Proceedings of the Cambridge Philosophical Society,} Vol \textbf{75}, March 1974, 165-173.
\bibitem [Kas88]{Kas88}
G.G. Kasparov, Equivariant K-theory and the Novikov conjecture, \emph{Invent. Math., 91}, no. 1, (1988), 147-201.
\bibitem [KL04]{KL04}
M. Kreck and W. L$\ddot{\rm u}$ck, \emph{The Novikov Conjecture}, Oberwolfach-Seminar, January 2004.
\bibitem [KL05]{KL05}
M. Kreck and W. L$\ddot{\rm u}$ck, The Novikov conjecture, \emph{Birkhäuser Verlag,} Basel, 2005. 
\bibitem [KL08]{KL08}
B. Kleiner and J. Lott, Notes on Perelman’s papers, \emph{Geom. Topol,}. \textbf{12} (2008), no. 5, 2587-2855.
\bibitem [KL09]{KL09}
M. Kreck and W. L$\ddot{\rm u}$ck, Topological rigidity for non-aspherical manifolds, \emph{Pure Appl. Math.,} \textbf{5} (2009), no.3, special issue, in honor of Friedrich Hirzebruch, 873-914.
\bibitem [Kne29]{Kne29}
H. Kneser, Geschlossene Flächen in dreidimensionalen Mannigfalligkeiten, \emph{Jahresber. Deutsch. Math.-Verein}, \textbf{38} (1929), 248-260.
\bibitem [KM63]{KM63}
M. Kervaire and J. Milnor, {\it Groups of homotopy spheres: I,} Annals of Math.,
\textbf{77} (1963), 504-537.
 \bibitem [KS77]{KS77}
R.C. Kirby and L.C. Siebenmann, Foundational Essays on Topological Manifolds, Smoothings, and
Triangulations, \emph{Annals of Math. Studies,} Princeton: University Press 1977.
\bibitem [Lee95]{Lee95}
B. Leeb, 3-manifolds with(out) metrics of nonpositive curvature, \emph{Invent. Math.,} \textbf{122} (1995), no.2, 277-289. 
\bibitem [LR75]{LR75}
R. Lee and F. Raymond, Manifolds covered by Euclidean space, \emph{Topology,} 14(1975), 49-57.
\bibitem [LR05]{LR05}
W. L$\ddot{\rm u}$ck and H. Reich, The Baum-Connes and the Farrell-Jones conjectures in K- and L- theory, In Handbook of K-theory. Vol. \textbf{1, 2}, 703-842. {Springer, Berlin,} 2005.
\bibitem [Luc02]{Luc02}
W.L$\ddot{\rm u}$ck, Chern characters for proper equivariant homology theories and applications to K- and L-theory, \emph{J. Reine Angew. Math.,} \textbf{543} (2002), 193-234. 
\bibitem [Luc05]{Luc05}
W. L$\ddot{\rm u}$ck, Survey on classifying spaces for families of subgroups, Infinite groups: geometric, combinatorial and dynamical aspects, 269–322, \emph{Progr. Math.,} 248, Birkhäuser, Basel, 2005. 
\bibitem [Man13]{Man13}
C. Manolescu, Pin(2)-equivariant Seiberg-Witten Floer homology and the Triangulation Conjecture, 2013, arXiv:1303.23554v2.
\bibitem [Mar83]{Mar83}
G.A. Margulis, Free totally discontinuous groups of affine transformations, \emph{Sov. Math. Dokl.,} \textbf{28}, no. 2, 1983.
\bibitem [Mil58]{Mil58}
J. Milnor, \emph{On simply connected 4-manifolds,} In Symposium internacional de topolog´ıa algebraica, Universidad Nacional Aut´onoma de M´exico and UNESCO, Mexico City, 1958, 122-128.
\bibitem [Mil62]{Mil62}
J. Milnor, A unique decomposition theorem for 3-manifolds, \emph{Amer. J. Math.,} \textbf{84} (1962).
\bibitem [Mil77]{Mil77}
J. Milnor, On fundamental groups of complete affinely flat manifolds, \emph{Adv. Math.,} \textbf{25} (1977), 178-187.
\bibitem [Mil03]{Mil03}
J. Milnor, Towards the Poincar$\grave{\rm e}$ Conjecture and the classification of 3-manifolds, \emph{Notices  AMS}, \textbf{50} (2003), 1226-1233. 
\bibitem [Mis74]{Mis74}
A.S. Miscenko, Infinite-dimensional representations of discrete groups and higher signature, \emph{Izv. Akad. Nauk. SSSR Ser. Mat.,} \textbf{38} (1974), 81-106.
\bibitem [Mos67]{Mos67}
G.D. Mostow,  Quasi-conformal mappings in n-space and the rigidity of hyperbolic space forms, \emph{Inst. Hautes Etudes Sci. Publ.,} \textbf{34} (1967), 53-104.
\bibitem [Mos73]{Mos73}
G.D. Mostow,  Strong rigidity of locally symmetric spaces, \emph{Ann. of Math. Studies,} Vol 78, Princeton University press, Princeton, NJ (1973).
\bibitem [Mou87]{Mou87}
G. Moussong, \emph{Hyperbolic Coxeter groups,} Ph.D.thesis, The Ohio State University, Mathematics, 1987.
\bibitem [MT07]{MT07}
J. Morgan and G. Tian, Ricci flow and the Poincar$\grave{\rm e}$ conjecture, \emph{Clay Mathematics Monographs, 3, American Mathematical Society,} Providence, RI; Clay Mathematics Institute, Cambridge,
MA, 2007.
\bibitem [MT08]{MT08}
J. Morgan and G. Tian, \emph{Completion of the Proof of the Geometrization Conjecture,} (2008). Available at the arXiv:0809.4040.
\bibitem [Nov64]{Nov64}
S.P. Novikov, Homotopically equivalent smooth manifolds, \emph{Izv. Akad. Nauk SSSR Ser Mat.,} \textbf{28} (1964), 365-474.
\bibitem [Nov65]{Nov65}
S.P. Novikov, Topological invariance of rational classes of Pontrjagin, \emph{Dokl. Akad. Nauk SSSR.,} 163 (1965), 298-300.
\bibitem [Nov65a]{Nov65a}
S.P. Novikov, Rational Pontrjagin classes. Homeomorphism and homotopy type of closed manifolds.I, \emph{Izv. Akad. Nauk SSSR Ser. Mat.,} \textbf{29} (1965), 1373-1388. 
\bibitem [Nov66]{Nov66}
S.P. Novikov, On manifolds with free abelian fundamental group and their application, \emph{Izv. Akad. Nauk SSSR Ser. Mat.,} \textbf{30} (1966), 207-246.
\bibitem [Nov70]{Nov70}
S.P. Novikov, Algebraic construction and properties of Hermitian analogues of K-theory over rings with involution from the viewpoint of Hamiltonian formalism. Applications to differenrial topology and the theory of characteristic classes II, \emph{Izv. Akad. Nauk. SSSR Ser. Mat.,} \textbf{34} (1970), 479-505. 
\bibitem [Oku01]{Oku01}
 B. Okun, Nonzero degree tangential maps between dual symmetric spaces, \emph{Algebraic and Geometric Topology,} \textbf{1} (2001), 709-718.
\bibitem [Oku02]{Oku02}
 B. Okun, Exotic smooth structures on nonpositively curved symmetric spaces, \emph{Algebraic and Geometric Topoology,} \textbf{2} (2002), 381-389.
\bibitem [Ols79]{Ols79}
A.Y. Ol'shanskii, An infinite simple torsion-free Noetherian group, \emph{Izv. Akad. Nauk SSSR Ser. Mat.,} \textbf{43}(6) (1979), 1328-1393.
\bibitem [Ont94]{Ont94}
P. Ontaneda, Hyperbolic manifolds with negatively curved exotic triangulations in dimension six, \emph{J. Differential Geom.,} \textbf{40} (1994), no. 1, 7-22. 
\bibitem [OOS07]{OOS07}
A.Y. Ol'shankskii, D. Osin, and M. Sapir, \emph{Lacunary hyperbolic groups,} arXiv:math.GR/0701365v1, 2007.
\bibitem [Pau04]{Pau04}
F. Paulin, Sur la théorie élémentaire des groupes libres (d'après Sela). (French) Astérisque no. 294 (2004), ix, 363-402. 
\bibitem [Pra73]{Pra73}
G. Prasad, Strong rigidity of Q-rank 1 lattices, \emph{Inventiones Mathematicae,} \textbf{21} (1973), 255-286.
\bibitem [PW85]{PW85}
E. Pedersen and C. Weibel, A non-connective delooping of algebraic K-theory, \emph{Topology,} Lecture Notes in Math., Vol. \textbf{1126}, Springer-Verlag, Berlin-Heidelberg-New York, (1985), 166-181.
\bibitem [Qui82]{Qui82}
F. Quinn, Ends of maps. II, \emph{Invent. Math.,} \textbf{68} (1982), no. 3, 353-424.
\bibitem [Qui05]{Qui05}
F. Quinn, \emph{Hyperelementary assembly for k-theory of virtually abelain groups}, Preprint, arXiv:math.KT/0509294, 2005
\bibitem [Ran92]{Ran92}
A.A. Ranicki, \emph{Algebraic L-theory and topological manifolds,} volume 102 of Cambridge Tracts in Mathematics. Cambridge University Press, Cambridge, 1992.
\bibitem [RA13]{RA13}
K. Ramesh, Smooth and PL-Rigidity Problems on Locally Symmetric Spaces (Submitted).
\bibitem [Ros94]{Ros94}
J. Rosenberg, Algebraic K-theory and its applications, \emph{Springer-Verlag,} New York, 1994.
\bibitem [Rou14]{Rou14}
S.K. Roushon, The isomorphism conjecture for groups with generalized free product structure \emph{'in'~Handbook of Group Action Vol II,} \textbf{ALM 32} (2014), 77-119. Eds. L. Ji, A. Papadopoulos and S.T. Yau. Higher Education Press and International Press, Beijing-Boston. 
\bibitem [Sap07]{Sap07}
M. Sapir, \emph{Some group theory problems,} arXiv:math.GR/0704.2899v1, 2007.
\bibitem [Sco83]{Sco83}
P. Scott, The geometries of 3-manifolds, \emph{Bull. London Math. Soc.,} \textbf{15} (1983), no.5, 401-487.
\bibitem [Sel01]{Sel01}
Z. Sela, Diophantine geometry over groups. I, \emph{Makanin-Razborov diagrams. Publ. Math. Inst. Hautes ´ Etudes Sci.,} \textbf{93} (2001), 31-105.
\bibitem [Sie72]{Sie72}
L.C. Siebenmann, Approximating cellular maps by homeomorphisms, \emph{Topology,} \textbf{11} (1972), 271-294.
\bibitem [Sta62]{Sta62}
J. Stallings, The piecewise-linear structure of Euclidean space, \emph{Proc. Cambridge Philos. Soc.,} \textbf{58} (1962), 481-488.
\bibitem [Sul71]{Sul71}
D. Sullivan, Geometric periodicity and the invariants of manifolds, \emph{Springer Lecture Notes}, \textbf{197} (1971), 44-75.
 \bibitem [Thu79]{Thu79}
W.P. Thurston, \emph{The geometry and topology of 3-Manifolds,} Geometry center, Minneapolis MN, 1979.
\bibitem [Thu82]{Thu82}
W.P. Thurston, Three dimensional manifolds, Kleinian groups and hyperbolic geometry, \emph{Bull. Amer. Math. Soc.,} \textbf{6} (1982), 357-381.
\bibitem [Thu97]{Thu97}
W.P. Thurston, \emph{Three-dimensional Geometry and Topology}, Vol.1 (Silvio Levy, ed.) Princeton Math.Ser.,vol.35, Princeton Univ.Press, 1997.
\bibitem [Tur88]{Tur88}
V.G. Turaev, Homeomorphisms of geometric three-dimensional manifolds, \emph{(Russian). Mat. Zametki.,} \textbf{43} (1988), no.4, 533-542, 575; translated in Math. Notes 43 (1988), no. 3-4, 307-312.
\bibitem [Wal69]{Wal69}
C.T.C. Wall, On homotopy tori and the annulus theorem, \emph{Bull. London Math. Soc.,} \textbf{1} (1969), 95-97.
\bibitem [Wal71]{Wal71}
C.T.C. Wall, \emph{Surgery on Compact Manifolds,} Academic Press, London, 1971.
\bibitem [Wal76]{Wal76}
F. Waldhausen, Algebraic K-theory of topological spaces. I, Algebraic and geometric topology (Proc. Sympos. Pure Math., Stanford Univ., Stanford,
Calif., 1976), Part 1, 35-60. \emph{Amer. Math. Soc.,} Providence, R.I., 1978.
\bibitem [Weg10]{Weg10}
C. Wegner, The K-theoretic Farrell-Jones conjecture for $CAT(0)$-groups, arXiv:1012.3349v1 [math.GT], 2010.
\bibitem [Yau71]{Yau71}
S.-T. Yau,  On the fundamental group of compact manifolds of non-positive curvature, \emph{Ann. of Math.,} (2) \textbf{93} (1971), 579-585. 

\end{thebibliography}
\end{document}